\documentclass[11pt]{article}
\usepackage[english]{babel}
\usepackage{amsfonts}
\usepackage[centertags]{amsmath}
\usepackage{amsthm}
\usepackage{latexsym,amssymb}
\usepackage{graphicx}
\usepackage{eqnarray}
\usepackage[all]{xy}
\input diagxy

\usepackage{picins}
\newtheorem{teo}{Theorem}[section]
\newtheorem{prop}[teo]{Proposition}
\newtheorem{defin}[teo]{Definition}
\newtheorem{lemma}[teo]{Lemma}
\newtheorem{corol}[teo]{Corollary}
\newtheorem{state}[teo]{Statement}

\newcommand{\ds}{\displaystyle}
\newtheoremstyle{obs}
  {3pt}
  {3pt}
  {}
  {}
  {\bfseries}
  {.}
  {.5em}
  {}
\theoremstyle{obs}

\newtheorem{remark}[teo]{Remark}

\newcommand{\comments}{\noindent \textbf{Comments: }}
\newcommand{\comment}{\noindent \textbf{Comment: }}

\newcommand{\note}{\noindent \textbf{Note: }}

\def\tabaddress#1{{\small\it\begin{tabular}[t]{c}#1
\\[1.2ex]\end{tabular}}}

\title{GEOMETRIC APPROACH TO PONTRYAGIN'S MAXIMUM PRINCIPLE}
\author{\sc M. Barbero-Li\~n\'an, M. C. Mu\~noz-Lecanda \thanks{{\bf e}\textendash{}{\it mail}:
  mbarbero@ma4.upc.edu, matmcml@ma4.upc.edu}
\\
 \tabaddress{Departamento de Matem\'atica Aplicada IV\\
  Edificio C\textendash{}3, Campus Norte UPC.
  C/ Jordi Girona 1. E\textendash{}08034 Barcelona. Spain \\
  Phone numbers: $+34934015992$, $+34934015959$. Fax number:
  $+34934015981$.
  }}
\date{\today}

\begin{document}
\maketitle \thispagestyle{empty} \setcounter{page}{0}

\begin{abstract}
Since the second half of the 20th century, Pontryagin's
Maximum Principle has been widely discussed and used as a
method to solve optimal control problems in medicine,
robotics, finance, engineering, astronomy. Here, we focus
on the proof and on the understanding of this Principle,
using as much geometric ideas and geometric tools as
possible. This approach provides a better and clearer
understanding of the Principle and, in particular, of the
role of the abnormal extremals. These extremals are
interesting because they do not depend on the cost
function, but only on the control system. Moreover, they
were discarded as solutions until the nineties, when
examples of strict abnormal optimal curves were found. In
order to give a detailed exposition of the proof, the paper
is mostly self\textendash{}contained, which forces us to
consider different areas in mathematics such as algebra,
analysis, geometry.
\end{abstract}

\bigskip \bigskip

\vbox{\raggedleft AMS s.\,c.\,(2000): 34A12, 49J15, 49J30,
49K05, 49K15, 
93C15. }\null

\clearpage \thispagestyle{empty} \setcounter{page}{0}
 \tableofcontents

\thispagestyle{empty}
\setcounter{page}{1}
\section{Introduction}

The importance of Pontryagin's Maximum Principle as a method to find
solutions to optimal control problems is the main justification for
this work. The use and the comprehension of this Principle does not
always gather together. The understanding of this Maximum Principle
never finishes as shows the continuous wide number of references in
this topic
\cite{2004Agrachev,1966Athans,2003Bloch,2005BonnardCaillau,2003BonnardChyba,
2005BulloLewis,1997Jurdjevic,2003Bavophd,2003Bavo,
67LeeMarkus,P62,S98Free,2000Sussmann,2005Sussmann} and references
therein. We try to contribute to this process through a differential
geometric approach.

When we can interfere in the evolution  of a dynamical
system, we deal with a control system; that is, a
differential equation depending on parameters, which are
called controls. The way we interfere in the control system
consists of changing the controls arbitrarily. In optimal
control problems, the controls are chosen such that the
integral of a given cost function is minimized. That
functional to be minimized can correspond with the time,
the energy, the length of a path or other magnitude related
to the system.

In general, to find a solution to an optimal control
problem is not straightforward. A valuable tactic to deal
with these problems is to restrict the candidates to be
solution through necessary conditions for optimality, such
as those given by Pontryagin's Maximum Principle. This
technique is used in a wide range of disciplines, as for
instance engineering
\cite{2005ChybaEt,2003ChybaEt,ApplOptControl,LewisOptControl},
aerospace \cite{ExampleCursoDavid}, robotics
\cite{1995NumericalTimeOptArmRobot,1999TimeOptRobot,RoboticsMurray},
medicine \cite{Biology}, economics
\cite{1970ApplOptControl,Salesman}, traffic flow
\cite{NetworkTrafic}. Nevertheless, it is worth remarking
that the Maximum Principle does not give sufficient
conditions to compute an optimal trajectory; it only
provides necessary conditions. Thus only candidates to be
optimal trajectories are found, called extremals. To
determine if they are optimal or not, other results related
to the existence of solutions for these problems are
needed. See
\cite{2004Agrachev,1966Athans,1962Filippov,67LeeMarkus} for
more details.

In 1958 the International Congress of Mathematicians was held in
Edinburgh, Scotland, where for the first time L. S. Pontryagin
talked publicly about the Maximum Principle. This Principle was
developed by a research group on automatic control created by
Pontryagin in the fifties. He was engaged in applied mathematics by
his friend A. Andronov and because scientists in the Steklov
Mathematical Institute were asked to carry out applied research,
especially in the field of aircraft dynamics.

At the same time, in the regular seminars on automatic
control in the Institute of Automatics and Telemechanics,
A. Feldbaum introduced Pontryagin and his collegues to the
time\textendash{}optimization problem. This allowed them to
study how to find the best way of piloting an aircraft in
order to defeat a zenith fire point in the shortest time as
a time\textendash{}optimization problem.

Since the equations for modelling the aircraft's problem are
nonlinear and the control of the rear end of the aircraft runs over
a bounded subset, it was necessary to reformulate the calculus of
variations known at that time. Taking into account ideas suggested
by E. J. McShane in \cite{39McShane}, Pontryagin and his
collaborators managed to state and prove the Maximum Principle,
which was published in Russian in 1961 and translated into English
\cite{P62} the following year. See \cite{Boltyanski} for more
historical remarks.

Initially the approach to optimal control problems was from
the point of view of the differential equations
\cite{1966Athans, 67LeeMarkus,P62,84Zeidler}, but later the
approach was from the differential geometry
\cite{2004Agrachev,2007BressanBook,1997Jurdjevic,S98Free}.
Furthermore, the Maximum Principle is being modified to
study stochastic control systems
\cite{84StochasticPMP,1986StochPMPBook} and discrete
control systems
\cite{2008ChybaIntegrator,BlochDiscretePMP,1967DiscretePMP}.
Lately, the Skinner\textendash{}Rusk formulation
\cite{SR-83} has been applied to study optimal control
problem for non\textendash{}autonomous control systems,
obtaining again the necessary conditions of Pontryagin's
Maximum Principle, as long as the differentiability with
respect to controls is assumed \cite{BEMMR-2007}. This
formulation is suitable to deal with implicit optimal
control problems that come up in engineering problems
described by the descriptor systems \cite{muller,M-1999}.
This Principle also admits a presymplectic formalism that
gives weaker necessary conditions for optimality
\cite{Lisboa,MarinaIbortPanorama,2003BcnControlSim}.

Therefore, it is concluded that Pontryagin's Maximum
Principle has had and still has a great impact in optimal
control theory. The references mentioned show that the
research is still active as for the understanding and also
for the applications of the Maximum Principle.

A symplectic Hamiltonian formalism to optimal control
problems is provided by the necessary conditions stated in
Pontryagin's Maximum Principle. The solutions to the
problem are in the phase space manifold of the system, but
the Maximum Principle relates solutions to a lift to the
cotangent bundle of that manifold. Thus, in order to find
candidates to be optimal solutions, not only the controls
but also the momenta must be chosen appropriately so that
the necessary conditions in the Maximum Principle are
fulfilled. These conditions are, in fact,
first\textendash{}order necessary conditions and they are
not always enough to determine the evolution of all the
degrees of freedom in the problem. That is why sometimes it
is necessary to use the high order Maximum Principle
\cite{98Bianchini,KawskiSurvey,Knobloch,1977Krener}. But,
even when we succeed in finding the controls and the
momenta in such a way that Hamilton's equations can be
integrated to obtain a trajectory on the manifold, the
controls and the momenta are not necessarily unique. In
other words, different controls and different momenta can
give the same trajectory on the manifold, although the
necessary conditions in the Maximum Principle will be
satisfied in different ways. The momenta and the controls
determine different kinds of trajectories, which can be
abnormal, normal, strict abnormal, strict normal and
singular. We point out that these different kinds of
extremals do not provide a partition of the set of
trajectories in the manifold, because it may happen that a
trajectory admits more than one lift to the momenta space
so that the trajectory is in two different categories.

For years, abnormal extremals were discarded because it was
thought that they could not be optimal
\cite{90Hamenstadt,86Strichartz}. The idea was that
abnormal extremals were isolated curves and thus it was
impossible to consider any variation of these curves.
However, in \cite{94Montgomery} it is proved that there
exist abnormal minimizers by giving an example in
subRiemannian geometry. Furthermore, in \cite{96LS} the
strict abnormal minimizers are characterized in a general
way, studying the length\textendash{}minimizing problem in
subRiemannian geometry when there are only two controls. To
be more precise, a large enough set with abnormal extremals
is given and it contains strict abnormal curves that are
locally optimal for the considered
control\textendash{}linear system. Here began a new
interest in the abnormal extremals
\cite{AgrachevAbnLagrange,AgrachevMorseInd,AZ,
2006BonnardTrelat,2003Bavophd,2003Bavo}. What makes these
extremals more special is that the abnormality does not
depend on the cost function. Hence, the abnormal extremals
can be determined exclusively using the geometry of the
control system. Thus abnormality and controllability must
be closely related. In fact, in order to have abnormal
minimizers, the system cannot be controllable. In control
theory, controllability is still one of the properties
under active research
\cite{AgrachevQuestionControl,MTNSCesar} and the same
happens with abnormality in optimal control theory.
Moreover, the controllability is related with the reachable
set. Thus, as first pointed out in
\cite{2005BulloLewis,2003Bavo}, the geometry of the
reachable set also helps to characterize the abnormal
extremals.

On the other hand, the cost function is essential to prove that
abnormal extremals are abnormal minimizers, as pointed out in
\S\ref{symplectic}. That is why the existence or
non\textendash{}existence of abnormal minimizers is only known for
specific control problems, mainly control\textendash{}linear and
control\textendash{}affine systems with
control\textendash{}quadratic cost functions or for
time\textendash{}optimal control problems
\cite{95AgrachevSarychevStrongMin,AZ,2003BonnardChyba,
ChitourJeanTrelatDiffGeom2006,ChitourJeanTrelat2006,Zelenko2Distr}.

How the necessary conditions of Pontryagin's Maximum
Principle are satisfied determines the kind of extremals
obtained, in particular, the abnormal ones. That is why the
thorough proof of the Maximum Principle given here gives
insights into the geometric understanding of the
abnormality. Any chance we have along the report to make a
comment about abnormality will be made because that might
help to characterize strict abnormality in the future.

In this paper, we go through the entire proof of
Pontryagin's Maximum Principle translating it into a
geometric framework, but preserving the outline of the
original proof. All details have been carefully proved,
making us to go into the details of concepts such as
time\textendash{}dependent variational equations and their
properties, separation conditions given by hyperplanes and
convexity. All this is included as appendices in order not
to disturb the continuous evolution of the concepts here
given. Nevertheless, we assume some knowledge in
diffe\-ren\-tial geometry, such as the core chapters of
\cite{2003Lee}, differential equations
\cite{2004Canizo,55Coddington,1998Hairer}, and convexity
\cite{2001Bertsekas,98Rockafellar}.

The control systems in this report are given by a vector
field along a projection, that is defined in \S
\ref{general} together with its properties. In the heart of
the report there are two big parts corresponding with two
different statements of Pontryagin's Maximum Principle. In
\S \ref{PMPfijo} and \S \ref{proofPMPfixed}, it is studied
the optimal control problem with both the time interval and
the endpoints given. If the final time is not given and the
endpoints are not fixed but they must be in specific
submanifolds, then the problem is studied in \S
\ref{proofFPMP} and \S \ref{SproofFPMP}. These four
sections have been written in an analogous way. First of
all, two different but equivalent statements of the optimal
control problems are given. The so\textendash{}called
extended system is the useful one in \S \ref{proofPMPfixed}
and \S \ref{SproofFPMP} because the functional to be
minimized is included as a new coordinate of the system.
The last subsection in \S \ref{PMPfijo} and \S
\ref{proofFPMP} explains the associated Hamiltonian problem
that leads to the statements of Maximum Principle. In this
way, the proof is just in \S \ref{proofPMPfixed} and \S
\ref{SproofFPMP}.

One part of the proof of Pontryagin's Maximum Principle
consists of perturbing the given optimal curve, therefore
we introduce in \S \ref{perturbedconesfixed} and \S
\ref{perturbedconesnonfixed} how this curve can be
perturbed depending on the known data. Above all, it is
important the complete proof of Proposition \ref{lemma2},
although known, to our knowledge, there is not a
self\textendash{}contained proof of it in the literature.

The appendices contain essential results for the core of
the report and also some explanation to make clear some
well\textendash{}known ideas related to
time\textendash{}dependent vector fields in Appendix
\ref{variationalsection}, the reachable set and the tangent
perturbation cone in Appendix \ref{approxreachable}.

The study of the time\textendash{}dependent variational
equations treated in Appendix \ref{variationalsection}
gives a clear picture of the flows of the complete lift and
of the cotangent lift of a time\textendash{}dependent
vector field via Propositions \ref{XT},
\ref{propgeomeaning}, \ref{cotangent}, \ref{constant}.
These results although known, to our knowledge, have not
appeared in the literature.

Appendix \ref{approxreachable} devotes to the careful study of the
connection between the reachable set and the tangent perturbation
cone, because the proof of Pontryagin's Maximum Principle suggests
that all the perturbation vectors generate a linear approximation of
the reachable set in some sense. That sense will become clear in
Proposition \ref{sumflows}, which proves a result assumed as true in
the literature.

To summarize the main contributions of the paper are:
\begin{itemize}
\item The proof of Proposition \ref{lemma2}, that is useful to prove Pontryagin's Maximum
Principle. All the proofs of this Proposition in the
literature, to our knowledge, are not written carefully
enough. This Proposition is adapted for Pontryagin's
Maximum Principle without fixing the final time in
Proposition \ref{lema2time}.
\item The complete proof of
Pontryagin's Maximum Principle in a symplectic framework as
in \cite{S98Free}, but here we include all the necessary
results and the analytical reasoning, which has been
sketched in great detail.
\item The highlight of the  properties concerning the abnormal extremals that can be deduced from the classical
result in \cite{P62}.
\item An analytic result necessary in Pontryagin's Maximum
Principle, which is proved in Proposition \ref{lowerew}.
This result is used in \cite{P62}, but without proving it.
\item The intrinsic study of the flows of the complete lift and
the cotangent lift of a time\textendash{}de\-pen\-dent
vector field in Appendix \ref{variationalsection},
including the proofs of Propositions \ref{XT},
\ref{propgeomeaning}, \ref{cotangent} and \ref{constant}.
\item The geometric understanding of the interpretation of the
tangent perturbation cone as linear approximation of the
reachable set in Appendix \ref{approxreachable}, including
the proof of Proposition \ref{sumflows}.
\end{itemize}

As for the future and actual research line, we point out that all
the effort to elaborate this work is being used to enlighten the
research, from a geometric point of view, on abnormal and strict
abnormal extremals in optimal control problems in general
\cite{2007MiguelMaria}, and for mechanical systems \cite{Lisboa}.
The study of strict abnormal minimizers impose us to consider
different cost functions, because only the property of being an
abnormal extremal depends exclusively on the geometry of the control
system. That makes the problem of searching for strict abnormal
minimizers much harder and the possible forthcoming results will be
valid only for determined optimal control problems.

To conclude this introduction, we remark that Pontryagin's Maximum
Principle provides first\textendash{}order necessary conditions for
optimality. These conditions are not always enough to determine the
controls for abnormal and singular extremals, then high order
Maximum Principle is necessary \cite{1977Krener}. The Maximum
Principle works with linear approximation of the trajectories,
whereas in the high order Maximum Principle high order perturbations
must be considered
\cite{98Bianchini,BS93,KawskiSurvey,Knobloch,1977Krener}. The way to
construct the proof is the same as in Pontryagin's Maximum
Principle, but now the tangent perturbation cones are bigger since
not only linear approximation of the trajectories are considered. In
the same way we have provided a geometric meaning to most of the
elements in Pontryagin's Maximum Principle, we expect to give a
geometric version of high order Maximum Principle suggested by
\cite{1977Krener}, focusing on abnormality.

 The origin of this report was a series of seminars and
talks with Professor Andrew D. Lewis during his stay in our
Department on sabbatical during the first term of 2005. We
tried to understand the details of the proof as a way to
work on some aspects of controllability and accesibility of
control systems with a cost function,
\cite{2005BulloLewis,1997Jurdjevic}, and where abnormal
solutions are in the accesibility sets.

In the sequel, unless otherwise stated, all the manifolds
are real, second countable and ${\cal C}^{\infty}$ and the
maps are assumed to be ${\cal C}^{\infty}$. Sum over
repeated indices is understood.

\section{General setting}\label{general}
From the differential geometric viewpoint a control system
is understood as a vector field depending on parameters.
Properties about how the integral curves of differential
equations depending on parameters evolve are explained in
\cite{2004Canizo,55Coddington,1998Hairer,93Michor} and used
in \S \ref{perturbedconesfixed} and \S
\ref{perturbedconesnonfixed}.

Let $M$ be a differentiable manifold of dimension $m$ and
$U$ be a set in $\mathbb{R}^k$. Consider the trivial
Euclidean bundle $\pi\colon M \times U \rightarrow M$.

\begin{defin}\label{DefVfprojection} A \textbf{vector field $X$ on $M$ along the projection $\pi$} is a mapping
$X\colon M \times U \rightarrow TM$ such that $X$ is
continuous on $M\times U$, continuously differentiable on
$M$ for every $u\in U$ and $\tau_M \circ X =\pi$, where
$\tau_M \colon TM \rightarrow M$ is the canonical tangent
projection.
\end{defin}
The set of vector fields along the projection $\pi$ is
denoted by $\mathfrak{X}(\pi)$. If $(V,x^i)$ is a local
chart at $x$ in $M$, then locally a vector field $X$ along
the projection is given by $f^i
\partial / \partial x^i$, where $f^i$ are functions defined on $V\times U$.

Let $I=[a,b] \subset \mathbb{R}$ be a closed interval, $(\gamma,u)
\colon I \rightarrow M \times U$ is an integral curve of $X$ if
$\dot{\gamma}(t)=X(\gamma(t),u(t))$. All these elements come
together in Diagram (\ref{Diagram 1}).
\begin{equation}\label{Diagram 1}\xymatrix{ &TM\ar[d]^{\txt{\small{$\tau_M$}}}\\
M\times U \ar[ur] ^{\txt{\small{$X$}}}\ar[r]^{\txt{\small{$\pi$}}}   &M\\
I\ar[u]^{\txt{\small{$(\gamma,u)$}}}\ar[ur]^{\txt{\small{$
\gamma$}}}&}\end{equation} In other words, $X$ is a vector
field depending on parameters in $U$. In this work, the
parameters are called \textit{controls} and are assumed to
be measurable mappings $u\colon I \rightarrow U$ such that
${\rm Im} \ u$ is bounded. Given the parameter $u$, we have
a time\textendash{}dependent vector field on $M$,
\begin{equation}\label{NotVFproj} \begin{array}{rcl} X^{\{u\}}
\colon I \times M &\longrightarrow &TM \\
(t,x) & \longmapsto & X^{\{u\}}(t,x)=X(x,u(t)).
\end{array}\end{equation}
For an integral curve $(\gamma,u)$ of $X$, it is said that $\gamma$
is an integral curve of $X^{\{u\}}$, as shown in the following
commutative diagram:
\begin{equation}\label{tdiagram}\bfig\xymatrix{I\times M \ar[r]^{\txt{\small{$X^{\{u\}}$}}}& TM
\\I \ar[u]^{\txt{\small{$( \gamma,\, {\rm Id})$}}}
\ar[r]_{\txt{\small{$(\gamma,u)$}}}
\ar[ur]^{\txt{\small{$\dot{\gamma}$}}} & M \times U
\ar[u]^{\txt{\small{$X$}}}} \efig \end{equation} That is,
$X^{\{u\}} \circ (\gamma, {\rm Id})=\dot{\gamma}=X\circ
(\gamma,u)$.

 A differentiable
time\textendash{}dependent vector field $X$ has associated
the \textit{time dependent flow} or \textit{evolution
operator of $X$} defined as
$$\begin{array}{rcl}
\Phi^X \colon  &I\times I\times M &\longrightarrow M\\ &
(t,s,x)&\longmapsto \Phi^X(t,s,x)=\Phi^X_{(s,x)}(t)
\end{array}$$
where $\Phi^X_{(s,x)}$ is the integral curve of $X$ with initial
condition $x$ at time $s$. See Appendix \ref{ApTimevf} for more
details. Moreover, the evolution operator defines a diffeomorphism
on $M$ that is used in the following section $\Phi^X_{(t,s)} \colon
M \rightarrow M$, $x \mapsto \Phi^X_{(t,s)}(x)=\Phi^X_{(s,x)}(t)$.

As the controls $u\colon I \rightarrow U$ are measurable and
bounded, the vector fields $X^{\{u\}}$ are measurable on $t$, and
for a fixed $t$, they are differentiable on $M$. Hence, the notion
of Carath\'eodory vector fields must be considered
\cite{2004Canizo,55Coddington} from now on. Then, we only consider
absolutely continuous curves $\gamma \colon I \rightarrow M$ to be
\textit{ge\-ne\-ra\-li\-zed integral curves} of the vector field
$X^{\{u\}}$; that is, they only satisfy $\dot{\gamma}=X\circ
(\gamma,u)$ at points where $\gamma$ is derivable, which happens
almost everywhere. The existence and uniqueness of these integral
curves are guaranteed once the parameter is fixed because of the
theorems of existence and uniqueness of differential equations
depending on parameters. For more details about absolute continuity,
see Appendix A and \cite{2004Canizo,55Coddington,65Varberg}.

\section{Pontryagin's Maximum Principle for fixed time and fixed endpoints}\label{PMPfijo}
We particularize the general setting described in \S
\ref{general} for optimal control theory. To make clear we
are in a specific case the manifold is denoted by $Q$,
instead of $M$.

\subsection{Statement of optimal control problem and
notation}\label{PCOfijo}

Let $Q$ be a differentiable manifold of dimension $m$ and $U\subset
\mathbb{R}^k$ a subset. Let us consider the trivial Euclidean bundle
$\pi\colon Q \times U \rightarrow Q$.


Let $X$ be a vector field along the projection $\pi\colon Q
\times U \rightarrow Q$ as in Definition
\ref{DefVfprojection}. If $(V,x^i)$ is a local chart at a
point in $Q$, the local expression of the vector field is
$X=f^i {\partial}/{\partial x^i}$ where $f^i$ are functions
defined on $V\times U$.

Let $I\subset \mathbb{R}$ be an interval and $(\gamma, u) \colon I
\rightarrow Q \times U$ be a curve. Given $F \colon Q\times U
\rightarrow \mathbb{R}$, let us consider the functional
$${\cal S}[\gamma,u]=\int_I F(\gamma,u)\, dt$$
defined on curves $(\gamma, u)$ with a compact interval as
domain. The function $F \colon Q \times U \rightarrow
\mathbb{R}$ is continuous on $Q \times U$ and continuously
diffe\-rentia\-ble with respect to $Q$ on $Q \times U$.

\begin{state} \textbf{(Optimal Control Problem, $OCP$)}
Given the elements $Q$, $U$, $X$, $F$, $I=[a,b]$ and the endpoint
conditions $x_a$, $x_b\in Q$, consider the following problem.

Find $(\gamma^*,u^*)$ such that
\begin{itemize}
\item[(1)] endpoint conditions: $\gamma^*(a)=x_a$, $\gamma^*(b)=x_b$,
\item[(2)] $\gamma^*$ is an integral curve of $X^{\{u^*\}}$: $\dot{\gamma^*}(t)=X(\gamma^*(t), u^*(t))$, $t\in I$, and
\item[(3)] minimal condition: ${\cal S}[\gamma^*,u^*]$ is minimum over all curves $(\gamma,u)$
satisfying  $(1)$ and $(2)$.
\end{itemize}
\end{state}

The tuple $(Q,U,X,F,I,x_a,x_b)$ denotes the \textit{optimal control
problem}. The function $F$ is called the \textit{cost function} of
the problem. The mappings $u \colon I \rightarrow U$ are called
\textit{controls}.

\comments
\begin{enumerate}
\item The curves considered in the previous statement satisfy the same
properties as the generalized integral curves of vector
fields along a projection described in \S \ref{general}.
That is, $\gamma$ is absolutely continuous and the controls
$u$ are measurable and bounded.
\item Locally, condition $(2)$ is equivalent to the fact that the curve $(\gamma^*,u^*)$
satisfies the differential equation $\dot{x}^i=f^i$.
\end{enumerate}

\subsection{The extended problem}\label{PMPEfijo}

Taking into account the elements defining the optimal control
problem and their pro\-per\-ties, we state an equivalent problem.

Given the $OCP$ $(Q,U,X,F,I,x_a,x_b)$, let us consider
$\widehat{Q}=\mathbb{R} \times Q$ and the trivial Euclidean
bundle $\widehat{\pi}\colon \widehat{Q} \times U
\rightarrow \widehat{Q}$.

Let $\widehat{X}$ be the following vector field along the
projection $\widehat{\pi}\colon \widehat{Q} \times U
\rightarrow \widehat{Q}$: $$\widehat{X}(x^0,x,u)=F(x,u)
{\partial}/{\partial x^0}|_{(x^0,x,u)} + X(x,u),$$ where
$x^0$ is the natural coordinate on $\mathbb{R}$. According
to Equation (\ref{NotVFproj}), this vector field can be
rewritten as $\widehat{X}^{\{u\}}$.

Given a curve $(\widehat{\gamma},u)=((x^0 \circ
\widehat{\gamma}, \gamma), u) \colon I \rightarrow
\widehat{Q} \times U$ such that $\widehat{\gamma}$ is
absolutely continuous and $u$ is measurable and bounded,
the previous elements come together in the following
diagram:
$$\bfig\xymatrix{ & T\widehat{Q}\ar[d]^{\txt{\small{$\tau_{\widehat{Q}}$}}}\\
\widehat{Q}\times U
\ar[ur]^{\txt{\small{$\widehat{X}$}}}\ar[r]^{\txt{\small{$\widehat{\pi}$}}}
&\widehat{Q}\ar[d]^{\txt{\small{$\pi_2$}}}\\
I\ar[u]^{\txt{\small{$(\widehat{\gamma},u)$}}}\ar[ur]^{\txt{\small{$\widehat{\gamma}$}}}
\ar[r]^{\txt{\small{$\gamma$}}}&Q}\efig$$ where $\pi_2$ is
the projection of $\widehat{Q}$ onto $Q$.

\begin{state}\label{stateEOCP}\textbf{(Extended Optimal Control Problem, $\mathbf{\widehat{OCP}}$)}
Given the above\textendash{}mentioned $OCP$
$(Q,U,X,F,I,x_a,x_b)$, $\widehat{Q}$ and $\widehat{X}$,
consider the fo\-llo\-wing problem.

Find $(\widehat{\gamma}^*,u^*)$ such that
\begin{itemize}
\item[(1)] endpoint conditions: $\widehat{\gamma}^*(a)=(0,x_a)$, $\gamma^*(b)=x_b$,
\item[(2)] $\widehat{\gamma}^*$ is an integral curve of $\widehat{X}^{\{u^*\}}$: $\dot{\widehat{\gamma}}^*(t)=\widehat{X}(\widehat{\gamma}^*(t),u^*(t))$, $t\in I$, and
\item[(3)] minimal condition: $\gamma^{*^0}(b)$ is minimum over all curves $(\widehat{\gamma},u)$
satisfying  $(1)$ and $(2)$.
\end{itemize}
\end{state}

The tuple $(\widehat{Q},U,\widehat{X},I,x_a,x_b)$ denotes the
\textit{extended optimal control problem}.

\begin{enumerate}
\item The functional $\gamma^{*^0}(b)$ to be minimized in the $\widehat{OCP}$ is equal
to the functional defined in the $OCP$. That is to say, we have
$$\widehat{\cal S}[\widehat{\gamma},u]=\gamma^0(b)=\int_a^b F(\gamma,u) dt={\cal S}[\gamma,u]$$
for curves $(\widehat{\gamma},u)$.
\item Locally, the condition (2) is equivalent to the fact that the curve $(\widehat{\gamma}^*,u^*)$
satisfies the differential equations $\dot{x}^0=F$, $\dot{x}^i=f^i$.
\end{enumerate}

The elements in the problem $(\widehat{M},U,\widehat{X},I,x_a,x_b)$
satisfy properties analogous to the ones fulfilled by the elements
in the problem $(M,U,X,F,I,x_a,x_b)$, but for different spaces; see
\S \ref{general}, \S \ref{PCOfijo} for more details about the
properties.

\subsection{Perturbation and associated
cones}\label{perturbedconesfixed} The following
constructions can be defined for any vector field depending
on parameters\textemdash{}see \S
\ref{general}\textemdash{}in particular, for those vector
fields defining a control system. In order not to make the
notation harder, we will construct everything on $M$, but
the same can be done on $\widehat{M}$ or on any other
convenient manifold, as for instance the tangent bundle
$TQ$ for the mechanical case.

\subsubsection{Elementary perturbation vectors: class I}
\label{elementarypertvectors} Now we study how integral
curves of the time\textendash{}dependent vector field
$X^{\{u\}}\colon M\times I \rightarrow TM$, introduced in
\S \ref{general}, change when the control $u$ is perturbed
in a small interval.

In the sequel, a measurable and bounded control $\; u \colon I=[a,b]
\rightarrow U$ and an absolutely continuous integral curve $ \;
\gamma\colon I \rightarrow M$ of $X^{\{u\}}$ are given. Let
$\pi_1=\{t_1,l_1,u_1\}$, where $t_1$ is a Lebesgue time in $(a,b)$
always for the $X\circ (\gamma,u)$\textemdash{}i.e. it satisfies
Equation (\ref{eqLebesgue})\textemdash{}$l_1\in \mathbb{R}^+$,
$u_1\in U$. From now on, to simplify, $t_1$ is called just a
Lebesgue time. For every $s\in \mathbb{R}^+$ small enough such that
$a<t_1-l_1 s$, consider $u[\pi_1^s]\colon I \rightarrow U$ defined
by
$$u[\pi_1^s](t) = \left\{\begin{array}{ll}  u_1, & t\in [t_1-l_1 s, t_1], \\
u(t), & {\rm elsewhere}. \end{array} \right. $$

\begin{defin} The function $u[\pi_1^s]$ is
called an \textbf{elementary perturbation of $u$ specified
by the data $\pi_1=\{t_1,l_1,u_1\}$}. It is also called a
\textbf{needle\textendash{}like variation}.
\end{defin}

Associated to $u[\pi_1^s]$, consider the mapping $\gamma
[\pi_1^s] \colon I \rightarrow M$, the generalized integral
curve of $X^{\{u[\pi_1^s]\}}$ with initial condition
$(a,\gamma(a))$.

Given $\epsilon>0$, define the map
$$\begin{array}{rcl} \varphi_{\pi_1} \colon
I\times [0,\epsilon] &\longrightarrow & M\\
( \; t \;, \; s  \;)\;\;&\longmapsto &
\varphi_{\pi_1}(t,s)= \gamma[\pi_1^s](t)
\end{array}$$

For every $t\in I$, $\varphi_{\pi_1}^{t}\colon [0,\epsilon]
\rightarrow M$ is given by $\varphi_{\pi_1}^{t}(s)=
\varphi_{\pi_1}(t,s)$.

As the controls are assumed to be measurable and bounded,
it makes sense to define the distance between two controls
$u, \overline{u}\colon I \rightarrow U$  as follows
$$d(u,\overline{u})=\int_I\left\|u(t)-\overline{u}(t)\right\| \ dt$$
where $\| \cdot \|$ is the usual norm in $\mathbb{R}^k$.
Here, a bounded control $u\colon I \rightarrow U$ means
that there exists a compact set in $U$ that contains ${\rm
Im}\ u$. The control $u[\pi_1^s]$ depends continuously on
the parameters $s$ and $\pi_1=\{t_1,l_1,u_1\}$; that is,
given $\epsilon>0$ there exists $\delta>0$ such that if
$\left| t_1-t_2\right|<\delta$, $\left|
l_1-l_2\right|<\delta$, $\left\| u_1-u_2\right\|<\delta$,
$\left| s_1-s_2\right|<\delta$, then $d(u[\pi_1^{s_1}],
u[\pi_2^{s_2}])<\epsilon$.

Hence the curve $\varphi_{\pi_1}^t$ depends continuously on $s$ and
$\pi_1=\{t_1,l_1,u_1\}$, then it converges uniformly to $\gamma$ as
$s$ tends to $0$. See \cite{2004Canizo,55Coddington} for more
details of the differential equations depending continuously on
parameters.

Let us prove that the curve $\varphi_{\pi_1}^{t_1}$ has a
tangent vector at $s=0$.
Let $u[\pi_1^s]$ be an elementary perturbation of $u$
specified by $\pi_1=\{t_1,l_1,u_1\}$ and consider the curve
$\varphi_{\pi_1}^{t_1}\colon [0,\epsilon]\rightarrow M$,
$\varphi_{\pi_1}^{t_1}(s)=\gamma[\pi_1^s](t_1)$.
\begin{prop}\label{tangentperturb} If $t_1$ is a Lebesgue time, then
the curve $\varphi_{\pi_1}^{t_1}\colon [0,\epsilon]\rightarrow M$
is differentiable at $s=0$. Its tangent vector is
$\left[X(\gamma(t_1), u_1)-X(\gamma(t_1), u(t_1))\right]\,
l_1$.
\end{prop}
\begin{proof}
It is enough to prove that for every differentiable
function $g\colon M\rightarrow \mathbb{R}$, there exists
$$A=\lim_{s \rightarrow 0} \, \frac{g(\varphi_{\pi_1}^{t_1}(s))
-g(\varphi_{\pi_1}^{t_1}(0))}{s}.$$ As this is a derivation
on the functions defined on a neighbourhood of
$\gamma(t_1)$, it is enough to prove the proposition for
the coordinate functions $x^i$ of a local chart at
$\gamma(t_1)$. Thus take $g=x^i$,
\begin{eqnarray*}
A&=&\lim_{s \rightarrow 0} \frac{(x^i \circ
\varphi_{\pi_1}^{t_1})(s)- (x^i \circ
\varphi_{\pi_1}^{t_1})(0)}{s} =\lim_{s \rightarrow 0}
\frac{( x^i \circ \gamma[\pi_1^s])(t_1)-(x^i
\circ\gamma)(t_1)}{s}\\&=&\lim_{s \rightarrow 0}
\frac{\gamma^i[\pi_1^s](t_1)-\gamma^i(t_1)}{s}.
\end{eqnarray*}
As
$\gamma$ is an absolutely continuous integral curve of
$X^{\{u\}}$ , $\dot{\gamma}(t)=X(\gamma(t),u(t))$ at every
Lebesgue time. Then integrating
$$\gamma^i(t_1)-\gamma^i(a)=\int^{t_1}_af^i(\gamma(t),u(t))dt$$
and similarly for $\gamma[\pi_1^s]$ and $u[\pi_1^s]$.
Observe that $\gamma[\pi_1^s](t)=\gamma(t)$ and
$u[\pi_1^s](t)=u(t)$ for $t\in[a,t_1-l_1 s)$. Then,
\begin{eqnarray*}
A&=&\lim_{s \rightarrow 0}
\frac{\int^{t_1}_{a}f^i(\gamma[\pi_1^s](t),
u[\pi_1^s](t))dt-\int^{t_1}_{a}f^i(\gamma(t), u(t))dt}{s}
\\ &=&\lim_{s \rightarrow 0} \frac{\int^{t_1}_{t_1-l_1
s}f^i(\gamma[\pi_1^s](t), u_1)dt-\int^{t_1}_{t_1-l_1
s}f^i(\gamma(t), u(t))dt}{s}.\end{eqnarray*} As $t_1$ is a
Lebesgue time, we use Equation (\ref{eqLebesgue}):
\[\int_{t -h}^t X(\gamma(s),u(s))ds=hX(\gamma(t),u(t))+
o(h)\] in such a way that \begin{eqnarray*} A&=&\lim_{s
\rightarrow 0} \frac{f^i(\gamma[\pi_1^s](t_1), u_1)l_1s
-f^i(\gamma(t_1), u(t_1))l_1 s +o(s)}{s}\\&=&\lim_{s
\rightarrow 0}[f^i(\gamma[\pi_1^s](t_1),
u_1)-f^i(\gamma(t_1), u(t_1))] \, l_1.
\end{eqnarray*} As $f^i$ is
continuous on $M$, we have \begin{eqnarray*} A&=&\lim_{s
\rightarrow 0}[f^i(\gamma[\pi_1^s](t_1),
u_1)-f^i(\gamma(t_1), u(t_1))] \, l_1= [f^i(\lim_{s
\rightarrow 0} \gamma[\pi_1^s](t_1), u_1)-f^i(\gamma(t_1),
u(t_1))] \, l_1\\&=& [f^i(\gamma(t_1),
u_1)-f^i(\gamma(t_1), u(t_1))]\, l_1
=\left[\left(X(\gamma(t_1),u_1)-X(\gamma(t_1),u(t_1))
\right) \ l_1\right](x^i).\end{eqnarray*}
\end{proof}

\begin{defin}\label{classI} The tangent vector $v[\pi_1]=\left(X(\gamma(t_1),u_1)-X(\gamma(t_1),u(t_1)) \right) \
l_1\in T_{\gamma(t_1)}M$ is the \textbf{elementary perturbation
vector associated to the perturbation data $\pi_1=\{t_1,l_1,u_1\}$}.
It is also called a \textbf{perturbation vector of class I}.
\end{defin}
\comments \begin{itemize}
\item[\textbf{(a)}] The previous proof shows the importance of defining
perturbations only at Lebesgue times, otherwise the
elementary perturbation vectors may not exist.

 \item[\textbf{(b)}] Observe that if we
change $\pi_1=\{t_1,l_1,u_1\}$ for $\pi_2=\{t_1,l_2,u_1\}$, then
$v[\pi_1]=(l_1/l_2) \ v[\pi_2]$. If $v[\pi_1]$ is a perturbation
vector of class I and $\lambda\in \mathbb{R}^+$, then $\lambda \
v[\pi_1]$ is also a perturbation vector of class I with perturbation
data $\{t_1,\lambda \ l_1, u_1\}$.

\item[\textbf{(c)}] We write ${\rm L}(w)g$ for the derivative of the function $g$
in the direction given by the vector $w\in T_xM$. Due to Proposition
\ref{tangentperturb}, for every differentiable function $g\colon M
\rightarrow \mathbb{R}$ we have
\begin{equation*}
\frac{g(\varphi_{\pi_1}^{t_1}(s))-g(\gamma(t_1))-s\ {\rm
L}(v[\pi_1])g}{s} \underset{s\rightarrow 0} \longrightarrow 0.
\end{equation*}
Hence
\begin{equation*}
g\left(\varphi_{\pi_1}^{t_1}(s)\right)=g\left(\gamma(t_1)\right)+s\,
{\rm L}(v[\pi_1])g+ o(s).
\end{equation*}
If $(x^i)$ are local coordinates of a chart at
$\gamma(t_1)$,
\begin{equation*}
x^i\left(\varphi_{\pi_1}^{t_1}(s)\right)=x^i\left(\gamma(t_1)\right)+s\,
v[\pi_1]^i+ o(s).
\end{equation*}
That is,
\begin{equation*}
\left(\varphi_{\pi_1}^{t_1}\right)^i(s)=\gamma^i(t_1)+s\,
v[\pi_1]^i+ o(s).
\end{equation*}
Now, if we identify the open set of the local chart and the
tangent space to $M$ at $\gamma(t_1)$ with the same space
$\mathbb{R}^m$, we write the following linear approximation
\begin{equation}\label{linearapprox}
\varphi_{\pi_1}^{t_1}(s)=\gamma(t_1)+s\, v[\pi_1]+ o(s).
\end{equation}
\end{itemize}

The initial condition for the velocity given by the
elementary perturbation vector evolves along the reference
trajectory $\gamma$ through the integral curves of the
complete lift $\left(X^T\right)^{\{u\}}$ of $X^{\{u\}}$, as
explained in Appendix \ref{completelift}. Note that
$\varphi^t_{\pi_1}(s)=\Phi^{X^{\{u\}}}_{(t,t_1)}\left(\varphi^{t_1}_{\pi_1}(s)\right)$
for $t\geq t_1$ because of the definition of
$\varphi_{\pi_1}$ and $u[\pi_1^s]$.

\begin{prop}\label{tangentperturb2} Let $V[\pi_1]\colon [t_1,b] \rightarrow TM$ be the integral
curve of the complete lift $\left(X^T\right)^{\{u\}}$ of
$X^{\{u\}}$ with initial condition
$(t_1,(\gamma(t_1),v[\pi_1]))$.
 For every Lebesgue time $t\in(t_1,b]$,  $V[\pi_1](t)$ is
the tangent vector to the curve $\varphi_{\pi_1}^t\colon
[0,\epsilon]\rightarrow M$ at $s=0$.
\end{prop}
\begin{proof}
The proof follows from Proposition \ref{XT} and the
definition of the curves considered.
\end{proof}

\subsubsection{Perturbation vectors of class II}

The control can be perturbed twice instead of only once, in
fact it may be modified a finite number of times. If $t_2$
is a Lebesgue time greater than $t_1$ and we perturb the
control with $\pi_1=\{t_1,l_1,u_1\}$ and
$\pi_2=\{t_2,l_2,u_2\}$, then we obtain the perturbation
data $\pi_{12}= \{(t_1,t_2),(l_1,l_2),(u_1,u_2)\}$, which
is given by
$$u[\pi_{12}^s](t)=\left\{ \begin{array}{ll}u_1, &
t\in [t_1-l_1 s, t_1], \\
u_2, & t \in [t_2-l_2 s,t_2], \\
u(t), & {\rm elsewhere}
\end{array}\right.$$
for every $s\in \mathbb{R}^+$ small enough such that $[t_1-l_1 s,
t_1] \cap [t_2-l_2 s,t_2]= \emptyset$. Then
$\gamma[\pi_{12}^s]\colon I \longrightarrow M$ is the generalized
integral curve of $X^{\{u[\pi_{12}^s]\}}$ with initial condition
$(a,\gamma(a))$. Observe that $\gamma[\pi_{12}^0](t)=\gamma(t)$.
Consider the curve $\varphi_{\pi_{12}}^{t_2} \colon
[0,\epsilon]\rightarrow M$ given by $\varphi_{\pi_{12}}^{t_2}
(s)=\gamma[\pi_{12}^s](t_2)$.

\begin{prop}\label{12} Let $t_1$, $t_2$ be Lebesgue times such that $t_1<t_2$.
The vector tangent to $\varphi_{\pi_{12}}^{t_2}
\colon [0,\epsilon]\rightarrow M$ at $s=0$ is $v[\pi_2]+
V[\pi_1](t_2)$, where $V[\pi_1]\colon [t_1,b]\rightarrow
TM$ is the generalized integral curve of
$\left(X^T\right)^{\{u\}}$ with initial condition
$(t_1,(\gamma(t_1),v[\pi_1]))$.
\end{prop}

\begin{proof}
Here we perturb the control first with $\pi_1$ along
$\gamma$ and we obtain $u[\pi_1^s]$. Then we perturb this
last control with the other perturbation data, $\pi_2$,
along $\gamma[\pi_1^s]$. Then the superindeces of the
tangent vectors denote the curve along which the
perturbation is made. As in the proof of Proposition
\ref{tangentperturb}, \begin{eqnarray*} A&=&\lim_{s
\rightarrow 0} \frac{(x^i \circ
\varphi_{\pi_{12}}^{t_2})(s)-(x^i
\circ\varphi_{\pi_{12}}^{t_2})(0)}{s}=\lim_{s \rightarrow
0} \frac{(x^i \circ \gamma[\pi_{12}^s])(t_2)-(x^i
\circ\gamma)(t_2)}{s}\\&=& \lim_{s \rightarrow 0}
\frac{\gamma^i[\pi_{12}^s](t_2)-\gamma^i(t_2)}{s}=\lim_{s
\rightarrow 0}
\left(\frac{\gamma^i[\pi_{12}^s](t_2)-\gamma^i[\pi_1^s](t_2)}{s}
+ \frac{\gamma^i[\pi_1^s](t_2)-\gamma^i(t_2)}{s} \right)\
.\end{eqnarray*} We understand $\gamma[\pi_{12}^s]$ as the
result of perturbing $\gamma[\pi_1^s]$ with $\pi_2$, and
use the linear approximation in Equation
(\ref{linearapprox}) for $\gamma[\pi_{12}^s](t_2)$ and
$\gamma[\pi_1^s](t_2)$ according to Proposition
\ref{tangentperturb}.
\begin{equation*}
\varphi_{\pi_{12}}^{t_2}(s)=\gamma[\pi_{12}^s](t_2)=\gamma[\pi_1^s](t_2)+s\,
v[\pi_2]^{\gamma[\pi_1^s]}+ o(s),
\end{equation*}
\begin{equation*}
\gamma[\pi_1^s](t_2)=\gamma(t_2)+s\,
V[\pi_1]^{\gamma}(t_2)+ o(s).
\end{equation*}
Then
$$A=\lim_{s \rightarrow 0} \left( \frac{s
(v[\pi_2]^{\gamma[\pi_1^s]})^i}{s} + \frac{s
(V[\pi_1]^{\gamma})^i(t_2)}{s} \right)=\lim_{s \rightarrow 0} \left(
(v[\pi_2]^{\gamma[\pi_1^s]})^i+ (V[\pi_1]^{\gamma})^i(t_2)
\right).$$ As $\gamma[\pi_1^s]$ depends on $s$ and $s$ tends to $0$,
$A={\rm L}\left(v[\pi_2]^{\gamma}+
V[\pi_1]^{\gamma}(t_2)\right)x^i$.
\end{proof}
Considering identifications similar to the ones used to write
Equation $(\ref{linearapprox})$, we have
\begin{equation*}
\varphi_{\pi_{12}}^{t_2}(s)=\gamma(t_2)+s v[\pi_2]+ s
V[\pi_1](t_2) +o(s).
\end{equation*}

Now we define how the control changes when it is perturbed
twice at the same time. If $t_1$ is a Lebesgue time,
$\pi'_1=\{t_1,l'_1,u'_1\}$ and
$\pi''_1=\{t_1,l''_1,u''_1\}$ are perturbation data, then
$\pi_{11}=\{(t_1,t_1),(l'_1,l''_1),(u_1',u_1'')\}$ is a
perturbation data given by
$$u[\pi_{11}^s](t)=\left\{ \begin{array}{ll}u_1', &
t\in [t_1-(l'_1+l''_1) s, t_1-l''_1 s], \\
u''_1, & t \in [t_1-l''_1 s,t_1], \\
u(t), & {\rm elsewhere}.
\end{array}\right.$$
for every $s\in \mathbb{R}^+$ small enough such that
$a<t_1-(l'_1+l''_1) s$. Then $\gamma[\pi_{11}^s]\colon I
\longrightarrow M$ is the generalized integral curve of
$X^{\{u[\pi_{11}^s]\}}$ with initial condition $(a,\gamma(a))$.
Observe that $\gamma[\pi_{11}^0](t)=\gamma(t)$. Consider the curve
$\varphi_{\pi_{11}}^{t_1} \colon [0,\epsilon]\rightarrow M$, defined
by $\varphi_{\pi_{11}}^{t_1} (s)=\gamma[\pi_{11}^s](t_1)$.

\begin{prop}\label{11} Let $t_1$ be a Lebesgue time. The vector tangent to $\varphi_{\pi_{11}}^{t_1}
\colon [0,\epsilon]\rightarrow M$ at $s=0$ is $v[\pi'_1]+
v[\pi''_1]$, where $v[\pi'_1]$ and $v[\pi''_1]$ are the
perturbation vectors of class I associated to $\pi'_1$ and
$\pi''_1$, respectively.
\end{prop}
\begin{proof}
As in the proof of Proposition \ref{tangentperturb}
$$A=\lim_{s \rightarrow 0} \frac{(x^i \circ
\varphi_{\pi_{11}}^{t_1})(s)- (x^i \circ
\varphi_{\pi_{11}}^{t_1})(0)}{s} =\lim_{s \rightarrow 0}
\frac{\gamma^i[\pi_{11}^s](t_1)-\gamma^i(t_1)}{s}. $$ As
$\gamma$ is an absolutely continuous integral curve of
$X^{\{u\}}$, $\dot{\gamma}(t)=X(\gamma(t),u(t))$ at every
Lebesgue time. Then, we integrate
\[\gamma^i(t_1)-\gamma^i(a)=\int^{t_1}_af^i(\gamma(t),u(t))dt\]
and similarly for $\gamma[\pi_{11}^s]$ and $u[\pi_{11}^s]$.
Observe that $\gamma[\pi_{11}^s](t)=\gamma(t)$ and
$u[\pi_{11}^s](t)=u(t)$ for $t\in[a,t_1-(l'_1+l''_1) s)$.
Then,
\begin{eqnarray*}
A&=&\lim_{s \rightarrow 0}
\frac{\int^{t_1}_{a}f^i(\gamma[\pi_{11}^s](t),
u[\pi_{11}^s](t))dt-\int^{t_1}_{a}f^i(\gamma(t),
u(t))dt}{s} \\ &=& \lim_{s \rightarrow 0}
\frac{\int^{t_1}_{t_1-(l'_1+l''_1)
s}f^i(\gamma[\pi_{11}^s](t),
u[\pi_{11}^s](t))dt-\int^{t_1}_{t_1-(l'_1+l''_1)
s}f^i(\gamma(t), u(t))dt}{s}\\&=&\lim_{s \rightarrow
0}\left( \frac{\int^{t_1-l''_1 s}_{t_1-(l'_1+l''_1)
s}\left(f^i(\gamma[\pi_1^{'s}](t),
u_1')-f^i(\gamma(t),u(t))\right)dt}{s}\right. \\
&+&\left.\frac{\int^{t_1}_{t_1-l''_1
s}\left[f^i(\gamma[\pi_{11}^s](t),
u_1^{''})-f^i(\gamma(t),u(t))\right]dt}{s}\right).
\end{eqnarray*} As
$t_1$ and $t_1-l''_1 s$ are Lebesgue times and
$a<t_1-(l'_1+l''_1)s$ for a small enough $s$, Equation
(\ref{eqLebesgue}) is used. Now we have
\begin{eqnarray*}
A&=&\lim_{s \rightarrow 0}\left\{
\frac{f^i(\gamma[\pi_1^{'s}](t_1-l''_1 s), u'_1)l'_1s
-f^i(\gamma(t_1-l''_1 s), u(t_1-l''_1 s))l'_1
s}{s}\right.\\&+&\left.\frac{ f^i(\gamma[\pi_{11}^s](t_1),
u''_1)l''_1s -f^i(\gamma(t_1), u(t_1))l''_1 s}{s}\right\}\\
&=&\lim_{s \rightarrow
0}\left([f^i(\gamma[\pi_1^{'s}](t_1-l''_1 s), u'_1)
-f^i(\gamma(t_1-l''_1 s), u(t_1-l''_1 s))] \,
l'_1\right.\\&+&\left. (f^i(\gamma[\pi_{11}^s](t_1),
u''_1)-f^i(\gamma(t_1), u(t_1))) \, l''_1\right).
\end{eqnarray*} As $f^i$
is continuous on $M\times U$, we have
\begin{eqnarray*}A&=&\left(f^i\left(\lim_{s \rightarrow
0}\gamma[\pi_1^{'s}](t_1-l''_1 s), u'_1\right)
-f^i\left(\lim_{s \rightarrow 0}\gamma(t_1-l''_1 s),
\lim_{s \rightarrow 0}u(t_1-l''_1 s)\right)\right) \,
l'_1\\&+& \left(f^i\left(\lim_{s \rightarrow
0}\gamma[\pi_{11}^s](t_1),
u''_1\right)-f^i\left(\gamma(t_1), u(t_1)\right)\right) \,
l''_1= [f^i(\gamma(t_1), u'_1)-f^i(\gamma(t_1), u(t_1))]\,
l'_1\\&+&(f^i(\gamma(t_1), u''_1)-f^i(\gamma(t_1),
u(t_1)))\, l''_1 ={\rm
L}\left(v[\pi_1']+v[\pi_1'']\right)(x^i).
\end{eqnarray*}
\end{proof}
Analogous to the linear approximation $(\ref{linearapprox})$, we
have
\begin{equation*}
\varphi_{\pi_{11}}^{t_1}(s)=\gamma(t_1)+s v[\pi'_1]+ s
v[\pi''_1]+o(s).
\end{equation*}

If we perturb the control $r$ times,
$\pi=\{\pi_1,\ldots,\pi_r\}$, with $a<t_1\leq \ldots \leq
t_r<b$, then $\gamma[\pi^s](t)$ is the gene\-ra\-lized
integral curve of $X^{\{u[\pi^s]\}}$ with initial condition
$(a,\gamma(a))$. Consider the curve $ \varphi_{\pi}^t\colon
[0,\epsilon]\rightarrow M$ for $t\in [t_r,b]$ given by $
\varphi_{\pi}^t(s)=\gamma[\pi^s](t)$.
\begin{corol}\label{tangentperturbn}
For $t\in[t_r,b]$, the vector tangent to the curve $\;
\varphi_{\pi}^t \colon \; [0,\epsilon]\rightarrow M\;$ at $\; s=0 \;
$ is $V[\pi_1](t)+\ldots+V[\pi_r](t)$, where $V[\pi_i]\colon
[t_i,b]\rightarrow TM$ is the generalized integral curve of
$\left(X^T\right)^{\{u\}}$ with initial condition
$(t_i,(\gamma(t_i),v[\pi_i]))$ for $i=1,\ldots,r$.
\end{corol}
This corollary may be easily proved by induction using
Propositions \ref{tangentperturb}, \ref{12}, \ref{11},
where all the possibilities of combination of perturbation
data have been studied. If $w$ is the vector tangent to $
\varphi_{\pi}^t$ at $s=0$,  the perturbation data will be
denoted by $\pi_w$. Bearing in mind the different
combination of vectors in Definition \ref{conicconvex}, we
have the following definition.
\begin{defin}
The conic non\textendash{}negative combinations of
perturbation vectors of class I and displacements by the
flow of $X^{\{u\}}$ of perturbation vectors of class I are
called \textbf{perturbation vectors of class II}.
\end{defin}

\subsubsection{Perturbation cones} Considering all the
elementary perturbation vectors, we define a closed convex
cone at every time containing at least  all displacements
of these vectors. To transport all the elementary
perturbation vectors, the pushforward of the flow of the
vector field $X^{\{u\}}$ is used. See Appendix
\ref{variationalsection}. Observe that the second comment
after Definition \ref{classI} guarantees that the set of
elementary perturbation vectors is a cone.

\begin{defin}\label{tangent} For $t\in (a,b]$, the \textbf{tangent perturbation cone} $K_t$
is the smallest closed convex cone in $T_{\gamma(t)}M$ that
contains all the displacements by the flow of $X^{\{u\}}$
of all the elementary perturbations vectors from all
Lebesgue times $\tau$ smaller than $t$:
\begin{equation*}
K_t=\overline{{\rm conv}\left(\bigcup_{\substack{ a<\tau\leq t \\
\tau \textrm{ is a Lebesgue time}}}
(\Phi_{(t,\tau)}^{X^{\{u\}}})_* ({\cal V}_{\tau})\right)},
\end{equation*}
where ${\cal V}_{\tau}$ denotes the set of elementary
perturbation vectors at $\tau$ and ${\rm conv}(A)$ means
the convex hull of the set $A$.
\end{defin}
To prove the following statement, we use results in
Appendices \ref{hyperplane} and \ref{Brouwer}; precisely
Proposition \ref{convexhull}, \ref{convex} and Corollary
\ref{scholium}.



\begin{prop} \label{lemma2} Let $t \in (a,b]$. If $v$ is a nonzero vector
in the interior of $K_t$, then there exists $\epsilon>0$ such that
for every $s\in(0,\epsilon)$ there are $s'>0$ and a perturbation of
the control $u[\pi^s]$ such that $\gamma[\pi^s](t)=\gamma(t)+s' v$.
\end{prop}

\begin{proof}
As $v$ is interior to $K_t$, by Proposition \ref{convex}, item
$(d)$, $v$ is in the interior of the cone  \[{\cal C}={\rm
conv}\left(\bigcup_{\substack{ a<\tau\leq t \\ \tau \textrm{ is a
Lebesgue time}}}\left(\Phi^{X^{\{u\}}}_{(t,\tau)}\right)_*{\cal
V}_{\tau}\right),\]  where ${\cal V}_{\tau}$ is the cone of
perturbation vectors of class I at time $\tau$. Hence, $v$ can be
expressed as a convex finite combination of perturbation vectors of
class I by Proposition \ref{convexhull}.

Let $(W,x^i)$ be a local chart of $M$ at $\gamma(t)$. We
suppose that the image of the local chart and $W$ are
identified locally with an open set of $\mathbb{R}^m$.
Through the local chart we also identify $T_{\gamma(t)}M$
with $\mathbb{R}^m$.
We consider the affine hyperplane $\Pi$ orthogonal to $v$
at the endpoint of the vector $v$ and identify $\Pi$ with
$\mathbb{R}^{m-1}$.

A ``closed" cone denotes a closed cone without the vertex.
Observe that such a cone is not closed, that is why we use
the inverted commas. We can choose a ``closed" convex cone
$\widetilde{{\cal C}}$ contained in the interior of ${\cal
C}$ such that $v$ lies in the interior of $\widetilde{{\cal
C}}$ and $\langle w, v \rangle
>0$ for every $w\in \widetilde{{\cal C}}$. For example, we can
consider a circular cone with axis $v$ satisfying the two
previous conditions, as assumed from now on. Hence
$$\Pi\cap \widetilde{{\cal C}}=v+\overline{B(0,R)},$$
where $\overline{B(0,R)}$ is the closure of an open ball in
the subspace orthogonal to $v$, denoted by $v^{\perp}$. For
$r\in v^{\perp}$, we will write $r$ instead of $0v+r$ as a
vector in $\mathbb{R}^m$.

Let us construct a diffeomorphism from the cone
$\widetilde{{\cal C}}$ to a cylinder of $\mathbb{R}^m$. If
$w\in \widetilde{{\cal C}}$, the orthogonal decomposition
of $w$ induced by $v$ and $v^{\perp}$ is
$$w=\frac{\langle w,v\rangle}{\|v\|}\frac{v}{\|v\|}+\left(w-\frac{\langle
w,v\rangle}{\langle v,v\rangle}v\right)=\frac{\langle
w,v\rangle}{\langle v,v\rangle} \left[v+\left(\frac{\langle
v,v\rangle}{\langle w,v\rangle}w-v\right)\right].$$ Observe that
$\frac{\langle v,v\rangle}{\langle w,v\rangle}w-v$ is a vector in
$\overline{B(0,R)}\subset v^{\perp}$. Considering the ``closed" cone
$\widetilde{{\cal C}}$ without the vertex, we have the map
\begin{eqnarray*}
\begin{array}{rcl}g \colon  \widetilde{{\cal C}} & \longrightarrow & \mathbb{R}^+ \times \overline{B(0,R)}
\\
w& \longmapsto & \left(\ds{\frac{\langle w,v\rangle}{\langle
v,v\rangle}, \frac{\langle v,v\rangle}{\langle
w,v\rangle}w-v}\right)=(s,r), \\
\end{array}
\end{eqnarray*}
that is a ${\cal C}^{\infty}$ diffeomorphism with inverse given by
\[
\begin{array}{rccl}g^{-1} \colon & \mathbb{R}^+ \times
\overline{B(0,R)}
& \longrightarrow &  \widetilde{{\cal C}}\\
&(s,r) & \longmapsto &
s(v+r)=w.\\
\end{array}
\]
Note that $g$ and $g^{-1}$ can be extended to an open cone,
without the vertex, containing $\widetilde{{\cal C}}$, so
the condition that $g$ is diffeomorphism is clear.

If we truncate $\widetilde{{\cal C}}$ by the affine hyperplane
$\Pi$, we obtain a bounded convex set $\widetilde{{\cal C}}_v$. The
restriction of $g$ to $\widetilde{{\cal C}}_v$ is $g_v \colon
\widetilde{{\cal C}}_v \rightarrow  (0,1] \times \overline{B(0,R)}$,
that is also a ${\cal C}^{\infty}$ diffeomorphism with inverse
$g_v^{-1} \colon  (0,1]\times \overline{B(0,R)} \rightarrow
\widetilde{{\cal C}}_v$.

If $r\in \overline{B(0,R)}$, then $w_0=v+r$ is interior to ${\cal
C}$. Hence, associated to $w_0$ we have a perturbation $\pi_{w_0}$
of the control $u$. Let $\gamma[\pi^s_{w_0}]\colon I \rightarrow M$
be the generalized integral curve of $X^{\{u[\pi_{w_0}^s]\}}$ with
initial condition $(a,\gamma(a))$ and consider the map
\begin{eqnarray*}
\begin{array}{rccl}\Gamma \colon  & [0,1]\times
\overline{B(0,R)} & \longrightarrow & M\\
&(s,r)& \longmapsto &
\Gamma(s,r)=\gamma[\pi^s_{w_0}](t)\\
&(0,r) & \longmapsto & \Gamma(0,r)=\gamma(t),
\end{array}
\end{eqnarray*}
which is continuous because $\gamma[\pi^s_{w_0}](t)$
depends continuously on $s$ and $\pi^s_{w_0}$ and
\[\lim_{(s,r)\rightarrow(0,r_0)}
\Gamma(s,r)=\gamma(t)=\Gamma(0,r_0).\] Hence, for every
$\epsilon>0$, there exist $\delta_1,\delta_2>0$ such that
if $|s|<\delta_1$ and $\|r\|<\delta_2$, then
$\|\Gamma(s,r)-\Gamma(0,0)\|=\|\gamma[\pi^s_{w_0}](t)-\gamma(t)\|<\epsilon$.

Taking $\epsilon$ such that $B(\gamma(t),\epsilon)$ is
contained in $W$, there exist $\delta_1,\delta_2>0$ such
that if $|s|<\delta_1$ and $\|r\|<\delta_2$, then
$\gamma[\pi^s_{w_0}](t)\in W$.

We consider now the map
\begin{eqnarray*}
\begin{array}{rccl}\Delta \colon  & [0,\delta_1]\times
\overline{B(0,\delta_2)} & \longrightarrow & T_{\gamma(t)}M\simeq \mathbb{R}^m\\
&(s,r)& \longmapsto &
\Delta(s,r)=\gamma[\pi^s_{w_0}](t)-\gamma(t)\\
&(0,r) & \longmapsto &
\Delta(0,r)=0\\
\end{array}
\end{eqnarray*}
that is continuous because
$\lim_{(s,r)\rightarrow(0,r_0)}\Delta(s,r)=0=\Delta(0,r_0)$.
Remember that we have identified $W$ with $\mathbb{R}^m$
via the local chart. With this in mind and using Equation
(\ref{linearapprox}), we can write
$$\gamma[\pi^s_{w_0}](t)-\gamma(t)=s(v+r)+o_r(s),$$
where $o_r(s)\in \mathbb{R}^m$.

We are going to show that, taking $(s,r)$ in an adequate
subset, $\Delta(s,r)$ lies in the interior of the cone
$\widetilde{{\cal C}}$.

Take a section of the cone through a plane containing $v$
and $w$, and compute the distance from the endpoint of $w$
to the boundary of the cone $\widetilde{{\cal C}}$. This is
given by
$$\frac{s\ (R-\|r\|)}{\sqrt{1+\left(\frac{R}{\|v\|}\right)^2}}.$$
This is the maximum value for the radius of an open ball
centered at the endpoint of $s(v+r)$ to be contained in
$\widetilde{{\cal C}}$.

Define the function
\begin{eqnarray*}
\begin{array}{rccl}\Theta \colon & [0,\delta_1] \times \overline{B(0,\delta_2)}  & \longrightarrow & \mathbb{R}^m\\
&(s,r) & \longmapsto &
\left(\gamma[\pi^s_{w_0}](t)-\gamma(t)-s(v+r)\right)/s \ = \ o_r(s)/s\\
&(0,r) & \longmapsto & 0.\\
\end{array}
\end{eqnarray*}
which is continuous because $\lim_{(s,r)\rightarrow(0,r_0)}
\Theta(s,r)=0=\Theta(0,r_0)$. Take
$$\epsilon=\frac{R-\delta_2}{\sqrt{1+\left(\frac{R}{\|v\|}\right)^2}},$$
then there exist $\overline{\delta}_1$,
$\overline{\delta}_2>0$ such that, if
$|s|<\overline{\delta}_1$ and $\|r\|<\overline{\delta}_2$,
then $\left\| \Theta(s,r)
\right\|=\left\|o_r(s)/s\right\|<\epsilon$.

If $(s,r)\in (0,\overline{\delta}_1)\times
\overline{B(0,\overline{\delta}_2)}$, then
$$\|\Delta(s,r)-s(v+r)\|=\|s(v+r)+o_r(s)-s(v+r)\|=\|o_r(s)\|\leq
s \epsilon <s
\frac{R-\|r\|}{\sqrt{1+\left(\frac{R}{\|v\|}\right)^2}}
$$ since $\|r\|\leq \overline{\delta}_2<\delta_2<R$. Thus we conclude that $\Delta(s,r)=s(v+r)+o_r(s)$
is in the interior of
the cone $\widetilde{\cal C}$ for every $(s,r)\in
(0,\overline{\delta}_1)\times
\overline{B(0,\overline{\delta}_2)}$.

Now, for $s\in (0,\overline{\delta}_1)$, we define the continuous
mapping
\begin{eqnarray}\label{Gs}
\begin{array}{rcl}G_s \colon  \overline{B(0,\overline{\delta}_2)} & \longrightarrow & \overline{B(0,R)} \subset
\mathbb{R}^{m-1} \\
r& \longmapsto &
G_s(r)=\left(\pi_2 \circ g \circ \Delta  \right)(s,r),\\
\end{array}
\end{eqnarray}
where $\pi_2 \colon \mathbb{R}^+\times \overline{B(0,R)}
\rightarrow \overline{B(0,R)}$, $\pi_2(s,r)=r$. Observe
that for $r_0\in \overline{B(0,\overline{\delta}_2)}$ we
have
$$\lim_{(s,r)\rightarrow (0,r_0)} G_s(r)= \lim_{(s,r)\rightarrow (0,r_0)} \left[ \frac{\langle
v,v\rangle}{s\langle v,v\rangle+\langle o(s),v\rangle}
(s(v+r)+ o(s))-v\right]=r$$ and
\begin{equation}\label{gDelta}(g\circ
\Delta)(s,r)=g(\gamma[\pi^s_{w_0}](t)-\gamma(t))=g(s(v+r)+o_r(s))=(s',r').
\end{equation}
Suppose that there exists $r\in\overline{B(0,R)}$ such that
$G_s(r)=0$. Then applying $g^{-1}$ to (\ref{gDelta}), we
have
\begin{equation}
\label{Delta}
\Delta(s,r)=\gamma[\pi^s_{w_0}](t)-\gamma(t)=g^{-1}(s',0)=s'v.\end{equation}
Hence, to conclude the proof we need to show that there
exists $r$ with $G_s(r)=0$ for $s$ small enough. To apply
Corollary \ref{scholium}, there must exist $r'\in
B(0,\overline{\delta}_2)$ such that $\|G_s(r)-r\|<\|r-r'\|$
for every $r \in
\partial \left(\overline{B(0,\overline{\delta}_2)}\right)$. We will
show that the condition is fulfilled for $r'=0$.

Consider the mapping
\begin{eqnarray*}
\begin{array}{rccl}{\cal G} \colon  & [0,\overline{\delta}_1]\times
\overline{B(0,\overline{\delta}_2)} & \longrightarrow &
\overline{B(0,R)} \subset
\mathbb{R}^{m-1} \\
&(s,r)& \longmapsto &
{\cal G}(s,r)=G_s(r)-r\\
&(0, r) & \longmapsto & {\cal G}(0,r)=0.\\
\end{array}
\end{eqnarray*}
For $r_0\in \overline{B(0,\overline{\delta}_2)}$, we have
$\lim_{(s,r)\rightarrow (0,r_0)} {\cal G}(s,r)=
\lim_{(s,r)\rightarrow (0,r_0)} G_s(r)-r=0$. Thus ${\cal G}$ is
continuous.

Given $r_0\in
\partial \left(\overline{B(0,\overline{\delta}_2)}\right)$, take
$\epsilon=\overline{\delta}_2/2$, then there exist
$\delta_0(0,r_0)$, $\delta_1(0,r_0)>0$ such that if
$|s|<\delta_0(0,r_0)$ and $\|r-r_0\|<\delta_1(0,r_0)$, then
$\|{\cal G}(s,r)-{\cal G}(0,r_0)\|<\overline{\delta}_2/2$.
Hence $\left\{B(r_0,\delta_1(0,r_0)) \, | \, r_0\in
\partial \left(\overline{B(0,\overline{\delta}_2)}\right)
\right\}$ is an open covering of the boundary of
$\overline{B(0,\overline{\delta}_2)}$; $\partial
\overline{B(0,\overline{\delta}_2)}$. As this is a compact
set, there exists a finite subcovering,
$$\{B(r_1,\delta_1(0,r_1)), \ldots, B(r_l,\delta(0,r_l))\}.$$
Take $\delta$ as the minimum of
$\{\delta_0(0,r_1),\ldots,\delta_0(0,r_l)\}$. Let us see
that, for every $(s,r)\in [0,\delta]\times
\partial \overline{B(0,\overline{\delta}_2)}$,
$\|G_s(r)-r\|<\|r\|$. As $r$ is in an open set of the
finite subcovering,
$$\|{\cal G}(s,r)\|=\|G_s(r)-r\|<\frac{\overline{\delta}_2}{2}<\overline{\delta}_2=\|r\|.$$
Hence, using Corollary \ref{scholium}, for every $s\in
(0,\delta)$, the set
$G_s(\overline{B(0,\overline{\delta}_2)})$ covers the
origin; that is, there exists $r\in
\overline{B(0,\overline{\delta}_2)}$ such that
$$G_s(r)=(\pi_2\circ g \circ \Delta)(s,r)=0.$$
Then, because of the definition of the mapping $G_s$ in
Equation (\ref{Gs}) and Equations (\ref{gDelta}) and
(\ref{Delta}), there exists $s'\in \mathbb{R}^+$ such that
$$ \gamma[\pi_{w_0}^s](t)= \gamma(t)+s'v.$$ To finish the proof we only need to take
$\pi^s=\pi^s_{w_0}$. In other words, we have a trajectory coming
from a perturbation of the control that meets the ray generated by
$v$, as wanted.
\end{proof}

\subsection{Pontryagin's Maximum Principle in the symplectic formalism for the optimal control
problem}\label{symplectic}

In this section, the $OCP$ is transformed into a Hamiltonian problem
that will allow us to state Pontryagin's Maximum Principle.

Given the $OCP$ $(Q,U,X,F,I,x_a,x_b)$ and the
$\widehat{OCP}$ $(\widehat{Q},U,\widehat{X},I,x_a,x_b)$,
let us consider the cotangent bundle $T^{\ast} \widehat{Q}$
with its natural symplectic structure that will be denoted
by $\omega$. If
$(\widehat{x},\widehat{p})=(x^0,x,p_0,p)=(x^0, x^1, \ldots
, x^m,p_0, p_1, \ldots , p_m)$ are local natural
coordinates on $T^* \widehat{Q}$, the form $\omega$ has as
its local expression $\omega=dx^0\wedge dp_0+dx^i\wedge
dp_i$.

For each $u \in U$, $H^u \colon T^* \widehat{Q} \rightarrow
\mathbb{R}$ is the Hamiltonian function defined by
$$H^u(\widehat{p})=H(\widehat{p},u)=
\langle \widehat{p},\widehat{X}(\widehat{x},u) \rangle =p_0
F(x,u) +\sum_{i=1}^m p_i f^i(x,u),$$ where $\widehat{p}\in
T^*_{\widehat{x}}\widehat{Q}$. The tuple $(T^* \widehat{Q},
\omega, H^u)$ is a Hamiltonian system. Using the notation
in (\ref{NotVFproj}), the associated Hamiltonian vector
field $Y^{\{u\}}$ satisfies the equation
$$i(Y^{\{u\}}) \omega =dH^u.$$
Thus  we get a family of Hamiltonian systems parameterized
 by $u$,
$H \colon T^* \widehat{Q} \times U \rightarrow \mathbb{R}$, and the
associated Hamiltonian vector field $Y\colon T^*\widehat{Q} \times U
\rightarrow T(T^*\widehat{Q})$ which is a vector field along the
projection $\widehat{\pi}_1 \colon T^*\widehat{Q} \times U
\rightarrow T^* \widehat{Q}$. Its local expression is
\[Y(\widehat{p},u)=\left(F(x,u) \frac{\partial}{\partial x^0}
+f^i(x,u)\frac{\partial}{\partial x^i} +0 \frac{\partial}{\partial
p_0} +\left(-p_0 \frac{\partial F}{\partial x^i}(x,u)- p_j
\frac{\partial f^j}{\partial x^i}(x,u)\right)
\frac{\partial}{\partial
p_i}\right)_{(\widehat{x},\widehat{p},u)}.\] It should be noted that
$Y=\widehat{X}^{T^*}$ is the cotangent lift of $\widehat{X}$. See
Appendix \ref{colift} for definition and properties of the cotangent
lift.

Given a curve $(\widehat{\lambda}, u) \colon I \rightarrow
T^{\ast}\widehat{Q} \times U$ such that it is absolutely continuous
on $T^{\ast}\widehat{Q}$, it is measurable and bounded on $U$, and
$\widehat{\gamma}=\pi_{\widehat{Q}} \circ \widehat{\lambda}$; if
$\pi_{\widehat{Q}} \colon T^*\widehat{Q} \rightarrow \widehat{Q}$ is
the natural projection, the previous elements come together in the
following diagram:
$$\bfig\xymatrix{ & & T(T^*\widehat{Q})\ar[d]^{\txt{\small{$\tau_{T^*\widehat{Q}}$}}}\\
\mathbb{R} & T^*\widehat{Q}\times U
\ar[l]_{\txt{\small{$H$}}}
\ar[ur]^{\txt{\small{$\widehat{X}^{T^*}$}}}\ar[r]^{\txt{\small{$\widehat{\pi}_1$}}}
&T^*\widehat{Q}\ar[d]^{\txt{\small{$\pi_{\widehat{Q}}$}}}\\
& I
\ar[u]^{\txt{\small{$(\widehat{\sigma},u)$}}}\ar[ur]^{\txt{\small{$\widehat{\sigma}$}}}
\ar[r]^{\txt{\small{$\widehat{\gamma}$}}}\ar[dr]^{\txt{\small{$\gamma$}}}&
\widehat{Q}\ar[d]^{\txt{\small{$\pi_2$}}}\\ & &Q }\efig$$

\begin{state}\textbf{(Hamiltonian Problem, $HP$)}
Given the $OCP$ $(Q,U,X,F,I,x_a,x_b)$, and the equivalent
$\widehat{OCP}$ $(\widehat{Q},U,\widehat{X},I,x_a,x_b)$, consider
the following problem.

Find $(\widehat{\sigma}^*, u^*)$ such that

\begin{itemize}

\item[(1)] $\widehat{\gamma}^*(a)=
(0,x_a)$ and $\gamma^*(b)= x_b$, if
$\widehat{\gamma}^*=\pi_{\widehat{Q}} \circ \widehat{\sigma}^*$,
$\gamma^*=\pi_2 \circ \widehat{\gamma}^* $.
\item[(2)]
$\dot{\widehat{\sigma}^*}(t)=\widehat{X}^{T^*}(\widehat{\sigma}^*(t),u^*(t))$,
$t \in I$.
\end{itemize}
\end{state}

The tuple $(T^*\widehat{Q}, U, \widehat{X}^{T^*}, I, x_a, x_b)$
denotes the \textit{Hamiltonian problem} as it has just been defined
and the elements satisfy the same properties as in \S \ref{general}.

 \comments The Hamiltonian problem satisfies analogous
conditions to those satisfied by the $OCP$ and the
$\widehat{OCP}$ defined in \S \ref{PCOfijo} and \S
\ref{PMPEfijo} respectively.
\begin{enumerate}
\item Given  $(\widehat{\sigma},u)$,  the
function  $u \colon I \rightarrow U$ allows  us to
 construct  a
time\textendash{}dependent vector field  on
$T^*\widehat{Q}$, $(\widehat{X}^{T^*})^{\{u\}} \colon
T^*\widehat{Q}\times I \rightarrow T(T^*\widehat{Q})$,
defined by  $$(\widehat{X}^{T^*})^{\{u\}} (\widehat{x},
\widehat{p}, t)=\widehat{X}^{T^*}(\widehat{x},\widehat{p},
u(t)).$$ Condition $(2)$ shows that $\widehat{\sigma}^*$ is
an integral curve of $(\widehat{X}^{T^*})^{\{u^*\}}$.
\item The vector field $(\widehat{X}^{T^*})^{\{u\}}$ is
$\pi_{\widehat{Q}}$\textendash{}projectable and projects
onto $\widehat{X}^{\{u\}}$. Thus if $\widehat{\sigma}$ is
an integral curve of $(\widehat{X}^{T^*})^{\{u\}}$,
$\widehat{\gamma}=\pi_{\widehat{Q}} \circ \widehat{\sigma}$
is an integral curve of $\widehat{X}^{\{u\}}$.


\item Locally, conditions (1) and (2) are equivalent to the fact that the curve $(\widehat{\sigma},u)$
satisfies the Hamilton equations of the system $(T^*\widehat{M},
\omega, H^u)$,
\begin{eqnarray} \dot{x}^0 &=& \frac{\partial
H^u}{\partial p_0}= F \nonumber\\
\dot{x}^i &=& \frac{\partial H^u}{\partial p_i}= f^i \nonumber\\
\label{dotp0}\dot{p}_0&=&-\frac{\partial H^u}{\partial x^0}= 0 \Rightarrow p_0= {\rm ct}\\
\label{dotp}\dot{p}_i&=&-\frac{\partial H^u}{\partial
x^i}=-p_0 \frac{\partial F}{\partial x^i}- p_j
\frac{\partial f^j}{\partial x^i},
\end{eqnarray}
and satisfies the conditions $\widehat{\gamma}(a)=
(0,x_a)$, $\gamma(b)= x_b$.

In the literature of optimal control, the system of
differential equations given by Equations (\ref{dotp0}),
(\ref{dotp}) is called the \textit{adjoint system}. In
differential geometry, the adjoint system is the
differential equations satisfied by the fiber coordinates
of an integral curve of the cotangent lift of a vector
field on $Q$. See Appendix \ref{colift} for more details.

Note that there is no initial condition for
$\widehat{p}=(p_0, p_1, \ldots, p_m)$, hence HP is not a
Cauchy problem.
\end{enumerate}

\comment So far we have considered a fixed control $u\in U$.
Therefore we have been working with a family of Hamiltonian systems
on the manifold $(T^*\widehat{Q},\omega)$ given by the Hamiltonians
$\{ H^u | \, u\in U\}$.

Given $u\colon I \rightarrow U$, then we consider the Hamiltonian
$H^{u(t)}$. The equation of the Hamiltonian vector field for the
Hamiltonian system $(T^*\widehat{Q},\omega, H^{u(t)})$ is
$$i(Y^{\{u(t)\}}) \omega =d_{\widehat{Q}} H^{u(t)},$$
where $d_{\widehat{Q}}$ is the exterior differential on the
manifold $T^*\widehat{Q}$. Observe that we have studied the
system defined by $(T^*\widehat{Q},\omega, H^{u(t)})$ as an
autonomous system by fixing the time $t$. The Hamiltonian
vector field obtained $Y^{\{u(t)\}}$ is a
time\textendash{}dependent vector field whose integral
curves satisfy the equation
\begin{equation}\label{Hamiltonianvf}
\dot{\widehat{\sigma}}(t)=Y^{\{u(t)\}}
(\widehat{\sigma}(t)), \quad t\in I.
\end{equation}
Observe that $Y^{\{u(t)\}}=(\widehat{X}^{T^*})^{\{u(t)\}}$.

Now we are ready to state Pontryagin's Maximum Principle
that provides the necessary conditions, which are in
general not sufficient, to find solutions of the optimal
control problem.

\begin{teo}\label{PMP} \textbf{ (Pontryagin's Maximum Principle, PMP)}
\\
If $(\widehat{\gamma}^*,u^*)\colon I \rightarrow \widehat{Q}\times
U$ is a solution of the extended optimal control problem, Statement
\ref{stateEOCP}, such that $\widehat{\gamma}^*$ is absolutely
continuous and $u^*$ is measurable and bounded, then there exists
$(\widehat{\sigma}^*, u^*)\colon I \rightarrow T^*\widehat{Q} \times
U$ such that:
\begin{enumerate}
\item it is a solution of the Hamiltonian problem,
that is, it satisfies Equation (\ref{Hamiltonianvf}) and the initial
conditions $\widehat{\gamma}^*(a)=(0,x_a)$ and $\gamma^*(b)=x_b$, if
$\gamma^*=\pi_2 \circ \widehat{\gamma}^*$;
\item $\widehat{\gamma}^*=\pi_{\widehat{Q}} \circ \widehat{\sigma}^*$;
\item \begin{itemize}
\item[(a)] $H(\widehat{\sigma}^*(t),u^*(t))= \sup_{u \in U} H(\widehat{\sigma}^*(t),
u)$ almost everywhere;
\item[(b)] $\sup_{u \in U} H(\widehat{\sigma}^*(t),
u)$ is constant everywhere;
\item[(c)] $\widehat{\sigma}^*(t)\neq 0 \in
T^{\ast}_{\widehat{\gamma}^*(t)} \widehat{Q}$ for each
$t\in [a,b]$;
\item[(d)] $\sigma_0^*(t)$ is constant and $\sigma_0^*(t)\leq 0$.
\end{itemize}
\end{enumerate}
\end{teo}

 \comments
\begin{enumerate}
\item There exists an abuse of notation between $\widehat{\sigma}(t)\in T^*\widehat{M}$ and
$\widehat{\sigma}(t)\in T^*_{\widehat{\gamma}^*(t)}\widehat{M}$. We
assume that the meaning of $\widehat{\sigma}$ in each situation will
be clear from the context.
\item Condition $(2)$ is immediately satisfied because
$\widehat{\sigma}^*$ is a covector along
$\widehat{\gamma}^*$.
\item Conditions $(3a)$ and $(3b)$ imply that the
Hamiltonian function is constant almost everywhere for $t\in [a,b]$.
\item In item $(3a)$, if $U$ is a closed set, then the maximum of the Hamiltonian
over the controls is considered instead of the the supremum
over the controls. But in condition $(3b)$ we can always
consider the maximum, instead of the supremum, because item
$(3a)$ guarantees that the supremum of the Hamiltonian is
reached in the optimal curve. Thus, from now on and
according to the classical literature, we refer to the
assertion $(3a)$ as the condition of maximization of the
Hamiltonian over the controls.
\item Condition $(3c)$ implies that $\sigma^*_0(t)\neq 0$ or $\sigma^*(t)\neq 0 \in
T^*_{\gamma^*(t)}M$ for each $t\in [a,b]$. Locally the condition
$(3c)$ states that for each $t\in [a,b]$ there exists a coordinate
of $\widehat{\sigma}^*(t)$ nonzero, $(p_i\circ
\widehat{\sigma}^*)(t)=\sigma^*_i(t)\neq 0$.
\item From the Hamilton's equations of the system $(T^*\widehat{M}, \omega,
H^{u(t)})$, it is concluded that $\sigma_0$ is constant along the
integral curves of $(\widehat{X}^{T^*})^{\{u(t)\}}$, since
$\dot{p}_0=0$. Hence the first result in $(3d)$ is immediate for
every integral curve of $(\widehat{X}^{T^*})^{\{u(t)\}}$. As
$\sigma_0^*$ is constant, $\widehat{\sigma}^*$ may be normalized
without loss of generality. Thus it is assumed that either
$\sigma^*_0=0$ or $\sigma^*_0=-1$ because of the second result in
$(3d)$.
\item Pontryagin's Maximum Principle only guarantees that
given a solution of $\widehat{OCP}$ there exists a solution
of $HP$. Hence, in principle, both problems are not
e\-qui\-va\-lent.
\end{enumerate}

Observe that Maximum Principle guarantees the existence of a
covector along the optimal curve, but it does not say anything about
the uniqueness of the covector. Indeed, this covector may not be
unique. Depending on the covector we associate with the optimal
curves, different kind of curves can be defined.
\begin{defin}  \label{definextremal}
A curve $(\widehat{\gamma}, u)\colon [a,b] \rightarrow
\widehat{Q} \times U$ for $\widehat{OCP}$ is
\begin{enumerate}
\item an \textbf{extremal} if there exist $\widehat{\sigma}\colon [a,b] \rightarrow
T^*\widehat{Q}$ such that $\widehat{\gamma}=\pi_{T\widehat{Q}} \circ
\widehat{\sigma}$ and $(\widehat{\sigma}, u)$ satisfies the
necessary conditions of PMP;
\item a \textbf{normal extremal} if it
is an extremal with $\sigma_0=-1$ and $\widehat{\sigma}$ is called a
\textbf{normal lift or momenta};
\item an \textbf{abnormal extremal} if it is an extremal with
$\sigma_0=0$ and $\widehat{\sigma}$ is called an \textbf{abnormal
lift or momenta};
\item a \textbf{strictly abnormal
extremal} if it is not a normal extremal, but it is
abnormal;
\item a \textbf{strictly normal
extremal} if it is not a abnormal extremal, but it is
normal.
\end{enumerate}
\end{defin}

In \cite{2004Agrachev,1996Troutman} there are some examples of
optimal control problems whose solutions are searched using
Pontryagin's Maximum Principle.

Observe that if $\widehat{\gamma}\colon I \rightarrow \widehat{Q}$
is an integral curve of a vector field, there always exists a lift
of $\widehat{\gamma}$ to a curve $\widehat{\sigma} \colon I
\rightarrow T^*\widehat{Q}$, given an initial condition for the
cofibers, which is an integral curve of the cotangent lift of the
given vector field on $\widehat{Q}$. Analogously, if the system is
given by a vector field along the projection $\widehat{\pi}\colon
\widehat{Q}\times U \rightarrow \widehat{Q}$.

Therefore, the items 1 and 2 in Theorem \ref{PMP} do not provide any
information related with the optimality. They only ask for the
fulfilment of a final condition in the integral curve. The
accomplishment of this depends on the accessibility of the problem,
see \cite{2005BulloAndrewBook,VanderSchaft}.

The real contribution of PMP is the third item related with the
optimality through the maximization of the Hamiltonian, that will be
only satisfied if the initial conditions for the fibers are chosen
suitably. This is the key element of the proof of Pontryagin's
Maximum Principle. In other words, we can always find a cotangent
lift of an integral curve such that conditions 1 and 2 are satisfied
under the assumption of accessibility, but it is not guaranteed the
fulfilment of conditions in assertion 3 in Theorem \ref{PMP}.

If we write the Hamiltonian function for the abnormal and the normal
case, the difference is that the cost function does not play any
role in the Hamiltonian for abnormal extremals. That is why it is
said the abnormal extremals only depend on the geometry of the
control system. But to determine the optimality of the abnormal
extremals, the cost function is essential. In fact, for the same
control system different optimal control problems can be stated
depending on the cost function, in such a way that the abnormal
extremals are minimizers only for some of the problems.

To conclude, the strict abnormality characterizes the abnormal
extremals that are not normal. An extremal is not normal when there
does not exist any covector that satisfies Hamilton's equations for
normality. Thus it is necessary to know the cost function in order
to prove that there exists only one kind of lift.

\section{Proof of Pontryagin's Maximum Principle for fixed time
and fixed endpoints}\label{proofPMPfixed}

To prove Pontryagin's Maximum Principle it is necessary to use
analytic results about absolute continuity and lower semicontinuity
for real functions, and properties of convex cones. For the details
see Appendix \ref{analytic} and Appendix \ref{hyperplane} and
references therein. The reader is referred to \S
\ref{perturbedconesfixed} for results on perturbations of a
trajectory in a control system.

In the literature of optimal control, the proof of the Maximum
Principle has been discussed taking into account varying hypotheses,
\cite{2004Agrachev,1966Athans,2005BonnardCaillau,1997Jurdjevic,S98Free,2000Sussmann,2005Sussmann}.
Most authors believe and justify that the origin of this Principle
is the calculus of variations, see \cite{84Zeidler} for instance.

\begin{proof}\textit{(Theorem \ref{PMP}: Pontryagin's Maximum Principle, PMP)}\\
\textbf{1.} As $(\widehat{\gamma}^*,u^*)$ is a solution of
$\widehat{OCP}$, if $\tau$ is in $[a,b]$, for every initial
condition $\widehat{\sigma}_{\tau}$ in
$T^*_{\widehat{\gamma}^*(\tau)}\widehat{Q}$, we have a solution of
$HP$, $(\widehat{\gamma}^*,\widehat{\sigma})\colon [a,b] \rightarrow
T^*\widehat{Q}$, satisfying that initial condition. The covector
$\widehat{\sigma}_{\tau}$ must be chosen conveniently so that the
remaining conditions of the PMP are satisfied.

According to \S \ref{perturbedconesfixed}, we construct the tangent
perturbation cone $\widehat{K}_b$ in $T_{\widehat{\gamma}^*(b)}
\widehat{Q}$ that contains all tangent vectors associated with
perturbations of the trajectory $\widehat{\gamma}^*$ corresponding
to variations of $u^*$; see Definition \ref{tangent}.

Let us consider the vector
$(-1,\mathbf{0})_{\widehat{\gamma}^*(b)}\in
T_{\widehat{\gamma}^*(b)}\widehat{Q}$, where the zero in bold
emphasizes that 0 is a vector in $T_{\gamma^*(b)}M$. The vector
$(-1,\mathbf{0})$ has the following properties:
\begin{enumerate}
\item the variation of $x^0(t)=\int^t_aF(\gamma^*(s),u^*(s))ds$ along $(-1,\mathbf{0})$ is negative;
\item it is not interior to $\widehat{K}_b$.

Let us prove the second assertion. Take a local chart at
$\widehat{\gamma}^*(b)$ and work on the image of the local
chart, in $\mathbb{R}^{m+1}$, without changing the
notation.

If $(-1,\mathbf{0})_{\widehat{\gamma}^*(b)}$ was interior to
$\widehat{K}_b$, by Proposition \ref{lemma2} there would exist a
po\-si\-ti\-ve number $\epsilon$, such that for every
$s\in(0,\epsilon)$, there would exist a positive number $s'$, close
to $s$, and a perturbation of the control $u[\pi^s]$ such that
$$\widehat{\gamma}[\pi^s](b)=(\gamma^0[\pi^s](b),\gamma[\pi^s](b))=\widehat{\gamma}^*(b)+s'(-1,\mathbf{0}).$$
For this perturbed trajectory we have
$$\gamma^0[\pi^s](b)<\gamma^{*^0}(b) \quad {\rm and} \quad \gamma[\pi^s](b)=
\gamma^*(b).$$ Hence there would be a trajectory,
$\widehat{\gamma}[\pi^s]$, from $\gamma^*(a)$ to
$\gamma^*(b)$ with less cost than $\widehat{\gamma}^*$.
Hence $\widehat{\gamma}^*$ would not be optimal. In other
words, $(-1,\mathbf{0})_{\widehat{\gamma}^*(b)}$ is the
direction of decreasing of the functional to be minimized
in the extended optimal control problem.


\end{enumerate}
The second property implies that $\widehat{K}_b$ cannot be equal to
$T_{\widehat{\gamma}^*(b)}\widehat{Q}$. As
$(-1,\mathbf{0})_{\widehat{\gamma}^*(b)}$ is not interior to
$\widehat{K}_b$, there exist a separating hyperplane of
$\widehat{K}_b$ and $(-1,\mathbf{0})_{\widehat{\gamma}^*(b)}$ by
Proposition \ref{propseparated}; that is, there exists a nonzero
covector determining a separating hyperplane. Let
$\widehat{\sigma}_b \in T_{\widehat{\gamma}^*(b)}^*{\widehat{Q}}$ be
nonzero such that $\ker \widehat{\sigma}_b$ is a separating
hyperplane satisfying
\begin{equation*}\begin{array}{l}\langle \widehat{\sigma}_b,
(-1,\mathbf{0}) \rangle \geq 0, \\ \langle \widehat{\sigma}_b,
\widehat{v}_b \rangle \leq 0 \quad \forall \, \widehat{v}_b\in
\widehat{K}_b.\end{array}
\end{equation*}
Observe that if $\widehat{\sigma}_b=0\in
T_{\widehat{\gamma}^*(b)}^*{\widehat{Q}}$, $\ker \widehat{\sigma}_b$
does not determine a hyperplane, but the whole space
$T_{\widehat{\gamma}^*(b)}^*{\widehat{Q}}$.

Given the initial condition $\widehat{\sigma}_b\in
T_{\widehat{\gamma}^*(b)}^*{\widehat{Q}}$, there exists only one
integral curve $\widehat{\sigma}^*$ of
$(\widehat{X}^{T^*})^{\{u^*\}}$ such that
$\widehat{\sigma}^*(b)=\widehat{\sigma}_b$.
Hence $(\widehat{\sigma}^*,u^*)$ is a solution of $HP$.\\

\textbf{2.} Obviously, by construction,
$\widehat{\gamma}^*=\pi_{\widehat{Q}} \circ
\widehat{\sigma}^*$.

Now we prove that $\widehat{\sigma}^*$, the solution of
$HP$,
satisfies the remaining conditions of the PMP.\\

\textbf{(3a)} $H(\widehat{\sigma}^*(t),u^*(t))= \sup_{u \in U}
H(\widehat{\sigma}^*(t), u)$ almost everywhere.

We are going to prove the statement for every Lebesgue time, hence
it will be true almost everywhere, see Appendix \ref{analytic} for
more details. Suppose that there exists a control
$\widetilde{u}\colon I \rightarrow \overline{U}$ and a Lebesgue time
$t_1$ such that $u^*$ does not give the supremum of the Hamiltonian
at $t_1$; that is,
$$H(\widehat{\sigma}^*(t_1),\widetilde{u}(t_1)) > H(\widehat{\sigma}^*(t_1),u^*(t_1)).$$
As $H(\widehat{p},u)= \langle
\widehat{p},\widehat{X}(\widehat{x},u) \rangle$,
\[ \langle
\widehat{\sigma}^*(t_1),\widehat{X}(\widehat{\gamma}^*(t_1),
\widetilde{u}(t_1))- \widehat{X}(\widehat{\gamma}^*(t_1),u^*(t_1))
\rangle
>0;\]
that is, $\langle \widehat{\sigma}^*(t_1),\widehat{v}[\pi_1]
\rangle>0$ where $\widehat{v}[\pi_1]=
\widehat{X}(\widehat{\gamma}^*(t_1), \widetilde{u}(t_1))-
\widehat{X}(\widehat{\gamma}^*(t_1),u^*(t_1))$ in $\widehat{K}_{t_1}
\subset T_{\widehat{\gamma}^*(t_1)}\widehat{Q}$ is the elementary
perturbation vector associated with the perturbation data $\pi_1=
\{t_1,1,\widetilde{u}(t_1)\}$ by Proposition \ref{tangentperturb}.

Let $\widehat{V}[\pi_1]\colon [t_1,b] \rightarrow T\widehat{Q}$ be
the integral curve of $(\widehat{X}^T)^{\{u^*\}}$ with
$(t_1,\widehat{\gamma}^*(t_1),\widehat{v}[\pi_1])$ as initial
condition. For $ \widehat{\sigma}^*$, solution of $HP$, the
continuous function $ \langle \widehat{\sigma}^*, \widehat{V}[\pi_1]
\rangle \colon [t_1,b] \rightarrow \mathbb{R}$ is constant
everywhere by Proposition \ref{constant}. Hence $\langle
\widehat{\sigma}^*(t_1),\widehat{v}[\pi_1] \rangle>0$ implies that
$\langle \widehat{\sigma}_b,\widehat{V}[\pi_1](b) \rangle>0$, which
is a contradiction with $\langle \widehat{\sigma}_b, \widehat{v}_b
\rangle \leq 0$ for every $\widehat{v}_b \in \widehat{K}_b$, since
$\widehat{V}[\pi_1](b) \in \widehat{K}_b$.

Therefore, $$H(\widehat{\sigma}^*(t), u^*(t))=\sup_{u \in U}
H(\widehat{\sigma}^*(t),u)$$ at every Lebesgue time
on $[a,b]$, so almost everywhere.\\


\textbf{(3b)} $\sup_{u \in U} H(\widehat{\sigma}^*(t), u)$ is
constant everywhere.

In fact, because of $(3a)$ we know that the supremum is achieved
along the optimal curve, so at the end we are going to prove that
$\max_{u \in U} H(\widehat{\lambda}^*(t), u)$ is constant
everywhere. To simplify the notation we define the function
$$\begin{array}{rcl} {\cal M} \circ \widehat{\sigma}^* \colon
I & \longrightarrow & \mathbb{R} \\
t & \longmapsto & {\cal M}( \widehat{\sigma}^*(t))=\sup_{u \in U}
H(\widehat{\sigma}^*(t), u).
\end{array}$$

In order to prove $(3b)$, it is enough to see that ${\cal
M}(\widehat{\sigma}^*(t))$ is constant everywhere.

First let us see that ${\cal M}\circ \widehat{\sigma}^*$ is lower
semicontinuous on $I$. See Appendix \ref{analytic} for details of
this property. As ${\cal M} (\widehat{\sigma}^*(t))$ is the supremum
of the Hamiltonian function with respect to control, for every
$\epsilon
>0$, there exists a control $u_{\cal M}\colon I \rightarrow U$ such that
\begin{equation}\label{defmax}
H(\widehat{\sigma}^*(t),u_{\cal M}(t))\geq {\cal M}
(\widehat{\sigma}^*(t))-\frac{\epsilon}{2}
\end{equation}
everywhere.

For each constant control $\widetilde{u} \in U$,
$H^{\widetilde{u}}\circ
\widehat{\sigma}^*=H(\widehat{\sigma}^*,\widetilde{u})\colon
I \rightarrow \mathbb{R}$ is continuous on $I$. Hence for
every $t_0\in I$ and $\epsilon>0$, there exists $\delta>0$
such that $\mid t- t_0\mid<\delta$, we have
$$ \mid H^{\widetilde{u}}(\widehat{\sigma}^*(t))
-H^{\widetilde{u}}(\widehat{\sigma}^*(t_0))\mid
<\frac{\epsilon}{2}.$$

If $\widetilde{u}=u_{\cal M}(t_0)$, then using the
continuity of $H^{\widetilde{u}}\circ \widehat{\sigma}^*$
$${\cal M}(\widehat{\sigma}^*(t))=\sup_{u\in U} H(\widehat{\sigma}^*(t),u)
\geq H(\widehat{\sigma}^*(t),u_{\cal M}(t_0)) \geq $$
$$\geq H(\widehat{\sigma}^*(t_0),u_{\cal M}(t_0))-\frac{\epsilon}{2} \geq {\cal M}(\widehat{\sigma}^*(t_0))-\epsilon.$$
The last inequality is true by evaluating Equation (\ref{defmax}) at
$t_0$. Hence ${\cal M}\circ \widehat{\sigma}^*$ is lower
semicontinuous at every $t_0\in I$; that is, ${\cal M}\circ
\widehat{\sigma}^*$ is lower semicontinuous on $I$.

The control $u^*$ is bounded, that means ${\rm Im}\, u^*$ is
contained in a compact set $D \subset U$. Let us define a new
function
$$\begin{array}{rcl} {\cal M}_D \colon T^*\widehat{Q}
& \longrightarrow & \mathbb{R} \\
\beta & \longmapsto & {\cal M}_D(\beta)=\sup_{\widetilde{u}\in D}
H(\beta,\widetilde{u}).
\end{array}$$
As $H(\beta,\, \cdot \,) \colon D
 \rightarrow \mathbb{R}$, $\widetilde{u} \mapsto H(\beta,\widetilde{u})$
 is continuous by hypothesis and $D$ is compact, for every $\beta \in T^*\widehat{Q}$ there
exists a control $\widetilde{w}_{\beta}$ that gives us the
maximum of $H(\beta,\widetilde{u})$
\begin{equation}\label{Hmaximum} {\cal M}_D(\beta)=\sup_{\widetilde{u} \in
D} H(\beta, \widetilde{u})=H(\beta, \widetilde{w}_{\beta}).
\end{equation}
Hence ${\cal M}_D$ is well\textendash{}defined on
$T^*\widehat{Q}$. The following sketch explains in a
compact way the necessary steps to prove that ${\cal M}
\circ \sigma^* $ is constant everywhere. In this sketch,
the figures in bold refer to statements which are going to
be proved in the next paragraphs and a.c. stands for
absolutely continuous and a.e. for almost everywhere.
$$\begin{array}{l}\left.\begin{array}{l} \hspace{6.2cm} {\cal M} \circ \widehat{\sigma}^* \, {\rm is} \; {\rm
lower \; semicontinuous \; on \;} I
\\H^{\widetilde{u}} \in {\cal C}^{1}(T^*\widehat{Q}) \\ \hspace{1cm}
\Downarrow_{\mathbf{1}}\\ H^{\widetilde{u}} \, {\rm is} \; {\rm
locally} \; {\rm Lipschitz} \;\forall \; \widetilde{u} \in D
\\ \hspace{1cm} \Downarrow_{\mathbf{2}} \\ \left.\begin{array}{l} {\cal M}_D \, {\rm is} \;
{\rm locally} \ {\rm Lipschitz}\ {\rm on} \ {\rm Im}
\widehat{\sigma}^*
\\ \\ \widehat{\sigma}^* \, {\rm is} \; {\rm a.c.}
\end{array}\right\}
\Rightarrow^{\mathbf{3}} {\cal M}_D \circ \widehat{\sigma}^* \, {\rm
is} \; {\rm a.c.}
\Rightarrow {\cal M}_D \circ \widehat{\sigma}^* \, {\rm is} \; {\rm continuous} \\
\hspace{6.2cm} {}^{\mathbf{4}} \, {\cal M}_D(\widehat{\sigma}^*(t))
\leq {\cal M}(\widehat{\sigma}^*(t)),  \; \forall
\; t\in [a,b]\\
\hspace{6.2cm} {}^{\mathbf{5}} \, {\cal
M}_D(\widehat{\sigma}^*(t))={\cal M}(\widehat{\sigma}^*(t)) \, {\rm
a.e.}
\end{array} \right\} \Rightarrow^{\mathbf{6}} \\ \hspace{6.2cm}
\left.\begin{array}{l}\overline{\hspace{0.28cm} {\cal M}_D \circ
\widehat{\sigma}^* \, {\rm is} \; {\rm a.c.} \hspace{2.2cm}}\\
{}^{\mathbf{7}} \, {\cal M}_D \circ \widehat{\sigma}^* \,
{\rm has} \; {\rm zero \; derivative} \hspace{2.7cm}
\end{array}\right\} \Rightarrow^{\mathbf{8}} \end{array}$$
\newline
$$
\begin{array}{l} \left.\begin{array}{l} \Rightarrow^{\mathbf{6}} \,({\rm A. \, 15})
\, {\cal M}_D(\widehat{\sigma}^*(t))={\cal M}(\widehat{\sigma}^*(t)) \; \forall \; t\in [a,b] \\
\Rightarrow^{\mathbf{8}} \, ({\rm A. \,16}) \, {\cal
M}_D(\widehat{\sigma}^*(t)) \, {\rm is} \; {\rm constant} \; \forall
\; t\in [a,b]
\end{array}\right\}\Rightarrow^{\mathbf{9}} {\cal M} (\widehat{\sigma}^* (t))
\, {\rm is} \; {\rm constant} \; \forall \; t\in [a,b]
\end{array}
$$

\textbf{1.} $H^{\widetilde{u}} \in {\cal
C}^1(T^*\widehat{Q}) \Rightarrow H^{\widetilde{u}}$ is
locally Lipschitz $\forall \, \widetilde{u}\in D$.

The Lipschitzian property applies to functions defined on a
metric space. As the property we want to prove is local, we
define the distance on a local chart as is explained in
Appendix \ref{analytic}. For every $\beta \in
T^*\widehat{Q}$, let $(V_{\beta},\phi)$ be a local chart
centered at $\beta$ such that $\phi(\beta)=0$ and
$\phi(V_{\beta})=B$, where $B$ is an open ball centered at
$0\in \mathbb{R}^{2m+2}$. If $\beta_1$ and $\beta_2$ are in
$V_{\beta}$, define
$d_{\phi}(\beta_1,\beta_2)=d(\phi(\beta_1),\phi(\beta_2))$
where $d$ is the Euclidean distance in $\mathbb{R}^{2m+2}$.

For every $\beta$ in $T^*\widehat{Q}$, we get an open
neighbourhood $V_{\beta}$ using the local chart
$(V_{\beta},\phi)$. As $H^{\widetilde{u}}$ is ${\cal
C}^1(T^*\widehat{Q})$ and $\widetilde{u}$ lies in the
compact set $D$, by the Mean Value Theorem for every
$\beta$ in $T^*\widehat{Q}$ there exists an open
neighbourhood $V_{\beta}$ such that
$|H^{\widetilde{u}}(\beta_1)-H^{\widetilde{u}}(\beta_2)|<
K_{\beta}d_{\phi}(\beta_1,\beta_2)$ where $K_{\beta}$ does
not depend on the control $\widetilde{u}$. Thus
$H^{\widetilde{u}}$ is locally Lipschitz on
$T^*\widehat{Q}$. Moreover, the Lipschitz constant and the
open neighbourhood $V_{\beta}$ do not depend on the control
since $\widetilde{u}$ is in a compact set.

\textbf{2.} $H^{\widetilde{u}}$ is locally Lipschitz
$\forall \, \widetilde{u}\in D$ $\Rightarrow {\cal M}_D $
is locally Lipschitz on ${\rm Im} \, \widehat{\sigma}^*$.

Let $\beta$ be in ${\rm Im} \, \widehat{\sigma}^*$, there exists an
open convex neighbourhood $V_{\beta}$ such that
\[|H^{\widetilde{u}}(\beta_1)-H^{\widetilde{u}}(\beta_2)|<K_{\beta}
d(\beta_1,\beta_2)\] for every $\widetilde{u}$ in $D$ and $\beta_1$,
$\beta_2$ in $V_{\beta}$. If $\widetilde{w}_1$, $\widetilde{w}_2$
are the controls in $D$ maximizing $H(\beta_1, \widetilde{u})$ and
$H(\beta_2,\widetilde{u})$, respectively, then
\begin{equation*}
H(\beta_1,\widetilde{w}_2)\leq H(\beta_1,\widetilde{w}_1),
\end{equation*}
\begin{equation*}
H(\beta_2,\widetilde{w}_1) \leq H(\beta_2,\widetilde{w}_2).
\end{equation*}
Moreover, $H^{\widetilde{w}_1}$ and $H^{\widetilde{w}_2}$
are Lipschitz on $V_{\beta}$ since the Lipschitz constant
and the neighbourhood is independent of the control. Then
using the last inequalities
\begin{eqnarray*}
-K_{\beta}d(\beta_1,\beta_2)& \leq &
H^{\widetilde{w}_2}(\beta_1)-H^{\widetilde{w}_2}(\beta_2) \leq
H^{\widetilde{w}_1}(\beta_1)-H^{\widetilde{w}_2}(\beta_2) \\ &\leq &
H^{\widetilde{w}_1}(\beta_1)-H^{\widetilde{w}_1}(\beta_2)\leq
K_{\beta}d(\beta_1,\beta_2).
\end{eqnarray*} Observe that by Equation
(\ref{Hmaximum}),
$H^{\widetilde{w}_1}(\beta_1)-H^{\widetilde{w}_2}(\beta_2)= {\cal
M}_D(\beta_1)-{\cal M}_D(\beta_2)$. Hence
\begin{equation}\label{mlocLips}
\left|{\cal M}_D(\beta_1)-{\cal M}_D(\beta_2) \right|\leq
K_{\beta}d(\beta_1,\beta_2), \quad \forall \; \beta_1, \beta_2 \in
V_{\beta};
\end{equation}
that is, ${\cal M}_D$ is locally Lipschitz on $ {\rm Im} \,
\widehat{\sigma}^*$. As $\widehat{\sigma}^*$ is absolutely
continuous,  $ {\rm Im} \, \widehat{\sigma}^*$ is compact. Thus we
may choose a Lipschitz constant independent of the point $\beta$.
Hence
$$\left|{\cal
M}_D(\beta_1)-{\cal M}_D(\beta_2) \right|\leq K d(\beta_1,\beta_2),
\quad \forall \; \beta_1, \beta_2 \in V_{\beta}.$$

\textbf{3.} ${\cal M}_D$ is locally Lipschitz on ${\rm Im}
\, \widehat{\sigma}^*$ and $\widehat{\sigma}^*$ is
absolutely continuous $\Rightarrow {\cal M}_D\circ
\widehat{\sigma}^*\colon I\rightarrow \mathbb{R}$ is
absolutely continuous $\Rightarrow {\cal M}_D\circ
\widehat{\sigma}^*\colon I\rightarrow \mathbb{R}$ is
continuous.

For every $t\in I$, let us consider the neighbourhood
$V_{\widehat{\sigma}^*(t)}$ where Equation (\ref{mlocLips})
is satisfied. As ${\rm Im} \, \widehat{\sigma}^*$ is a
compact set, \vspace{-0.3cm}
\begin{itemize}
\item there exists a finite open subcovering
$V_{\widehat{\sigma}^*(t_1)},\ldots,V_{\widehat{\sigma}^*(t_r)}$
of $\{V_{\widehat{\sigma}^*(t)} \ ; \ t\in I\}$, and
\item there exists a
Lebesgue number $l$ of the subcovering; that is, for every two
points in an open ball of diameter $l$ there exists an open set of
the finite subcovering containing both points.
\end{itemize}
\vspace{-0.3cm} For the Lebesgue number $l$, by the uniform
continuity of $\widehat{\sigma}^*$, there exists a
$\delta_l>0$ such that for each $t_1$, $t_2$ in $I$ with
$|t_2-t_1|<\delta_l$, then
$d(\widehat{\sigma}^*(t_2),\widehat{\sigma}^*(t_1))<l$.
Thus there exists an open set of the finite subcovering
containing $\widehat{\sigma}^*(t_1)$ and
$\widehat{\sigma}^*(t_2)$.

On the other hand, taken $\epsilon>0$ the absolutely
continuity of $\widehat{\sigma}^*$ determines a
$\delta_{\epsilon}>0$.

To prove the absolute continuity of ${\cal M}_D\circ
\widehat{\sigma}^*$, take $\delta={\rm min}
\{\delta_l,\delta_{\epsilon}\}$. Then, for every finite
number of nonoverlapping subintervals $(t_{i_1},t_{i_2})$
of $I$ such that $\sum^n_{i=1}|t_{i_2}-t_{i_1}|<\delta$,
\begin{eqnarray*}\sum_{i=1}^n \left|{\cal
M}_D(\widehat{\sigma}^*(t_{i_2}))-{\cal
M}_D(\widehat{\sigma}^*(t_{i_1})) \right| \leq \sum_{i=1}^n
K
d(\widehat{\sigma}^*(t_{i_2}),\widehat{\sigma}^*(t_{i_1}))
\leq K \epsilon.
\end{eqnarray*}
In the first step we use that $\delta<\delta_l$ to
guarantee that $\widehat{\sigma}^*(t_{i_2})$ and
$\widehat{\sigma}^*(t_{i_1})$ are contained in the same
open set of the finite subcovering of ${\rm Im}\,
\widehat{\sigma}^*$. That allows us to use the property of
being locally Lipschitzian. Secondly, we use that
$\delta<\delta_{\epsilon}$ to apply the absolute continuity
of $\widehat{\sigma}^*$.

As ${\cal M}_D\circ \widehat{\sigma}^*$  is absolutely continuous on
$I$, ${\cal M}_D \circ \widehat{\sigma}^*$ is continuous on $I$.

\textbf{4.} $ {\cal M}_D(\widehat{\sigma}^*(t)) \leq {\cal
M}(\widehat{\sigma}^*(t))$ everywhere.

Observe that
$$ {\cal
M}_D(\widehat{\sigma}^*(t))=\sup_{u \in D}
H(\widehat{\sigma}^*(t),u) \leq \sup_{u \in U}
H(\widehat{\sigma}^*(t),u)={\cal
M}(\widehat{\sigma}^*(t)),$$ for each $t\in I$.\\

\textbf{5.} ${\cal M}(\widehat{\sigma}^*(t))={\cal
M}_D(\widehat{\sigma}^*(t))$ almost everywhere.

For each $t\in I$ there exists a control $w(t)$ maximizing
$H(\widehat{\sigma}^*(t),u)$ over the controls in $D$ because of
condition $(3a)$,
$${\cal
M}_D(\widehat{\sigma}^*(t))= \sup_{u \in D} H(\widehat{\sigma}^*(t),
u)=H(\widehat{\sigma}^*(t), w(t)).$$ As $u^*(t) \in D$ for each
$t\in I$,
$$\sup_{u \in D} H(\widehat{\sigma}^*(t),
u)=\sup_{u \in U} H(\widehat{\sigma}^*(t),u)={\cal
M}(\widehat{\sigma}^*(t))=H(\widehat{\sigma}^*(t), u^*(t))$$ almost
everywhere by (3a). Thus ${\cal M}(\widehat{\sigma}^*(t))={\cal
M}_D(\widehat{\sigma}^*(t))$ a.e..

\textbf{6.} Applying Proposition \ref{lowerew}, we have
${\cal M}_D(\widehat{\sigma}^*(t))={\cal
M}(\widehat{\sigma}^*(t))$ everywhere on $I$, because
${\cal M}_D \circ \widehat{\sigma}^*$ is continuous on $I$,
${\cal M}\circ \widehat{\sigma}^*$ is lower semicontinuous,
${\cal M}_D(\widehat{\sigma}^*(t))\leq{\cal
M}(\widehat{\sigma}^*(t))$ everywhere and ${\cal
M}_D(\widehat{\sigma}^*(t))={\cal
M}(\widehat{\sigma}^*(t))$ almost everywhere.

\textbf{7.} ${\cal M}_D \circ \widehat{\sigma}^*$ has zero
derivative.

As ${\cal M}_D \circ \widehat{\sigma}^*$ is absolutely
continuous on $I$, by Corollary \ref{aederivative} it has a
derivative almost everywhere. As the intersection of two
sets of full measure is not empty, see Appendix
\ref{analytic}, there exists a $t_0 \in I$ such that ${\cal
M}_D \circ \widehat{\sigma}^*$ is derivable at $t_0$ and
${\cal
M}_D(\widehat{\sigma}^*(t_0))=H(\widehat{\sigma}^*(t_0),u^*(t_0))
$. For each $t \neq t_0$, by the definition of ${\cal
M}_D$, we have
$${\cal
M}_D(\widehat{\sigma}^*(t))=\sup_{u \in D} H(\widehat{\sigma}^*(t),
u)\geq H(\widehat{\sigma}^*(t),u^*(t_0))
$$ because $u^*(t_0) \in D$. Thus
$ {\cal M}_D(\widehat{\sigma}^*(t))-{\cal
M}_D(\widehat{\sigma}^*(t_0)) \geq
H(\widehat{\sigma}^*(t),u^*(t_0))-H(\widehat{\sigma}^*(t_0),u^*(t_0))$.

If $t-t_0 >0$,
\begin{equation*}\frac{{\cal
M}_D(\widehat{\sigma}^*(t))-{\cal
M}_D(\widehat{\sigma}^*(t_0))}{t-t_0} \geq
\frac{H(\widehat{\sigma}^*(t),u^*(t_0))-H(\widehat{\sigma}^*(t_0),u^*(t_0))}
{t-t_0}.
\end{equation*}
Let us compute the right derivative of ${\cal M}_D \circ
\widehat{\sigma}^*$ at $t_0$ \begin{eqnarray*}
\left.\frac{d ({\cal M}_D \circ
\widehat{\sigma}^*)}{dt}\right|_{t=t_0^+}&=&\lim_{t\rightarrow
t_0^+}\frac{{\cal M}_D(\widehat{\sigma}^*(t))-{\cal
M}_D(\widehat{\sigma}^*(t_0))}{t-t_0}\geq\lim_{t\rightarrow
t_0^+}\frac{H^{u^*(t_0)}(\widehat{\sigma}^*(t))-H^{u^*(t_0)}(\widehat{\sigma}^*(t_0))}{t-t_0}\\
&=&{\rm
L}_{\widehat{X}^{T^{*\{u^*(t_0)\}}}_{\widehat{\sigma}^*(t_0)}}
H^{u^*(t_0)}=0
\end{eqnarray*}
since 
$i\left(\widehat{X}^{T^{*\{u^*(t_0)\}}}_{\widehat{\sigma}^*(t_0)}\right)
\omega
=\left(dH^{u^*(t_0)}\right)_{\widehat{\sigma}^*(t_0)}$.

Similarly, if $t-t_0<0$,
$$\left.\frac{d ({\cal
M}_D \circ \widehat{\sigma}^*)}{dt}\right|_{t=t_0^-} \leq 0.$$ Hence
the derivative of ${\cal M}_D \circ \widehat{\sigma}^*$ is zero
almost everywhere.

\textbf{8.} Applying Theorem \ref{constew}, ${\cal
M}_D\circ \widehat{\sigma}^*$ is constant everywhere,
because ${\cal M}_D\circ \widehat{\sigma}^*$ is absolutely
continuous.

\textbf{9.} As ${\cal M}_D(\widehat{\sigma}^*(t))$ and
${\cal M}(\widehat{\sigma}^*(t))$ coincide everywhere,
${\cal M}\circ \widehat{\sigma}^*$ is constant everywhere
on
$I$.\\


\textbf{(3c)} $\widehat{\sigma}^*(t)\neq 0 \in
T^{\ast}_{\widehat{\gamma}^*(t)} \widehat{Q}$ for each $t\in [a,b]$.

Let us suppose that there exists $\tau \in [a,b]$ such that
$\widehat{\sigma}^*(\tau)=0\in
T^*_{\widehat{\gamma}^*(\tau)}\widehat{Q}$. As $\widehat{\sigma}^*$
is a generalized integral curve of $(\widehat{X}^{T^*})^{\{u^*\}}$,
a linear vector field over $\widehat{X}$, we have
$\widehat{\sigma}^*(t)=0$ for each $t\in [a,b]$. As there exists at
least a time such that $\widehat{\sigma}^*(\tau)\neq 0$, we arrive
at a contradiction. Hence $\widehat{\sigma}^*(t)\neq 0$ for each
$t\in [a,b]$.

\textbf{(3d)} $\sigma_0^*(t)$ is constant, $\sigma_0^*(t)\leq 0$.

From the equations satisfied by the generalized integral curves of
$(\widehat{X}^{T^*})^{\{u^*\}}$, we have $p_0$ is constant. It was
seen that $\langle \widehat{\sigma}_b,(-1,\mathbf{0}) \rangle \geq
0$ is equivalent to $(p_0\circ
\widehat{\sigma}^*)(b)=\sigma_0(b)\leq 0$. Hence $\sigma_0\leq 0$
for each $t\in [a,b]$.
\end{proof}

\comment As $\widehat{\sigma}_b$ is determined up to multiply by a
positive real number, we may assume that $\sigma_0 \in \{-1, 0 \}$.

The way in which perturbations have been used in this proof gives
some clue concerning the fact that the tangent perturbation cone is
understood as an approximation of the reachable set defined in
Appendix \ref{approxreachable}. A precise meaning of this
approximation is explained in Appendix \ref{approxreachable}.

The covector in the proof has been chosen such that
\begin{eqnarray*} &&\langle \widehat{\sigma}_b,
(-1,\mathbf{0}) \rangle \geq 0, \\  &&\langle \widehat{\sigma}_b,
\widehat{v}_b \rangle \leq 0 \quad \forall \, \widehat{v}_b\in
\widehat{K}_b.
\end{eqnarray*}
In the abnormal case $\sigma_0=0$ and the first inequality is
satisfied with equality. Thus the covector is contained in the
separating hyperplane. It would be interesting to determine
geometrically what else must happen in order to have abnormal
minimizers.

\section{Pontryagin's Maximum Principle for nonfixed time and nonfixed
endpoints}\label{proofFPMP}

Now, that Pontryagin's Maximum Principle has been proved for time
and endpoints fixed, let us state the different problems related to
Pontryagin's Maximum Principle with nonfixed time and nonfixed
endpoints.

\subsection{Statement of the optimal control
problem with time and endpoints nonfixed} We consider the elements
$Q$, $U$, $X$, $F$, ${\cal S}$ and $\pi_2$ with the same properties
as in \S \ref{general}, \S \ref{PCOfijo}. Let $S_a$ and $S_f$ be
submanifolds of $Q$.

\begin{state} \textbf{(Free Optimal Control Problem,
$FOCP$)} Given the elements $Q$, $U$, $X$, $F$, and the
disjoint submanifolds of $Q$, $S_a$ and $S_f$, consider the
following problem.

Find $b\in \mathbb{R}$ and $(\gamma^*,u^*)\colon
[a,b]\rightarrow Q\times U$ such that
\begin{itemize}
\item[(1)] endpoint conditions: $\gamma^*(a) \in S_a$, $\gamma^*(b) \in S_f$,
\item[(2)] $\gamma^*$ is an integral curve of $X^{\{u^*\}}$: $\dot{\gamma^*}=X^{\{u^*\}} \circ (\gamma^*, {\rm id})$,
and
\item[(3)] minimal condition: ${\cal S}[\gamma^*,u^*]=\int^b_aF(\gamma^*(t),u^*(t))dt$ is minimum over all curves $(\gamma,u)$
satisfying $(1)$ and $(2)$.
\end{itemize}
\end{state}

The tuple $(Q,U,X,F,S_a,S_f)$ denotes the \textit{free optimal
control problem}.

\begin{state}\label{stateEFOCP} \textbf{(Extended Free Optimal Control Problem, $\mathbf{\widehat{FOCP}}$)}
Given the $FOCP$, $(Q,U,X,F,S_a,S_f)$, and the elements
$\widehat{Q}$ and $\widehat{X}$ defined in \S \ref{PMPEfijo},
consider the following problem.

Find $b\in \mathbb{R}$ and $(\widehat{\gamma}^*,u^*)\colon
[a,b] \rightarrow \widehat{Q}\times U$, with
$\gamma^*=\pi_2 \circ \widehat{\gamma}^*$, such that

\begin{itemize}
\item[(1)] endpoint conditions: $\widehat{\gamma}^*(a)\in \{0\} \times S_a$, $\gamma^*(b) \in
S_f$,
\item[(2)] $\widehat{\gamma}^*$ is an integral curve of $\widehat{X}^{\{u^*\}}$:
$\dot{\widehat{\gamma}^*}=\widehat{X}^{\{u^*\}} \circ (\widehat{\gamma}^*, {\rm id})
$, and
\item[(3)] minimal condition: $\gamma^{*^0}(b)$ is minimum over all curves $(\widehat{\gamma},u)$
satisfying $(1)$ and $(2)$.
\end{itemize}
\end{state}

The tuple $(\widehat{Q},U,\widehat{X},S_a,S_f)$ denotes the
\textit{extended free optimal control problem}.
\subsection{Perturbation of the time and the
endpoints}\label{perturbedconesnonfixed} In this case of nonfixed
time and nonfixed endpoint optimal control problems, we not only
modify the control as explained in \S \ref{perturbedconesfixed}, but
also modify the final time and the endpoint conditions. As was
mentioned in \S \ref{perturbedconesfixed}, the following
constructions obtained from perturbing the final time and the
endpoint conditions are also general for any vector field depending
on parameters.
\subsubsection{Time perturbation vectors and associated
cones}\label{applied}

We study how to perturb the interval of definition of the
control taking advantage of the fact that the final time is
another unknown for the free optimal control problems.

Let $X$ be a vector field on $M$ along the projection $\pi\colon
M\times U \rightarrow M$, $I\subset \mathbb{R}$ be a closed interval
and $(\gamma,u)\colon I=[a,b]\rightarrow M\times U$ a curve such
that $\gamma$ is an integral curve of $X^{\{u\}}$.

Let $\pi_{\pm}=\{\tau,l_{\tau},\delta \tau,u_{\tau}\}$, where $\tau$
is a Lebesgue time in $(a,b)$ for $X\circ (\gamma,u)$, $l_{\tau}\in
\mathbb{R}^+ \cup \{0\}$, $\delta \tau \in \mathbb{R}$, $u_{\tau}\in
U$. For every $s\in \mathbb{R}^+$ small enough such that
$a<\tau-(l_{\tau}-\delta \tau) s$, consider $u[\pi_{\pm}^s]\colon
[a,b+\delta \tau s] \rightarrow U$ defined by
$$u[\pi_{\pm}^s](t) = \left\{\begin{array}{ll} u(t), & t\in [a,\tau-(l_{\tau}-\delta \tau) s], \\
u_{\tau}, & t\in
\left(\right.\tau-(l_{\tau}-\delta \tau) s, \tau+\delta \tau s\left.\right], \\
u(t), & t\in \left(\right.\tau+\delta \tau s,b+\delta \tau
s\left.\right], \end{array} \right. $$ if $\delta \tau <0$,
and by
$$u[\pi_{\pm}^s](t) = \left\{\begin{array}{ll}  u(t), & t\in [a, \tau-(l_{\tau}-\delta \tau)s], \\
u_{\tau}, & t\in
\left(\right.\tau-(l_{\tau}-\delta \tau) s, \tau+\delta \tau s\left.\right], \\
u(t-\delta \tau s), & t\in \left(\right.\tau+\delta \tau s,
b+\delta \tau s\left.\right], \end{array} \right. $$ if
$\delta \tau\geq 0$.
\begin{defin} The function $u[\pi_{\pm}^s]$ is
called a \textbf{perturbation of $u$ specified by the data
$\pi_{\pm}=\{\tau,l_{\tau},\delta \tau,u_{\tau}\}$}.
\end{defin}

Associated to $u[\pi_{\pm}^s]$ we consider the mapping
$\gamma [\pi_{\pm}^s] \colon [a,b+\delta \tau s]
\rightarrow M$, the gene\-ra\-li\-zed integral curve of
$X^{\{u[\pi_{\pm}^s]\}}$ with initial condition
$(a,\gamma(a))$.

Given $\epsilon>0$, define
$$\begin{array}{rcl} \varphi_{\pi_{\pm}} \colon
[\tau,b]\times [0,\epsilon] &\longrightarrow & M\\
(t,s)&\longmapsto & \varphi_{\pi_{\pm}}(t,s)=
\gamma[\pi_{\pm}^s](t+\delta \tau s)
\end{array}$$

For every $t\in [\tau,b] $, $\varphi_{\pi_{\pm}}^{t}\colon
[0,\epsilon] \rightarrow M$ is given by
$\varphi_{\pi_{\pm}}^{t}(s)= \varphi_{\pi_{\pm}}(t,s)$.

As explained in \S \ref{perturbedconesfixed}, the control
$u[\pi_{\pm}^s]$ depends continuously on the parameters $s$ and
$\pi_{\pm}=\{\tau,l_{\tau},\delta \tau,u_{\tau}\}$. Hence the curve
$\varphi_{\pi_{\pm}}^t$ depends continuously on $s$ and
$\pi_{\pm}=\{\tau,l_{\tau},\delta \tau,u_{\tau}\}$, then it
converges uniformly to $\gamma$ as $s$ tends to $0$. See
\cite{2004Canizo,55Coddington} for more details of the differential
equations depending continuously on parameters.

Let us prove that the curve $\varphi_{\pi_{\pm}}^{\tau}$ has a
tangent vector at $s=0$; cf. Proposition \ref{tangentperturb}.
\begin{prop}\label{tangentperturbtime} Let $\tau$ be a Lebesgue time. If $u[\pi_{\pm}^s]$
is the perturbation of the control $u$ specified by the data
$\pi_{\pm}=\{\tau,l_{\tau},\delta\tau,u_{\tau}\}$ such that $\tau+s
\delta\tau$ is a Lebesgue time, then the curve
$\varphi_{\pi_{\pm}}^{\tau}\colon [0,\epsilon]\rightarrow M$ is
differentiable at $s=0$. Its tangent vector is
\[X(\gamma(\tau),u(\tau))\, \delta \tau + \left[X(\gamma(\tau),
u_{\tau})-X(\gamma(\tau), u(\tau))\right]\, l_{\tau}.\]
\end{prop}
\begin{proof}
As in the proof of Proposition \ref{tangentperturb}, we compute the
limit
$$A=\lim_{s \rightarrow 0} \frac{(x^i \circ
\varphi_{\pi_{\pm}}^{\tau})(s)- (x^i \circ
\varphi_{\pi_{\pm}}^{\tau})(0)}{s} =\lim_{s \rightarrow 0}
\frac{\gamma^i[\pi_{\pm}^s](\tau+\delta \tau s)-\gamma^i(\tau)}{s}$$
As $\gamma$ is an absolutely continuous integral curve of
$X^{\{u\}}$ , $\dot{\gamma}(t)=X(\gamma(t),u(t))$ at every Lebesgue
time. Then by integration
$$\gamma^i(\tau)-\gamma^i(a)=\int^{\tau}_af^i(\gamma(t),u(t))dt$$
and similarly for $\gamma[\pi_{\pm}^s]$ and
$u[\pi_{\pm}^s]$. Observe that
$\gamma[\pi_{\pm}^s](t)=\gamma(t)$ and
$u[\pi_{\pm}^s](t)=u(t)$ for $t\in[a,\tau -(l_{\tau}-\delta
\tau) s]$.

 Here, we should consider three different
possibilities
\begin{itemize}
\item if $0\leq \delta \tau \leq l_{\tau}$, then $\tau-(l_{\tau}-\delta \tau)s<\tau<\tau+\delta
\tau s$;
\item if $\delta \tau < 0$, then $ \tau-(l_{\tau}-\delta \tau)s<\tau+\delta \tau
s<\tau$;
\item if $0<l_{\tau}<\delta \tau$, then
$\tau < \tau -(l_{\tau}- \delta \tau )s < \tau + \delta
\tau s$.
\end{itemize}
We prove the proposition for the first case and the other cases
follow analogously. \begin{eqnarray*}A&=&\lim_{s \rightarrow 0}
\frac{\int^{\tau+\delta \tau s}_{a}f^i(\gamma[\pi_{\pm}^s](t),
u[\pi_{\pm}^s](t))dt-\int^{\tau}_{a}f^i(\gamma(t), u(t))dt}{s}\\
&=& \lim_{s \rightarrow 0} \frac{\int^{\tau+\delta \tau
s}_{\tau-(l_{\tau}-\delta \tau) s}f^i(\gamma[\pi_{\pm}^s](t),
u_{\tau})dt-\int^{\tau}_{\tau-(l_{\tau}-\delta \tau)
s}f^i(\gamma(t), u(t))dt}{s}.\end{eqnarray*} We need $\tau+\delta
\tau s$ to be a Lebesgue time in order to use Equation
(\ref{eqLebesgue}).
\begin{eqnarray*}
A&=&\lim_{s \rightarrow 0} \frac{f^i(\gamma[\pi_{\pm}^s](\tau+\delta
\tau s), u_{\tau})l_{\tau}s -f^i(\gamma(\tau), u(\tau))
(l_{\tau}-\delta \tau) s +o(s)}{s}\\&=& \lim_{s \rightarrow 0}
f^i(\gamma[\pi_{\pm}^s](\tau+\delta \tau s), u_{\tau})l_{\tau}
-f^i(\gamma(\tau), u(\tau)) (l_{\tau}-\delta \tau).\end{eqnarray*}
As $f^i$ is continuous on $M$, we have
\begin{eqnarray*}
A&=&[f^i(\gamma(\tau), u_{\tau})-f^i(\gamma(\tau), u(\tau))]\,
l_{\tau}+ f^i(\gamma(\tau), u(\tau)) \ \delta \tau\\ & =&{\rm
L}(\left[X(\gamma(\tau),u(\tau)) \ \delta
\tau+\left(X(\gamma(\tau),u_{\tau})-X(\gamma(\tau),u(\tau)) \right)
\ l_{\tau}\right])(x^i).\end{eqnarray*}
\end{proof}

\begin{defin}\label{classItime} The tangent vector \[v[\pi_{\pm}]=X(\gamma(\tau),u(\tau)) \
\delta \tau+\left[X(\gamma(\tau),u_{\tau})-X(\gamma(\tau),u(\tau))
\right] \ l_{\tau}\] is the \textbf{perturbation vector associated
to the perturbation data $\pi_{\pm}=\{\tau,l_{\tau},\delta
\tau,u_{\tau}\}$}.
\end{defin}

If we disturb the control $r$ times at $r$ different Lebesgue times
as in \S \ref{elementarypertvectors} and also the domain of the
curve $(\gamma,u)$ as just described, that is,
$\pi=\{\pi_1,\ldots,\pi_r,\pi_{\pm}\}$, with $a<t_1\leq \ldots \leq
t_r \leq \tau <b$, then $\gamma[\pi^s]$ is the gene\-ra\-lized
integral curve of $X^{\{u[\pi^s]\}}$ with initial condition
$(a,\gamma(a))$. Consider the curve $ \varphi_{\pi}^t\colon
[0,\epsilon]\rightarrow M$ for $t\in [\tau,b]$ given by $
\varphi_{\pi}^t(s)=\gamma[\pi^s](t+\delta \tau s)$.
\begin{corol}
The vector tangent to the curve $\varphi_{\pi_{\pm}}^t
\colon [0,\epsilon]\rightarrow M$ at $s=0$ is
$X(\gamma(t),u(t)) \ \delta
\tau+V[\pi_1](t)+\ldots+V[\pi_n](t)$, where $V[\pi_i]\colon
[t_i,b]\rightarrow TM$ is the generalized integral curve of
$\left(X^T\right)^{\{u\}}$ with initial condition
$(t_i,(\gamma(t_i),v[\pi_i]))$.
\end{corol}
This corollary may be proved taking into account Propositions
\ref{tangentperturb2} and \ref{tangentperturbtime}, Corollary
\ref{tangentperturbn} and Appendix \ref{variationalsection}.

Now, at a Lebesgue time $t\in(a,b)$, we construct a new cone that
contains the perturbation vectors in Definition \ref{tangent} and
$\pm X(\gamma(t), u(t))$.

\begin{defin} \label{time} The \textbf{time perturbation cone}
$K_t^{\pm}$ at every Lebesgue time $t$ is the smallest closed cone
in the tangent space at $\gamma(t)$ containing $K_t$ and $\pm
X(\gamma(t), u(t))$,
$$K_t^{\pm}=\overline{{\rm conv}\left(\{\pm \alpha X(\gamma(t), u(t)) \, | \, \alpha \in \mathbb{R}\} \bigcup
\left( \bigcup_{\substack{a<\tau\leq t \\ \tau \textrm{ is a
Lebesgue time}}} \left(\Phi^{X^{\{u\}}}_{(t,\tau)}\right)_*{\cal
V}_{\tau}\right)\right)},
$$
where ${\cal V}_{\tau}$ denotes the set of elementary perturbation
vectors at $\tau$, see Definition \ref{tangent}.
\end{defin}
Enlarging the cone $K_{\tau}$ to $K_{\tau}^{\pm}$ allows us to
introduce time variations.
\begin{prop}\label{inclusion} If $t_2$ is a Lebesgue time greater than
$t_1$, then \[\left(\Phi^{X^{\{u\}}}_{(t_2,t_1)}\right)_*
K^{\pm}_{t_1} \subset K^{\pm}_{t_2}.\]
\end{prop}
\begin{proof}
We have $$ K_{t_1}^{\pm}=\overline{{\rm conv}\left(\{\pm \alpha
X(\gamma(t_1), u(t_1)) \, | \, \alpha \in
\mathbb{R}\} \bigcup \left( \bigcup_{\substack{a<\tau\leq t_1 \\
\tau \textrm{ is a Lebesgue time}}}
\left(\Phi^{X^{\{u\}}}_{(t_1,\tau)}\right)_*{\cal
V}_{\tau}\right)\right)}.
$$
Just for simplicity we use ${\cal C}^{\pm}_{t_1}$ to denote
$${\rm conv}\left(\{\pm
\alpha X(\gamma(t_1), u(t_1)) \, | \, \alpha \in
\mathbb{R}\} \bigcup \left( \bigcup_{\substack{a<\tau\leq t_1 \\
\tau \textrm{ is a Lebesgue time}}}
\left(\Phi^{X^{\{u\}}}_{(t_1,\tau)}\right)_*{\cal
V}_{\tau}\right)\right).$$
\begin{itemize}
\item[1.] The set ${\cal C}^{\pm}_{t_1}$ being convex, if
$v$ is interior to $K^{\pm}_{t_1}$, then $v$ is interior to ${\cal
C}_{t_1}^{\pm}$ by Proposition \ref{convex}, item (d). Hence, by
Proposition \ref{convexhull}
$$v=\delta t_1 X(\gamma(t_1),u(t_1))+\sum_{i=1}^r l_iV[\pi_i](t_1),$$
where every $V[\pi_i](t_1)$ is the transported perturbation vector
$v[\pi_i]$ of class I  from $t_i$ to $t_1$ by the flow of
$X^{\{u\}}$. By definition of the cone and the linearity of the
flow, $\left(\Phi^{X^{\{u\}}}_{(t_2,t_1)}\right)_*v $ is in
$K^{\pm}_{t_2}$, since
$\left(\Phi^{X^{\{u\}}}_{(t_2,t_1)}\right)_*\left(X(\gamma(t_1),u(t_1))\right)=
X(\gamma(t_2),u(t_2))$, because both sides of the equality are the
unique solutions of the variational equation along $\gamma$
associated with $X^{\{u\}}$ with initial condition $(t_1,
X(\gamma(t_1),u(t_1))$. See Appendix \ref{completelift} for more
details.

\item[2.] If $v$ is in the boundary of $K^{\pm}_{t_1}$,
then there exists a sequence of vectors $(v_ j)_{j\in \mathbb{N}}$
in the interior of $K^{\pm}_{t_1}$ such that
$$\lim_{j\rightarrow \infty}v_j = v.$$ Due to the continuity of the flow
$$\lim_{j\rightarrow \infty}\left(\Phi^{X^{\{u\}}}_{(t_2,t_1)}\right)_*v_j
=\left(\Phi^{X^{\{u\}}}_{(t_2,t_1)}\right)_*v.$$ All the elements of
the convergent sequence are in
the closed cone $K^{\pm}_{t_2}$, hence the limit
$\left(\Phi^{X^{\{u\}}}_{(t_2,t_1)}\right)_*v$ is also in
$K^{\pm}_{t_2}$.

If the interior of $K^{\pm}_{t_1}$ is empty, we consider the
relative topology and the reasoning follows as before. See Appendix
\ref{hyperplane} for details.
\end{itemize}
\end{proof}

For the time perturbation cone $K^{\pm}_{\tau}$ and the
corresponding perturbation vectors, it can be proved properties
analogous to the ones stated in Propositions \ref{tangentperturb},
\ref{tangentperturb2}, \ref{12} and \ref{11}.

\begin{prop}\label{lema2time}  Let $t\in(a,b)$ be a Lebesgue time.
If $v$ is a nonzero vector interior to $K^{\pm}_t$, then there
exists $\epsilon>0$ such that for every $s\in(0,\epsilon)$ there are
$s'>0$ and a perturbation of the control $u[\pi^s_{\pm}]$ such that
$\gamma[\pi^s_{\pm}](t+s\delta t)=\gamma(t)+s'v$.
\end{prop}

\begin{proof} The proof follows the same line as the proof of Proposition
\ref{lemma2}, but now the tangent space to $M$ at $\gamma(t+s\delta
t)$ is also identified with $\mathbb{R}^m$ through the local chart
of $M$ at $\gamma(t)$.

We use the same functions as in the proof of Proposition
\ref{lemma2}, but keeping in mind that
$\Gamma(s,r)=\gamma[\pi^s_{w_0}](t)$ is replaced by
$\Gamma(s,r)=\gamma[\pi^s_{w_0}](t+s\delta t)$.
\end{proof}

\subsubsection{Perturbing the endpoint conditions}

Now we consider that the endpoint conditions for the integral curves
of $X^{\{u\}}$ varies on submanifolds of $M$. Let $S_a$ be a
submanifold of $M$ and $\gamma(a)$ in $S_a$; consider the integral
curve $\gamma\colon I\rightarrow M$ of $X^{\{u\}}$ with initial
condition $(a,\gamma(a))$.

We consider the curve $\gamma[\pi^s_{\pm}]$ obtained from a
time perturbation of the control $u$ associated with a
vector in the time perturbation cone. The initial condition
is disturbed along a curve $\delta\colon
[0,\epsilon]\rightarrow S_a$ with initial tangent vector
$v_a$ in $T_{\gamma(a)}S_a$ and $\delta(0)=\gamma(a)$.
Taking into account Appendix \ref{geometric meaning}, \S
\ref{elementarypertvectors} and considering that
$T_{\gamma(a)}S_a$ and an open set at $\delta(a)$ are
identified with $\mathbb{R}^m$, the integral curve
$\gamma_{\delta(s)}[\pi^s_{\pm}]\colon I \rightarrow M$ of
$X^{\{u[\pi^s_{\pm}]\}}$ with initial condition
$(a,\delta(s))$ can be written as
$$\gamma_{\delta(s)}[\pi^s_{\pm}](t)=\gamma(t)+s
\left(\Phi^{X^{\{u\}}}_{(t,a)}\right)_*v_a+s v[\pi_{\pm}](t)+o(s).$$
We define a cone that includes the time perturbation vectors, the
elementary perturbation vectors and the vectors coming from changing
the initial condition on $S_a$ along different curves $\delta\colon
[0,\epsilon]\rightarrow S_a$ through $\gamma(a)$ and contained in
$S_a$.

\begin{defin}\label{initialtimecone} Let $t$ be a Lebesgue time. The cone ${\cal K}_t$ is
the smallest closed and convex cone containing the time perturbation
cone at time $t$ and the transported of the tangent space to $S_a$
from $a$ to $t$ through the flow of $X^{\{u\}}$.
$${\cal K}_t= \overline{{\rm
conv}(K^{\pm}_{t} \bigcup
(\Phi^{X^{\{u\}}}_{(t,a)})_*(T_{\gamma(a)}S_a))}$$
\end{defin}

\begin{prop}\label{initialtimelemma2}  Let $t$ be a Lebesgue time in $(a,b)$ and $S\subset M$ be a
submanifold with boundary. Suppose that $\gamma(t)$ is on
the boundary of $S$. Let $T$ be the half\textendash{}plane
tangent to $S$ at $\gamma(t)$. If ${\cal K}_t$ and $T$ are
not separated, then there exists a perturbation of the
control $u[\pi^s_{\pm}]$ and $x_a\in S_a$ such that the
integral curve $\gamma_{x_a}[\pi^s_{\pm}]$ of
$X^{\{u[\pi^s_{\pm}]\}}$ with initial condition $(a,x_a)$
meets $S$ at a point in the relative interior of $S$.
\end{prop}

\begin{proof}
As ${\cal K}_t$ and $T$ are not separated, by Proposition
\ref{propseparated} there no exists any hyperplane
containing both and there is a vector $v$ in the relative
interior of both ${\cal K}_t$ and $T$. By Corollary
\ref{splitSussmann}, if ${\cal K}_t$ and $T$ are not
separated,
$$T_{\gamma(t)}M={\cal K}_t- T.$$
See Appendix \ref{hyperplane} for the notation and properties. If
$V$ is an open set of a local chart at $\gamma(t)$, we identify $V$
with $\mathbb{R}^m$ and also the tangent space at $\gamma(t)$,
$T_{\gamma(t)}M$, in the same sense defined for Equation
(\ref{linearapprox}). Let us consider an orthonormal basis in
$T_{\gamma(t)}M$, $\{e_1, \ldots, e_m\}$. If we take
$e_0=-(e_1+\ldots+e_m)$, the vector $0\in T_{\gamma(t)}M$ is
expressed as an affine combination of $e_0, e_1, \ldots, e_m$:
$$0=\frac{1}{m+1}e_0+\ldots+\frac{1}{m+1}e_m.$$
Each $w$ in $T_{\gamma(t)}M$ is written uniquely as
$$w=a^1e_1+\ldots + a^me_m$$ and as an affine combination
of $e_0, e_1, \ldots, e_m$:
$$w=\sum^m_{i=0}b^i(w)e_i=re_0+\sum^m_{i=0}(r+a^i)e_i\quad {\rm with} \quad
r=\frac{1-\sum^m_{i=1}a^i}{m+1}.$$ Hence, we define the continuous
mapping
$$\begin{array}{rcl} {\cal G}\colon  T_{\gamma(t)}M & \longrightarrow & \mathbb{R}^{m+1} \\
w&\longmapsto & (b^0(w),b^1(w),\ldots, b^m(w)).
\end{array}$$
As $b^i(0)>0$ for every $i=0,\ldots,m$, there exists an open ball
$B(0,r)$ centered at $0$ with radius $r$ such that for every $w\in
B(0,r)$, $b^i(w)>0$ for $i=0,\ldots,m$. Now we consider the
restriction of ${\cal G}$ to the closed ball $\overline{B(0,r)}$,
${\cal G}_{\overline{B(0,r)}}\colon \overline{B(0,r)} \rightarrow
[0,1]^{m+1}$.
Choose vectors $e_i^{\cal K}\in {\cal K}_t$ and $e_i^T\in T$ such
that
$$e_i=e_i^{\cal K}-e_i^T.$$ As $v$ lies in the
relative interior of both convex sets, $e_i=(e_i^{\cal
K}+v)-(e_i^T+v)=e_i^{\cal K}-e_i^T$. Then $e_i^{\cal K}$ and $e_i^T$
can be chosen in the relative interior of ${\cal K}$ and $T$,
respectively, because $v$ is in the relative interior of both. For
any $w\in \overline{B(0,r)}$,
$$w=\sum^m_{i=0}b^i(w)e_i=\sum^m_{i=0}b^i(w)\left(e_i^{\cal K}-e_i^T\right).$$
Then we can define
$$\begin{array}{rcl} {F_1}\colon  \overline{B(0,r)} & \longrightarrow & {\cal K}_t \\
w&\longmapsto & F_1(w)=\sum^m_{i=0}b^i(w)e_i^{\cal K},
\end{array} $$
$$\begin{array}{rcl} {F_2}\colon \overline{B(0,r)} & \longrightarrow & T \\
w&\longmapsto & F_2(w)=\sum^m_{i=0}b^i(w)e_i^T,
\end{array} $$
and let us consider the mapping
$$\begin{array}{rcl} G\colon \mathbb{R}\times \overline{B(0,r)} & \longrightarrow & \mathbb{R}^m\\
(s,w) & \longmapsto &
(\gamma[\pi^s_{F_1(w)}](t)-\gamma[\pi^s_{F_2(w)}](t))/s,
\end{array}$$
where $\gamma[\pi^s_{F_1(w)}]$ is the perturbation curve associated
to $\pi^s_{F_1(w)}$ and
$\gamma[\pi^s_{F_2(w)}](t)=\gamma(t)+sF_2(w)$ is the straight line
through $\gamma(t)$ with tangent vector $F_2(w)$.
As the perturbation vectors are in the relative interior of the
convex cones, we use the linear approximation in
(\ref{linearapprox}) in such a way that $G(s,w)=F_1(w)-F_2(w)+o(1)$.
Hence
$$\lim_{s\rightarrow  0}G(s,w)=F_1(w)-F_2(w)=w,$$
Hence, for any positive number $\epsilon$, there exists
$s_0>0$ such that if $s<s_0$ then $\|G(s,w)-w\|<\epsilon$.
Take $\epsilon<r$, then
$$\|G(s,w)-w\|<\epsilon < r=\|w\|$$
for every $w$ in the boundary of $\overline{B(0,r)}$. Thus the map
$G_s\colon \overline{B(0,r)}\rightarrow \mathbb{R}^m$,
$G_s(w)=G(s,w)$, satisfies the hypotheses of Corollary
\ref{scholium} for the point $0$ in $B(0,r)$. Hence, the point $0$
is in the image of $\overline{B(0,r)}$ through $G_s$ and there
exists $w$ such that $G_s(w)=0$; that is,
$$\gamma[\pi^s_{F_1(w)}](t)=\gamma[\pi^s_{F_2(w)}](t).$$ Therefore, there
exists a perturbation of the control such that the associated
trajectory meets $S$ in an interior point since $F_2(w)$ lies in the
relative interior of $T$.
\end{proof}

\subsection{Pontryagin's Maximum Principle with time and endpoints nonfixed}

Bearing in mind the symplectic formalism introduced in \S
\ref{symplectic}, we define the corresponding Hamiltonian
Problem when the time and the endpoints are nonfixed.

\begin{state}\textbf{(Free Hamiltonian Problem, $FHP$)}
Given the $FOCP$, $(Q,U,X,F,S_a,S_f)$, and the equivalent
$\widehat{FOCP}$, $(\widehat{Q},U,\widehat{X},S_a,S_f)$, consider
the following problem.

Find $b\in \mathbb{R}$
 and $(\widehat{\sigma},u)\colon [a,b]
\rightarrow T^*\widehat{Q}\times U$, with
$\widehat{\gamma}=\pi_{\widehat{Q}} \circ \widehat{\sigma}$ and
$\gamma=\pi_2 \circ \widehat{\gamma}$, such that
\begin{itemize}
\item[(1)] $\widehat{\gamma}(a) \in \{ 0 \} \times S_a
$, $\gamma(b)\in S_f$, and
\item[(2)] $\dot{\widehat{\sigma}}=(\widehat{X}^{T^*})^{\{u\}} \circ (\widehat{\sigma}, {\rm id})$.
\end{itemize}
\end{state}
The tuple $(T^*\widehat{Q}, U, \widehat{X}^{T^*}, S_a, S_f)$ denotes
the \textit{free Hamiltonian problem}.

\comments
\begin{enumerate}
\item The minimum of the interval of definition of the curves
is $a$, but the maximum is not fixed.
\item The curves $\gamma$,
$\widehat{\gamma}$ and $\widehat{\sigma}$ are assumed to be
absolutely continuous. So they are genera\-li\-zed integral
curves of $X^{\{u\}}$, $\widehat{X}^{\{u\}}$ and
$(\widehat{X}^{T^*})^{\{u\}}$, respectively, in the sense
defined in \S \ref{general}.
\end{enumerate}

Now, we are ready to state the Free Pontryagin's Maximum
Principle that provides the necessary conditions, but in
general not sufficient, for finding solutions of the free
optimal control problem.

\begin{teo} \textbf{(Free Pontryagin's Maximum
Principle, FPMP)}\label{PMfreevariable}
\\
If $(\widehat{\gamma}^*,u^*)\colon [a,b] \rightarrow \widehat{Q}
\times U$ is a solution of the extended free optimal control
problem, Statement \ref{stateEFOCP}, then there exists
$(\widehat{\sigma}^*, u^*)\colon [a,b] \rightarrow T^*\widehat{Q}
\times U$ such that:
\begin{enumerate}
\item it is a solution of the associated free Hamiltonian problem;
\item $\widehat{\gamma}^*=\pi_{\widehat{Q}} \circ \widehat{\sigma}^*
$;
\item \begin{itemize}
\item[(a)] $H(\widehat{\sigma}^*(t),u^*(t))= \sup_{u \in U} H(\widehat{\sigma}^*(t),
u)$ almost everywhere;
\item[(b)] $\sup_{u \in U} H(\widehat{\sigma}^*(t),
u)= 0$ everywhere;
\item[(c)] $\widehat{\sigma}^*(t)\neq 0 \in
T^*_{\widehat{\gamma}^*(t)} \widehat{Q}$ for each  $t \in [a,b]$;
\item[(d)] $\sigma^*_0(t)$ is constant, $\sigma^*_0(t) \leq 0$;
\item[(e)] transversality conditions:$\sigma^*(a) \in {\rm ann} \,
T_{\gamma^*(a)}S_a$ and $\sigma^*(b) \in {\rm ann} \,
T_{\gamma^*(b)}S_f$.
\end{itemize}
\end{enumerate}
\end{teo}

Observe that once we have the optimal solution of the
$\widehat{FOCP}$, the final time and the endpoint conditions are
known and fixed. We would like to apply just Theorem \ref{PMP} in
order to prove Theorem \ref{PMfreevariable}. However, this is not
possible because the freedom to chose the final time and the
endpoint conditions, only restricted to submanifolds, in Statement
\ref{stateEFOCP} is used in the proof to consider variations of the
optimal curve that are slightly different from the variations used
in the case of fixed time, see \S \ref{perturbedconesfixed} and \S
\ref{perturbedconesnonfixed} to compare them.

Apart from the transversality conditions, the main difference
between FPMP and PMP is the fact that the domain of the curves in
the optimal control problems is unknown. That introduces a new
necessary condition: the supremum of the Hamiltonian must be zero,
not just constant. Then, from $(3a)$ and $(3b)$ it may be concluded
that the Hamiltonian is zero almost everywhere. For instance, in the
time optimal problems the Hamiltonian along extremals must be zero.

There are different statements of Pontryagin's Maximum
Principle. In \S \ref{symplectic} we have considered the
statement of PMP for a fixed\textendash{}time problem
without transversality conditions to simplify the proof,
although it may be stated the PMP for the
fixed\textendash{}time problem with variable endpoints
where the transversality conditions appear. There also
exists the PMP for the free time problem with the
degenerate case that the submanifolds are only a point,
then the Theorem is the following one.
\begin{teo} \textbf{(Free Pontryagin's Maximum
Principle without variable endpoints)}
\\
If $(\widehat{\gamma}^*,u^*)\colon [a,b] \rightarrow \widehat{Q}
\times U$ is a solution of the extended free optimal control
problem, Statement \ref{stateEFOCP} with $S_a=\{x_a\}$ and
$S_f=\{x_f\}$, then there exists $(\widehat{\sigma}^*, u^*)\colon
[a,b] \rightarrow T^*\widehat{Q} \times U$ such that:
\begin{enumerate}
\item it is a solution of the associated free Hamiltonian problem;
\item $\widehat{\gamma}^*=\pi_{\widehat{Q}} \circ \widehat{\sigma}^*
$;
\item \begin{itemize}
\item[(a)] $H\left(\widehat{\sigma}^*(t),u^*(t)\right)= \sup_{u \in U} H\left(\widehat{\sigma}^*(t),
u\right)$ almost everywhere;
\item[(b)] $\sup_{u \in U} H\left(\widehat{\sigma}^*(t),
u\right)= 0$ everywhere;
\item[(c)] $\widehat{\sigma}^*(t)\neq 0 \in
T^*_{\widehat{\gamma}^*(t)} \widehat{Q}$ for each  $t \in [a,b]$;
\item[(d)] $\sigma^*_0(t)$ is constant, $\sigma^*_0(t) \leq 0$;
\end{itemize}
\end{enumerate}
\end{teo}
The only difference with Theorem \ref{PMfreevariable} is that the
transversality conditions do not appear.

\section{Proof of Pontryagin's Maximum Principle for
nonfixed time and nonfixed endpoints}\label{SproofFPMP}

In the proof of Theorem \ref{PMfreevariable} we use notions about
perturbations of the trajectories of a system introduced in \S
\ref{perturbedconesnonfixed}, but they are slightly different from
the perturbations in \S \ref{perturbedconesfixed} used to prove
Theorem \ref{PMP}.

\begin{proof}\textit{(Theorem \ref{PMfreevariable}: Free Pontryagin's Maximum Principle, FPMP)}

Given a solution of the $\widehat{FOCP}$, we only need an
appropriate initial condition in the fibers of
$\pi_{\widehat{Q}}\colon T^*\widehat{Q} \rightarrow
\widehat{Q}$ to find a solution of the $FHP$, because this
initial condition is not given in the hypotheses of the
Free Pontryagin's Maximum Principle. It is not possible to
use Theorem \ref{PMP} directly because the perturbation
cones are not the same. Indeed, we need to consider changes
in the interval of definition of the curves. These changes
imply the inclusion of
${\pm}\widehat{X}(\widehat{\gamma}^*(t_1), u^*(t_1))$ in
the perturbation cone at time $t_1$. All the times
considered in this proof are Lebesgue times for the vector
field giving the optimal curve.

By Proposition \ref{inclusion}, for $t_2>t_1$,
$$\left(\Phi^{\widehat{X}^{\{u^*\}}}_{(t_2,t_1)}\right)_*
\widehat{K}^{\pm}_{t_1} \subset \widehat{K}^{\pm}_{t_2}.$$

Let us consider the limit cone as follows
$$\widehat{K}^{\pm}_{b}=\overline{
\bigcup_{\substack{a<\tau\leq b \\ \tau \textrm{ is a Lebesgue
time}}} \left(\Phi^{\widehat{X}^{\{u\}}}_{(b,\tau)}\right)_*
\widehat{K}^{\pm}_{\tau}}.$$  Observe that it is a closed cone and
it is convex because it is the union of an increasing family of
convex cones. Let us show that
$(-1,\mathbf{0})_{\widehat{\gamma}^*(b)}$ is not interior to
$\widehat{K}^{\pm}_{b}$. Indeed, suppose that
$(-1,\mathbf{0})_{\widehat{\gamma}^*(b)}$ is interior to the limit
cone, then it will be interior to \[\bigcup_{\substack{a<\tau < b
\\ \tau \textrm{ is a Lebesgue time}}}
(\Phi^{\widehat{X}^{\{u\}}}_{(b,\tau)})_* \widehat{K}^{\pm}_{\tau}\]
by Proposition \ref{convex}, item (d). As we have an increasing
family of cones, there exists a time $\tau$ such that
$(-1,\mathbf{0})_{\widehat{\gamma}^*(b)}$ is interior to
$(\Phi^{\widehat{X}^{\{u\}}}_{(b,\tau)})_*
\widehat{K}^{\pm}_{\tau}$.
Let us see that this is not possible.

If $(-1,\mathbf{0})_{\widehat{\gamma}^*(b)}$ is interior to
$(\Phi^{\widehat{X}^{\{u\}}}_{(b,\tau)})_*
\widehat{K}^{\pm}_{\tau}$, then, by Proposition \ref{lema2time},
there exists $\epsilon >0$ such that, for every $s\in(0,\epsilon)$,
there exist $s'>0$ and a perturbation of the control
$u[\pi^s_{w_0}]$ such that
$$\widehat{\gamma}[\pi^s_{w_0}](b+s\delta
\tau)=\widehat{\gamma}^*(b)+s'(-1,\mathbf{0}).$$ Hence
$$\gamma^0[\pi^s_{w_0}](b+s\delta \tau
)<\gamma^{*^0}(b) \quad {\rm and}\quad \gamma[\pi^s_{w_0}](b+s\delta
\tau )=\gamma^*(b).$$ That is, the trajectory $\gamma[\pi^s_{w_0}]$
arrives at the same endpoint as $\gamma^*$ but with less cost. Then
$\widehat{\gamma}^*$ cannot be optimal as assumed. Thus
$(-1,\mathbf{0})_{\widehat{\gamma}^*(b)}$ is not interior to
$\widehat{K}^{\pm}_{b}$.

As $(-1,\mathbf{0})_{\widehat{\gamma}^*(b)}$ is not in the interior
of $\widehat{K}_{b}^{\pm}$, by Proposition \ref{propseparated} there
exists a covector $\widehat{\sigma}_{b} \in
T^*_{\widehat{\gamma}^*(b)} \widehat{Q}$ such that
\begin{eqnarray*}
\langle \widehat{\sigma}_{b},(-1,\mathbf{0}) \rangle &\geq& 0,\\
\langle \widehat{\sigma}_{b}, \widehat{v}_{b} \rangle &\leq& 0 \quad
\forall \; \widehat{v}_{b} \in \widehat{K}^{\pm}_{b}.
\end{eqnarray*}
The initial condition for the covector must not only satisfy the
previous inequalities, but also the transversality conditions. In
order to prove this, it is necessary to have the separability of two
new cones.

\textbf{(3e)} Hence, the initial condition in the fibers of
$T^*\widehat{Q}$ may be chosen satisfying the \textbf{transversality
conditions}. We consider the manifold with boundary given by
\[M_f=\{(x^0,x) \; | \; x\in S_f, \; x^0 \leq \gamma^{*^0}(b)\}.\] The
set of tangent vectors to $M_f$ at $\widehat{\gamma}^*(b)$ is the
convex set whose generators are
$(-1,\mathbf{0})_{\widehat{\gamma}^*(b)}$ and $T_f=\{0\}\times
T_{\gamma^*(b)}S_f$.

Given $\tau\in[a,b]$, consider the following closed convex sets
$$\begin{array}{ll}{\cal K}_{\tau}= \overline{{\rm
conv}(\widehat{K}^{\pm}_{\tau} \bigcup
(\Phi^{\widehat{X}^{\{u^*\}}}_{(\tau,a)})_*(T_a))}, &\quad {\rm
where} \ \ T_a=\{0\}\times T_{\gamma^*(a)}S_a, \\ &\\
{\cal J}_{\tau}= \overline{{\rm
conv}((-1,\mathbf{0})_{\widehat{\gamma}^*(\tau)} \bigcup
(\Phi^{\widehat{X}^{\{u^*\}}}_{(b,\tau)})^{-1}_*(T_f))}, &\quad {\rm
where} \ \ T_f=\{0\}\times T_{\gamma^*(b)}S_f,
\end{array}$$
and the manifold $M_{\tau}$ obtained transporting $M_f$ from $b$ to
$\tau$ using the flow of $\widehat{X}^{\{u^*\}}$. Observe that
${\cal J}_{\tau}$ is the closure of the set of tangent vectors to
$M_{\tau}$ at the point $\widehat{\gamma}^*(\tau)$. We are going to
show that the cones ${\cal K}_b$ and ${\cal J}_b$ are separated,
using Proposition \ref{initialtimelemma2}.

Observe that ${\cal J}_b$ is a half\textendash{}plane tangent to
$M_f$ and $\widehat{\gamma}^*(b)$ is on the boundary of $M_f$ by
construction. Hence, if ${\cal K}_b$ and ${\cal J}_b$ were not
separated, by Proposition \ref{initialtimelemma2} there would exist
a perturbation of the control $u[\pi^s_{w_0}]$ and $x_a\in S_a$ such
that the integral curve $\gamma_{x_a}[\pi^s_{w_0}]$ with initial
condition $(a,x_a)$ meets $M_f$ at a point in the relative interior
of $M_f$.
Hence we have found a trajectory with less cost than the optimal one
because of the definition of $M_f$ and this is not possible because
of the optimality of $\widehat{\gamma}^*$. Thus ${\cal K}_b$ and
${\cal J}_b$ are separated. So, by Proposition \ref{propseparated},
there exists a
 covector $\widehat{\sigma}_{b} \in
T^*_{\widehat{\gamma}^*(b)} \widehat{Q}$ such that
\begin{eqnarray}
\langle \widehat{\sigma}_{b}, \widehat{v}_{b} \rangle &\leq& 0 \quad
\forall \ \widehat{v}_{b} \in {\cal
K}_{b}, \label{ineq1trans}\\
\langle \widehat{\sigma}_{b}, \widehat{w}_{b} \rangle &\geq& 0 \quad
\forall \ \widehat{w}_{b} \in {\cal J}_{b}. \label{ineq2trans}
\end{eqnarray}
This covector separates the vector
$(-1,\mathbf{0})_{\widehat{\gamma}^*(b)}\in {\cal J}_b$ and the cone
$\widehat{K}^{\pm}_{b}\subset {\cal K}^{\pm}_b$. Let
$\widehat{\sigma}^*$ be the integral curve of
$(\widehat{X}^{T^*})^{\{u^*\}}$ with initial condition
$\widehat{\sigma}_{b}\in T_{\widehat{\gamma}^*(b)}\widehat{Q}$ at
$b$.

As $T_f$ is contained in ${\cal J}_{b}$, we have $\langle
\widehat{\sigma}_{b}, \widehat{v} \rangle \geq 0$ for every
$\widehat{v} \in T_f$. As $T_f$ is a vector space, if
$\widehat{v}\in T_f$, then $-\widehat{v}\in T_f$. Hence, we have
$$\langle \widehat{\sigma}_{b}, \widehat{v}
\rangle =0 \quad  {\rm for} \;\; {\rm every} \;\; \widehat{v}\in
T_f.$$ That is,
$$\langle \widehat{\sigma}_{b}, (0,v) \rangle =0 \quad  {\rm for} \;\; {\rm
every} \; \;
 v \in T_{\gamma^*(b)} S_f.$$
This is equivalent to $\langle \sigma_{b}, v\rangle=0$ for every $v
\in T_{\gamma^*(b)}S_f$; that is, $\sigma_{b}=\sigma^*(b)$ is in the
annihilator of $T_{\gamma^*(b)} S_f$ as wanted.

For every $\widehat{w}_b\in {\cal J}_b$, if $\widehat{W}\colon I
\rightarrow T\widehat{Q}$ is the integral curve of
$(\widehat{X}^T)^{\{u^*\}}$ with initial condition $\widehat{w}_b$
at time $b$, then by Proposition \ref{constant} the pairing
continuous natural function $\langle \widehat{\sigma}^*,\widehat{W}
\rangle \colon I \rightarrow \mathbb{R}$ is constant everywhere and
$\langle \widehat{\sigma}^*(a),\widehat{W}(a) \rangle \geq 0$ by
Equation (\ref{ineq2trans}). As
$(\Phi^{\widehat{X}^{\{u^*\}}}_{(b,a)})^{-1}_*({\cal J}_b)={\cal
J}_a$ by the continuity and the linearity of the flow, the
transversality condition at $a$ is proved
analogously as the transversality condition at $b$ proved above.\\

Since $(\widehat{\gamma}^*,u^*)$ is a solution of the
$\widehat{FOCP}$, it is also a solution of $\widehat{OCP}$ with time
and endpoints fixed and given by the curve. Hence, we can apply
Pontryagin's Maximum Principle for time and endpoints fixed, Theorem
\ref{PMP}. If the curve $(\widehat{\gamma}^*,u^*)$ is a solution of
$\widehat{OCP}$ with $I=[a,b]$ and endpoints $\widehat{\gamma}^*(a)$
and $\widehat{\gamma}^*(b)$, $(\widehat{\sigma}^*,u^*)\colon [a,b]
\rightarrow T^*\widehat{Q} \times U$ is a solution of the $HP$, such
that $ \widehat{\gamma}^*=\pi_{\widehat{Q}}\circ
\widehat{\sigma}^*$, and moreover $\widehat{\sigma}^*$ satisfies
that
\begin{itemize}
\item[(3a)] $H(\widehat{\sigma}^*(t),u^*(t))= \sup_{u \in U} H(\widehat{\sigma}^*(t),
u)$ almost everywhere.
\item[(3b)] $\sup_{u \in U} H(\widehat{\sigma}^*(t),
u)$ is constant everywhere.
\item[(3c)] $\widehat{\sigma}^*(t)\neq 0 \in
T^*_{\widehat{\gamma}^*(t)}\widehat{Q}$ for every $ t \in
[a,b]$.
\item[(3d)] $\sigma^*_0(t)$ is constant, $\sigma^*_0(t) \leq 0$.
\end{itemize}

Observe that it only remains to prove \textbf{(3b)} of the Free
Pontryagin's  Maximum Principle, since \textbf{(3a)}, \textbf{(3c)}
and \textbf{(3d)} are the
same in both Theorems \ref{PMP}, \ref{PMfreevariable}.\\

\textbf{(3b)} Due to (3a) we already know that the supremum of the
Hamiltonian is constant everywhere along $(\widehat{\sigma}^*,
u^*)$. Now, let us prove that the supremum can be taken to be zero
everywhere.

Take  $\widehat{v}_b={\pm}\widehat{X}(\widehat{\gamma}^*(b),
u^*(b))\in \widehat{K}^{\pm}_b$, let $\widehat{V}\colon I
\rightarrow T\widehat{Q}$ be the integral curve of
$(\widehat{X}^T)^{\{u^*\}}$ with initial condition
$(b,\widehat{\gamma}(b),\widehat{v}_b)$, then the continuous
function $\langle \widehat{\sigma}^*,\widehat{V} \rangle \colon I
\rightarrow \mathbb{R}$ is constant everywhere by Proposition
\ref{constant}. Thus,
$$\langle \widehat{\sigma}^*(t), \widehat{V}(t)\rangle=\langle \widehat{\sigma}^*(t),
{\pm}\widehat{X}(\widehat{\gamma}^*(t), u^*(t))\rangle \leq 0 \quad
{\rm for} \;\; {\rm every} \;\; t\in I$$ by Equation
(\ref{ineq1trans}), and this implies that
$$\langle
\widehat{\sigma}^*(t),\widehat{X}(\widehat{\gamma}^*(t),
u^*(t))\rangle=0.$$ As $\langle
\widehat{\sigma}^*(t),\widehat{X}(\widehat{\gamma}^*(t),
u^*(t))\rangle = H(\widehat{\sigma}^*(t), u^*(t))$, the Hamiltonian
function is zero everywhere and the supremum of the Hamiltonian
function is zero everywhere by Theorem \ref{PMP}.
\end{proof}

Observe that the initial condition for the covector in this proof
has been chosen such that the tangent spaces to the initial and
final submanifolds are contained in the separating hyperplane
defined by the covector. In this statement of the Maximum Principle
the initial condition for the covector must satisfy more conditions
than in Theorem \ref{PMP} (namely the transversality conditions).


\appendix

\section*{Appendices}

This last part of the report is mainly devoted to state and
prove some of the results used in the proof of the Maximum
Principle and to give more understanding to some key
points.

\section{Results on real functions}\label{analytic}

In this Appendix we focus on some necessary technicalities for the
proof of Pontryagin's Maximum Principle. These are related with
results from analysis and the notion of a Lebesgue point for a real
function. In this paper the notion of a Lebesgue point is applied to
vector fields. For more details about all this, see
\cite{68Royden,65Varberg,89Zaanen}.

\begin{defin} Let $(X,d_X)$ and $(Y,d_Y)$ be metric spaces. A function
$f \colon X \rightarrow Y$ is \textbf{Lipschitz} if there
exists $K\in \mathbb{R}$ such that $d_Y(f(x_1),f(x_2))\leq
K\,d_X(x_1,x_2)$ for all $x_1,x_2\in X$.

A function $f \colon X \rightarrow Y$ is \textbf{locally Lipschitz}
if, for every $x\in X$ there exists an open neighbourhood $V$ of $x$
and $K\in \mathbb{R}^+$ such that $ d_Y(f(x_1),f(x_2))\leq
K\,d_X(x_1,x_2)$ for all $x_1$ and $x_2$ in $V$.
\end{defin}
If $M$ is a differentiable manifold, $g$ is a Riemannian metric on
$M$ and $d_g\colon M \times M \rightarrow \mathbb{R}$ is the induced
distance; then $(M,d_g)$ is a metric space where the notion of
Lipschitz on $M$ can be defined. A real\textendash{}valued function
$F\colon M\rightarrow \mathbb{R}$ is locally Lipschitz if, for every
$p\in M$ we take the local chart $(V,\phi)$ such that $\phi(p)=0$,
$\phi(V)=B(0,r)$ is the open ball centered at the origin with radius
$r>0$ in the standard Euclidean space, and $F\circ \phi^{-1}\colon
B(0,r) \rightarrow \mathbb{R}$ is Lipschitz. That is, there exists
$K\in \mathbb{R}^+$ with
$$|F(p_1)-F(p_2)|=|(F\circ \phi^{-1})(\phi(p_1))-(F\circ
\phi^{-1})(\phi(p_2))| \leq Kd(\phi(p_1),\phi(p_2)), \quad \forall
p_1,p_2\in V.$$ Hence, given the local chart $(V,\phi)$, we define a
distance $d_{\phi}\colon V \times V \rightarrow \mathbb{R}$ on $V$,
$d_{\phi}(p_1,p_2)=d(\phi(p_1),\phi(p_2))$. Consequently, $(V,\phi)$
is a metric space with the topology induced by the open set $V$ in
$M$. This distance is equivalent to the distance induced by the
Riemannian metric on $V$.
Observe that the notion of locally Lipschitz for functions on
manifolds depends on the local chart, but ${\mathcal C}^1$ functions
are always locally Lipschitz.

\begin{defin} A function $f\colon [a,b] \rightarrow
\mathbb{R}$ is \textbf{uniformly continuous on $[a,b]$} if, for
every $\epsilon>0$, there exists $\delta >0$ such that for any $t,s
\in [a,b]$ with $|t-s|<\delta$, we have $|f(t)-f(s)|< \epsilon$.
\end{defin}

\begin{defin} A function $f\colon [a,b] \rightarrow
\mathbb{R}$ is \textbf{absolutely continuous on $[a,b]$}
if, for every $\epsilon
> 0$, there exists $\delta
> 0$ such that for every finite number of nonoverlapping
subintervals $(a_i, b_i)$ of $[a,b]$ with $\sum_{i=1}^n |b_i-a_i| <
\delta $, we have $\sum_{i=1}^n | f(b_i)-f(a_i)|< \epsilon$.
\end{defin}

We consider an interval $I=[a,b]$ in $\mathbb{R}$ with the usual
Lebesgue measure. A statement is said to be satisfied \textit{almost
everywhere} if it is fulfilled in $I$ except on a zero measure set.
A measurable subset $A\subset I$ is said of full measure if $I-A$
has measure zero. Recall that if $A, B \subset I$ and $I-A$, $I-B$
have measure zero, then $A \cap B$ is not empty.

Results in \cite{68Royden}, pp. 96, 100, 105 allow one to prove the
following result.
\begin{prop}\label{aederivative} If $f$ is absolutely continuous, then $f$ has
a derivative almost everywhere.
\end{prop}

\begin{teo}\label{constew}[\cite{68Royden}, pp.105 and \cite{65Varberg}, pp.836]
If $f$ is absolutely continuous and $f'(t)=0$ almost everywhere on
$[a,b]$, then $f$ is a constant function.
\end{teo}

\begin{defin} A real\textendash{}valued function $f$ on a metric
space
$(X,d)$ is called \textbf{lower semicontinuous at $x_0\in X$} if,
for every $\epsilon
>0$, there exists $\delta(\epsilon, x_0)>0$ such that
$f(x)\geq f(x_0)-\epsilon$ whenever $d(x,x_0)\leq
\delta(\epsilon,t_0)$.

If $f$ is lower semicontinuous at
every
point of $(X,d)$, it is said to be \textbf{lower
semicontinuous on $(X,d)$}.
\end{defin}

The following result is stated by Pontryagin et al. in \cite{P62},
page 102, but it is neither proved nor stated as a proposition. We
believe it is appropriate to write it with more detail because it is
used in \S \ref{proofPMPfixed}.

\begin{prop}\label{lowerew} Let $f$ and $g$ be real functions,
$f,g \colon [a,b] \rightarrow \mathbb{R}$. If $f$ is con\-ti\-nuous,
$g$ is lower semicontinuous, $f\leq g$ and $f=g$ almost everywhere
then $f=g$ everywhere.
\end{prop}

\begin{proof}
Let $t_0\in[a,b]$. As $g$ is lower semicontinuous on $[a,b]$, for
every $\epsilon>0$ there exists $\delta(\epsilon,t_0)=\delta>0$ such
that
$$g(t)\geq g(t_0)-\epsilon$$
whenever $|t-t_0|<\delta(\epsilon, t_0)$.

Since $f$ and $g$ coincide almost everywhere on $[a,b]$,
there exists $t_1 \in (t_0-\delta,t_0+\delta)$ such that
$f(t_1)=g(t_1)$. Moreover, $f\leq g$, so
\begin{equation}\label{fg}
f(t_0)\leq g(t_0)\leq g(t_1)+\epsilon =f(t_1)+\epsilon.
\end{equation}
The continuity of $f$ guarantees that for every $\epsilon'>0$, there
exists $\delta'>0$ such that if $|t_1-t_0|< \delta'$, then
$f(t_1)-\epsilon'<f(t_0)<f(t_1)+\epsilon'$. Hence Equation
(\ref{fg}) is rewritten as follows:
$$f(t_0)\leq g(t_0) \leq f(t_0)+\epsilon'+\epsilon.$$

As this inequality is valid for every $\epsilon,\epsilon'>0$,
$g(t_0)=f(t_0)$ for every $t_0\in [a,b]$. Thus $f=g$ everywhere.
\end{proof}


\subsection{Lebesgue points for a real
function}

After introducing the concept of measurable function and some
properties of such functions, we state Lebesgue's differentiation
theorem, which enables us to distinguish certain points for a
measurable function. In the entire paper we consider the Lebesgue
measure in $\mathbb{R}$. See the book by Zaanen \cite{89Zaanen} for
more details.

\begin{defin} A function $f \colon [a,b]\subset \mathbb{R} \rightarrow
\mathbb{R}$ is \textbf{measurable} if the set $\{t \in [a,b]\,
\colon \, f(t)> \alpha\}$ is measurable for every $\alpha \in
\mathbb{R}$.
\end{defin}


\begin{defin} A function $f\colon
[a,b]\rightarrow \mathbb{R}$ is \textbf{Lebesgue integrable} over
each Lebesgue measurable set of finite measure if $\nu(x) = \int_a^x
f d\mu$ is well defined for every $x\in[a,b]$.
\end{defin}

\begin{teo} \label{diffLebesgue}\textbf{(Lebesgue's Differentiation Theorem \cite{89Zaanen})}
Let $\mu$ be the Lebesgue measure. If $f\colon [a,b]\rightarrow
\mathbb{R}$ is a Lebesgue integrable function over every Lebesgue
measurable set of finite measure, then for $\nu(x) = \int_a^x f
d\mu$,
$$D \nu (x_+)=D \nu (x_-)= f(x)$$
holds for $\mu$\textendash{}almost every $x\in [a,b]$,
where $D \nu (x_+)$, $D \nu (x_-)$ are the right and left
derivatives of $\nu$ respectively.
\end{teo}

The equality $D \nu (x_-)= f(x)$ almost everywhere may be
rewritten as follows for $h>0$
$$\lim_{h\rightarrow
0}\frac{\nu(x-h)-\nu(x)}{-h}=f(x) \quad {\rm a.e.}
\Leftrightarrow \lim_{h \rightarrow 0} \frac{\int_a^{x-h}
f(t)dt -\int_a^xf(t)dt}{-h}=f(x) \quad {\rm a.e.}
\Leftrightarrow $$
\begin{equation*}\label{Lebesguetime} \Leftrightarrow \lim_{h \rightarrow 0}
\frac{\int^x_{x-h}f(t)dt}{h}=f(x) \quad {\rm a.e.}
\Leftrightarrow \int^x_{x-h}f(t)dt=hf(x)+ o(h) \quad {\rm
a.e.}
\end{equation*}
\begin{defin}
If $f\colon [a,b]\rightarrow \mathbb{R}$ is a measurable function,
$x\in(a,b)$ is a \textbf{Lebesgue point for $\mathbf{f}$} if,
$$\lim_{h\rightarrow 0} \int^x_{x-h} \frac{f(t)-f(x)}{h}\,dt=0.$$
\end{defin}

\begin{remark}
As Theorem \ref{diffLebesgue} is true almost everywhere,
the set of Lebesgue points for a measurable function has
full measure.
\end{remark}
\begin{remark}
Observe that if $u\colon I \rightarrow U$ is measurable and bounded,
then it is integrable and the set of Lebesgue points for $u$ has
full measure. If $f\colon U\rightarrow \mathbb{R}$ is continuous,
then $f\circ u \colon I \rightarrow \mathbb{R}$ is integrable, and
the intersection of Lebesgue points for $u$ and $f\circ u$ has full
measure.
\end{remark}
 \note Assume we have a manifold $Q$, an open set $U\subset \mathbb{R}^k$ and a continuous vector field $X$ along the
projection $\pi \colon Q \times U \rightarrow Q$. If
$(\gamma,u)\colon I=[a,b] \rightarrow Q\times U$, where $\gamma$ is
absolutely continuous and $u$ is measurable and bounded, then
$X\circ (\gamma,u)\colon I \rightarrow TQ$ is a measurable vector
field along $(\gamma,u)$, in the sense that in any coordinate system
its coordinate functions are measurable. A point $t\in(a,b)$ is a
\textit{Lebesgue point for $u$} if
\begin{equation}\label{eqLebesgue}
\int_{t -h}^t X(\gamma(s),u(s))ds=hX(\gamma(t),u(t))+ o(h).
\end{equation}
The Lebesgue points for a vector field are useful in \S
\ref{perturbedconesfixed}, \S \ref{perturbedconesnonfixed} and in
the following appendix to guarantee the differentiability of some
curves, that is, the existence of its tangent vector. See
\cite{2004Canizo,55Coddington} for more details about differential
equations and measurability.

\section{Time\textendash{}dependent variational equations}
\label{variationalsection}

The variational equations give us an approach to how the
integral curves of vector fields vary when the initial
condition varies along a curve. These equations have a
formulation on the tangent and the cotangent bundle. Here
we are interested in studying the variational equations
associated to time\textendash{}dependent vector fields, and
in proving some relationship between the solutions of
variational equations on the tangent bundle and the ones on
the cotangent bundle. See \cite{93Michor} for more details
about these concepts.

\subsection{Time\textendash{}dependent vector fields}\label{ApTimevf}
As seen in \S \ref{general}, control systems are associated to a
time\textendash{}dependent vector field through a vector field along
a projection. For $I\subset \mathbb{R}$, a \textit{differentiable
time\textendash{}dependent vector field $X$} is a mapping $X\colon I
\times M \rightarrow TM$ such that each $(t,x)\in I \times M$ is
assigned to a tangent vector $X(t,x)$ in $T_xM$. For every $(s,x)\in
I\times M$, \textit{the integral curve of $X$ with initial condition
$(s,x)$} is denoted by $\Phi^X_{(s,x)} \colon J_{(s,x)} \subset
I\rightarrow M$ and satisfies
\begin{enumerate}
\item $\Phi^X_{(s,x)}(s)=x$.
\item $\ds{\left.\frac{d}{dt}\right|_t \Phi^X_{(s,x)}=X(t,
\Phi^X_{(s,x)}(t))}$, $t\in J$.
\end{enumerate}
The domain of $\Phi^X_{(s,x)}$ is denoted by $J_{(s,x)}\subset I$
because it depends on the initial condition for the integral curves.

The \textit{time dependent flow} or \textit{evolution operator of
$X$} is the mapping
\begin{equation}\label{timeflow}\begin{array}{rcl}
\Phi^X \colon  & I\times I\times M &\longrightarrow M\\
&(t,s,x)&\longmapsto \Phi^X(t,s,x)=\Phi^X_{(s,x)}(t)
\end{array}\end{equation}
defined in a maximal open neighborhood $V\times M$ of $\Delta_I
\times M$, where $\Delta_I$ is the diagonal of $I\times I$, and
$\Phi^X$ satisfies
\begin{enumerate}
\item $\Phi^X(s,s,x)=x$.
\item $\ds{\left.\frac{d}{dt}\right|_t \ \left(
\Phi^X(t,s,x)\right)=X(t, \Phi^X(t,s,x))}$.
\end{enumerate}
To obtain the original vector field through the evolution operator,
the expression in the second assertion must be evaluated at $s=t$,
$$\left.\left[\left.\frac{d}{dt}\right|_t \ \left(
\Phi^X(t,s,x)\right)\right]\right|_{s=t}=X(t,x).$$ There is
a time\textendash{}independent vector field on the manifold
$I\times M$ associated to $X$ and given by
$\widetilde{X}(t,x)=\partial/\partial t_{\left| (t,x)
\right.}+X(t,x)$.  For $(t,s,x)\in V\times M$,  the flow of
$\widetilde{X}$ is $\Phi^{\widetilde{X}}\colon I\times
I\times M\rightarrow I\times M$ such that
$\Phi^{\widetilde{X}}_{(s,x)}$ is the integral curve of
$\widetilde{X}$ with initial condition $(s,x)$ at time $0$
and
$\Phi^{\widetilde{X}}(t,(s,x))=(s+t,\Phi^X(s+t,(s,x)))$.
The theorems in differential equations about the existence
and uniqueness of solutions guarantee the existence and
uniqueness of the evolution operator $\Phi^X$ defined
maximally.

For $(t,s)\in V \subset I \times I$,
\[
\begin{array}{rl}
\Phi^X_{(t,s)} \colon  M &\longrightarrow M\\
x&\longmapsto \Phi^X_{(t,s)}(x)=\Phi^X_{(s,x)}(t)
\end{array}
\]
is a diffeomorphism on $M$ satisfying
$\Phi^X_{(t,s)}=\Phi^X_{(t,r)}\circ \Phi^X_{(r,s)}$ for $r\in I$,
such that $(r,s)$, $(t,r) \in V$.


\subsection{Complete lift}\label{completelift}

Fromt the evolution operator of time\textendash{}dependent vector
fields in Equation $(\ref{timeflow})$, it is determined the
evolution operator of a particular vector field on $TM$.

Let $X_t\colon M \rightarrow TM$ be a vector field on $M$ such that
$X_t(x)=X(t,x)$ for every $t\in I$. The \textit{complete or tangent
lift of $X_t$ to $TM$} is the time\textendash{}dependent vector
field $X_t^T$ on $TM$ satisfying
$$X^T_t=\kappa_M \circ TX_t,$$
where $\kappa_M$ is the canonical involution of $TTM$; that is, a
mapping $\kappa_M\colon TTM \rightarrow TTM$ such that
$\kappa_M^2={\rm Id}$ and $\tau_{TM}\circ \kappa_M=T\tau_M$. See
\cite{93Michor} for more details in the definition. Moreover,
observe that $X_t$ is a vector field that makes the following
diagram commutative:
$$\bfig\xymatrix{TTM\ar[rr]^{\txt{\small{$\tau_{TM}$}}}
\ar@/_1pc/[dd]_{\txt{\small{$T\tau_M$}}} && TM
\ar@/_1pc/[dd]_{\txt{\small{$\tau_M$}}}
\\ && \\
TM\ar@/_1pc/[uu]_{\txt{\small{$TX_t$}}}
\ar[rr]^{\txt{\small{$\tau_M$}}}&& M
\ar@/_1pc/[uu]_{\txt{\small{$X_t$}}}}\efig$$ If $(x,v)\in TM$, then
$TX_t(x,v)=(x,X_t(x),T_xX_t(v)) \in T_{(x,X_t(x))}(TM)$.

Let $(W,x^i)$ be a local chart at $x$ in $M$ such that $X_t=X^i_t \,
\partial/\partial x^i$ where $X^i_t(x)=X^i(t,x)$ and $X^i\in {\mathcal C}^{\infty}(I\times W)$. If $(x^i,v^i)$
are the induced local coordinates in $TM$, then locally
$$X^T(t,x,v)= X^i(t,x) \left.\frac{\partial}{\partial x^i}\right|_{(x,v)} +
\frac{\partial X^i}{\partial x^j}(t,x)\,v^j
\left.\frac{\partial}{\partial v^i}\right|_{(x,v)}.$$

The equations satisfied by the integral curves of $X^T$ are
called \textit{variational equations}.

\begin{prop}\label{XT}
If $X$ is a time\textendash{}dependent vector field on $M$ and
$\Phi^X$ is the evolution operator of $X$, then the map $\Psi \colon
I \times I \times TM \rightarrow TM$ defined by
$$\Psi(t,s,(x,v))=\left(\Phi^X(t,s,x), T_x\Phi^X_{(t,s)}(v)\right)$$
is the evolution operator of $X^T$.
\end{prop}
\begin{proof}We have to prove that
$$ \left\{ \begin{array} {l} \Psi(s,s,(x,v))=(x,v); \\ \\ \ds{\left.\frac{d}{dt}\right|_t \ \left(
\Psi(t,s,(x,v)) \right)=X^T(t, \Psi(t,s,(x,v))).}
\end{array} \right. $$

The first item is proved easily,
$$\Psi(s,s,(x,v))=\left(\Phi^X(s,s,x), T_x\left(\Phi^X_{(s,s)}\right)(v)\right)=(x,v).$$

As for the second assertion, we use that $\Phi^X_{(t,s)}\colon M
\rightarrow M$ is a ${\cal C}^{\infty}$ diffeomorphism satisfying
\begin{eqnarray*}T_t\left(T_x\Phi^X_{(t,s)}(v)\right)
1&=&\left.\frac{d}{dt}\right|_t \left(T_x\Phi^X_{(t,s)}(v)\right) 1
= \left(
T_x\left(\left.\frac{d}{dt}\right|_t\left(\Phi^X_{(t,s)}\right) 1
\right)\right)(v)\\ &=&T_x\left(T_t\Phi^X_{(s,x)}
1\right)(v),\end{eqnarray*}  and we obtain
\begin{equation*}\begin{array}{l} \ds{\left.\frac{d}{dt}\right|_t\left(\Psi(t,s,x,v)\right)=\left(\left.\frac{d}{dt}\right|_t
\left(\Phi^X(t,s,x)\right), \left.\frac{d}{dt}\right|_t
\left(T_x\Phi^X_{(t,s)}(v)\right)\right)}\\ \\= \ds{\left(X(t,
\Phi^X(t,s,x)),\left(
T_x\left(\left.\frac{d}{dt}\right|_t\left(\Phi^X_{(t,s)}\right)\right)\right)(v)\right)}
\\ \\ = \left(X(t, \Phi^X(t,s,x)),
\left(T_x\left(X_t(\Phi^X(t,s,x))\right)\right)(v)\right)\\
\\= \left(X(t, \Phi^X(t,s,x)),
\left(T_{\Phi^X(t,s,x)}\left(X_t\right)\circ
T_x(\Phi^X(t,s,x))\right)(v)\right)\\ \\=\left(X(t, \Phi^X(t,s,x)),
T_{\Phi^X(t,s,x)}\left(X_t\right)\left(T_x(\Phi^X_{(t,s)})(v)\right)\right)=X^T(t,\Psi(t,s,x,v)).
\end{array}
\end{equation*}
Hence, the evolution operator of $X^T$ is the complete lift of the
evolution operator of $X$. The integral curves of $X^T$ are vector
fields along the integral curves of $X$.
\end{proof}

\subsubsection{About the geometric meaning of the complete
lift}\label{geometric meaning} The integral curves of $X^T$ must be
understood as the linear approximation of the integral curves of $X$
when the initial condition varies along a curve in $M$. This idea
appears in \S \ref{perturbedconesfixed} and \S \ref{proofPMPfixed}.

Let us explain the next figure. Given an integral curve of $X$ with
initial condition $(s,x)$, we consider a curve $\sigma$ starting at
the point $x$ of the integral curve. Every point of $\sigma$ can be
considered as the initial condition at time $s$ for an integral
curve of $X$. Thus the flow of $X$ transports the curve $\sigma$ at
a different curve $\delta_t$ point by point. The resultant curve is
related with the complete lift of $X$ as the following results
prove.
\begin{center}
\includegraphics*[bb=0 0 360 250,scale=0.6]{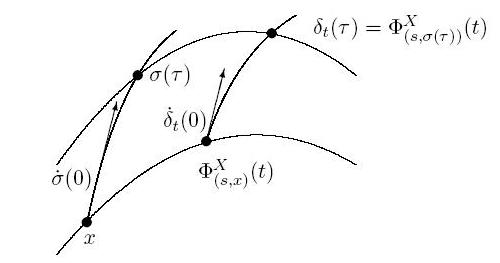}
\end{center}
\begin{prop}\label{propgeomeaning} Let $X\colon I\times M \rightarrow TM$ be a
time\textendash{}dependent vector field with evolution operator
$\Phi^X$ and $(s,x)\in I\times M$. For $\epsilon>0$, let $\sigma
\colon (-\epsilon,\epsilon) \subset \mathbb{R} \rightarrow M$ be a
${\mathcal C}^{\infty}$ curve such that $\sigma(0)=x\in M$. For
every $t\in I$, we define a curve $\delta_t \colon
(-\epsilon,\epsilon) \rightarrow M$ such that
\begin{enumerate} \item
$\delta_t(\tau)=\Phi^X_{(s,\sigma(\tau))}(t)$,
\item $\delta_s(\tau)=\sigma(\tau)$, and
\item $\delta_t(0)=\Phi^X_{(s,x)}(t)$.
\end{enumerate}
Then $\dot{\delta}_t(0)=T_x \Phi^X_{(t,s)}(\dot{\sigma}(0))$.
\end{prop}
\begin{proof}
\begin{eqnarray*}\dot{\delta}_t(0)&=&\left( T_0\delta_t(\tau)\right)
\left.\frac{d}{d\tau}\right|_0=\left(
T_0\left(\Phi^X_{(s,\sigma(\tau))}(t)\right)\right)
\left.\frac{d}{d\tau}\right|_0= \left(
T_0\left(\Phi^X_{(t,s)}(\sigma(\tau))\right)\right)
\left.\frac{d}{d\tau}\right|_0\\&& \\
&=&T_{\sigma(0)}\Phi^X_{(t,s)}\left(T_0(\sigma(\tau))
\left.\frac{d}{d\tau}\right|_0\right)=T_{\sigma(0)}\Phi^X_{(t,s)}\left(\dot{\sigma}(0)\right)=
T_x\Phi^X_{(t,s)}\left(\dot{\sigma}(0)\right).
\end{eqnarray*}
\end{proof}

\begin{corol} Let $X$ be a time\textendash{}dependent vector field on $M$. For $x\in M$, $x\in T_xM$ and
for a small enough $\epsilon>0$, let $\sigma \colon
(-\epsilon,\epsilon) \subset \mathbb{R} \rightarrow M$ be a
${\mathcal C}^{\infty}$ curve such that $\sigma(0)=x\in M$ and
$\dot{\sigma}(0)=v$. If $\delta_t$ is the curve defined in
Proposition \ref{propgeomeaning}, then $
\dot{\delta}_{(\cdot)}(\tau)\colon I \rightarrow TM$, $t\mapsto
\dot{\delta}_t(\tau)$ is the integral curve of $X^T$ with initial
condition $(s,\dot{\sigma}(\tau))$.
\end{corol}

\begin{proof} The proof just comes from Propositions \ref{XT} and
\ref{propgeomeaning} and the definition of the curve $\delta_t$.
\end{proof}

\subsection{Cotangent lift}\label{colift} Given $(t,s)\in I\times I$, the
evolution operator $\Phi^{X^T}_{(t,s)}$ is a diffeomorphism
on $TM$ and a linear isomorphism on the fibers on $TM$, so
it makes sense to consider its transpose and inverse,
$(\;^{\tau}\Phi^{X^T}_{(t,s)})^{-1}=\Lambda_{(t,s)}$. It is
a linear isomorphism on the fibers on $T^*M$ and satisfies
$\Lambda_{(t,s)}=\Lambda_{(t,r)}\circ \Lambda_{(r,s)}$ for
$r\in I$. Hence $\Lambda\colon I\times I \times T^*M
\rightarrow T^*M$ is the evolution operator of a
time\textendash{}dependent vector field on $T^*M$, called
the \textit{cotangent lift $X^{T^*}$ of $X$ to $T^*M$}.

 Another vector field on $T^*M$ may be associated to $X$
using the concepts in Hamiltonian formalism. For every $t\in I$, the
Hamiltonian system $\ds{\left(M,\omega, \widehat{X_t}\right)}$ given
by a symplectic manifold $(M,\omega)$ and the Hamiltonian function
$\widehat{X_t}\colon T^*M \rightarrow \mathbb{R}$,
$\widehat{X_t}(p_x)=i_{X(t,x)}p_x$ where $p_x\in T^*_xM$; has
associated a Hamiltonian vector field $Z_t$ that satisfies
Hamilton's equations $i_{Z_t}\omega= d \widehat{X_t}$.

In local coordinates $(x,p)$ for $T^*M$, $Z\colon I\times T^*M
\rightarrow TT^*M$ is given by
$$Z(t,x,p)= X^i(t,x) \left.\frac{\partial}{\partial x^i}\right|_{(x,p)} -
\frac{\partial X^j}{\partial x^i}(t,x)\,p_j
\left.\frac{\partial}{\partial p_i}\right|_{(x,p)}.$$

The equations satisfied by the integral curves of $Z$ in
the fibers are the \textit{adjoint variational equations on
the cotangent bundle}. In the literature, they are
sometimes called \textit{adjoint equations}.

Let us prove that both vector fields $Z$ and $X^{T^*}$
associated to $X$ are the same.
\begin{prop}\label{cotangent}
If $X$ is a time\textendash{}dependent vector field on $M$
and $\Phi^X$ is the evolution operator of $X$, then
$\Lambda \colon I \times I \times T^*M \rightarrow T^*M$
such that
$$\Lambda(t,s,(x,p))=(\Phi^X(t,s,x), \left(\;^{\tau}T_x \Phi^X_{(t,s)}\right)^{-1}(p))$$
is the evolution operator of $Z$. Thus $Z=X^{T^*}$.
\end{prop}
\begin{proof}
We have to prove that
$$ \left\{ \begin{array} {l} \Lambda(s,s,(x,p))=(x,p),
 \\ \\\ds{\left.\frac{d}{dt}\right|_t \left(
\Lambda(t,s,(x,p))\right)=Z(t, \Lambda(t,s,(x,p))).}
\end{array} \right. $$

The first item is proved easily,
$$\Lambda(s,s,(x,p))=\left(\Phi^X(s,s,x), \left(\;^{\tau}T_x \Phi^X_{(s,s)}\right)^{-1}(p)\right)
=(x,{\rm Id}\ (p))=(x,p).$$ As $\Phi^X_s\colon I \times M
\rightarrow M$ is ${\cal C}^{\infty}$, in local coordinates we have
$$\left.\frac{d}{dt}\right|_t \left(\;^{\tau}T_x
\Phi^X_{(t,s)}\right) =\;^{\tau}T_x\left(\left.\frac{d}{dt}\right|_t
\Phi^X_{(t,s)}\right),$$ where both mappings go from $T_x^*M$ to
$T^*_{\Phi^X(t,(s,x))}M$. Now let us prove the second assertion:
\begin{equation*}
\begin{array}{l}
\ds{\left.\frac{d}{dt}\right|_t\left(
\Lambda(t,s,x,p)\right)=\left(\left.\frac{d}{dt}\right|_t
\left(\Phi^X(t,s,x)\right), \left.\frac{d}{dt}\right|_t
\left(\left(\;^{\tau}T_x \Phi^X_{(t,s)}\right)^{-1}(p)
\right)\right)}\\
= \left(X(t, \Phi^X(t,s,x)),\left(-\left(\;^{\tau}T_x
\Phi^X_{(t,s)}\right)^{-1}\circ \left(\left.\frac{d}{dt}\right|_t
\left(\;^{\tau}T_x \Phi^X_{(t,s)}\right)\right)
\circ\left(\;^{\tau}T_x
\Phi^X_{(t,s)}\right)^{-1}\right)(p)\right)\\
=\left(X(t,\Phi^X(t,s,x)), \left(-\left(\;^{\tau}T_x
\Phi^X_{(t,s)}\right)^{-1}\circ
\left(\;^{\tau}T_x\left(\left.\frac{d}{dt}\right|_t
\Phi^X_{(t,s)}\right)\right)\right)\left(\left(\;^{\tau}T_x
\Phi^X_{(t,s)}\right)^{-1}(p)\right)\right)\\
=\left(X(t,
\Phi^X(t,s,x)),-\left(\;^{\tau}T_{\Phi^X_{(t,s)}(x)}\left(X_t\circ
\Phi^X_{(t,s)}\circ \left(\Phi^X_{(t,s)}\right)^{-1}
\right)\right)\left(\left(\;^{\tau}T_x
\Phi^X_{(t,s)}\right)^{-1}(p)\right)\right)\\
=\left(X(t,
\Phi^X(t,s,x)),-\left(\;^{\tau}T_{\Phi^X_{(t,s)}(x)}\left(X_t\right)\right)\left(\left(\;^{\tau}T_x
\Phi^X_{(t,s)}\right)^{-1}(p)\right)\right)=Z(t,\Lambda(t,s,x,p)).\end{array}
\end{equation*}
Hence, the evolution operator of $Z$ is the cotangent lift of the
evolution operator of $X$. Thus $Z=X^{T^*}$.
\end{proof}

\subsection{A property for the complete and cotangent lift}

The previous propositions allow us to determine an invariant
function along integral curves of $X$.

\begin{prop}\label{constant}
Let $X\colon I \times M \rightarrow TM$ be a
time\textendash{}dependent vector field and let $X^T\colon I \times
TM \rightarrow TTM$ and $X^{T^*}\colon I \times T^*M \rightarrow
TT^*M$ be the complete lift and cotangent lift of $X$, respectively.
If $\gamma \colon I\rightarrow M$ is an integral curve of $X$ with
initial condition $(s,x)$, $V\colon I \rightarrow TM$ is the
integral curve of $X^T$ with initial condition $(s,(x,v))$ where
$v\in T_{\gamma(s)}M$, and $\Lambda\colon I \rightarrow T^*M$ is the
integral curve of $X^{T^*}$ with initial condition $(s,(x,p))$ where
$p \in T^*_{\gamma(s)}M$, then
$$\begin{array}{rcl} \langle \Lambda, V \rangle \colon I&
\rightarrow &\mathbb{R} \\
t &\mapsto & \langle \Lambda(t), V(t) \rangle
\end{array}$$
is constant along $\gamma$.
\end{prop}

\begin{proof} If $\Phi^X$ is the evolution operator of $X$, the
evolution operators of $X^T$ and $X^{T^*}$ are
$$\Phi^{X^T}(t,s,(x,v))=\left(\Phi^X(t,s,x), T_x\Phi^X_{(t,s)} (v)\right),$$
$$\Phi^{X^{T^*}}(t,s,(x,p))=\left(\Phi^X(t,s,x), \left(\;^{\tau}T_x \Phi^X_{(t,s)}\right)^{-1}(p)\right),$$
respectively, because of Propositions \ref{XT} and \ref{cotangent}.
Hence
\begin{eqnarray*}\langle \Lambda(t), V(t) \rangle &=&
\left<\left(\;^{\tau}T_x \Phi^X_{(t,s)}\right)^{-1}(p),
T_x\Phi^X_{(t,s)} (v)\right>\\ &=& \left<\;^{\tau}\left(\left(T_x
\Phi^X_{(t,s)}\right)^{-1}\right)(p), T_x\Phi^X_{(t,s)} (v)\right>\\
& =& \left< p, \left(\left(T_x \Phi^X_{(t,s)}\right)^{-1}\circ
\left(T_x\Phi^X_{(t,s)}\right)\right) (v) \right>=\langle p, v
\rangle = {\rm constant}.
\end{eqnarray*}
\end{proof}

\section{The tangent perturbation cone as an approximation of the reachable
set}\label{approxreachable}

In control systems, the reachable sets are useful to determine the
accessibility and the controllability of the systems. In optimal
control, the reachable set has a great importance for distinguishing
the abnormal optimal curves from the normal ones
\cite{2005BulloLewis,2003Bavo}. A key point in the proof of the
Maximum Principle depends on the understanding of that linear
approximation of the reachable set in a neighborhood of a point in
the optimal curve. This interpretation of the tangent perturbation
cone has been studied in \cite{2004Agrachev}, but we will study it
in a great and clear detail in this appendix.

In the sequel, we explain why this interpretation of the
tangent perturbation cone is feasible. Remember from \S
\ref{ApTimevf} that a time\textendash{}dependent vector
field on $M$ has associated the evolution operator
$\Phi^X\colon I\times I \times M\rightarrow M$, $(t,s,x)
\mapsto \Phi^X(t,s,x)$ as defined in Equation
(\ref{timeflow}).

\begin{prop}\label{sumflows} Let $X$, $Y$ be time\textendash{}dependent vector fields on $M$,
then there exists a time\textendash{}dependent vector field $Z$ such
that
\[\Phi^{X+Y}_{(t,s)}(x)=(\Phi^X_{(t,s)}\circ
\Phi_{(t,s)}^{Z})(x)\] and
$Z=(\Phi^X_{(t,s)*})^{-1}Y=(\Phi^X_{(t,s)})^*Y$.
\end{prop}
\begin{proof} 
For any initial time $s$, we define the diffeomorphism
$\tilde{\Phi}^X_s\colon I \times M \rightarrow I \times M$,
$(t,x)\rightarrow (t,\Phi^X_s(t,x))$ such that
$\tilde{\Phi}^X_s(s,x)=(s,x)$. We look for a
time\textendash{}dependent vector field $Z$ on $M$ such that
\begin{equation}\label{multflow}\tilde{\Phi}^{X+Y}_s(t,x)=(\tilde{\Phi}^X_s
\circ\tilde{\Phi}^Z_s)(t,x) \; .
\end{equation}
This expression has been assumed true in
\cite{2004Agrachev,2000AndrewBulloCDC} for $s=0$, but it has not
been carefully proved. On the left\textendash{}hand side of Equation
(\ref{multflow}) we have
\[\tilde{\Phi}^{X+Y}_s(t,x)=(t,\Phi^{X+Y}_s(t,x)) \]
and the right\textendash{}hand side is
\[(\tilde{\Phi}^X_s
\circ\tilde{\Phi}^Z_s)(t,x)=\tilde{\Phi}^X_s(t,\Phi^Z_s(t,x))=(t,\Phi^X_s(t,\Phi^Z_s(t,x)))
\; .
\]
Thus Equation (\ref{multflow}) is satisfied if and only if
\begin{eqnarray}\Phi^{X+Y}_s(t,x)=\Phi^X_s(t,\Phi^Z_s(t,x))&=&(\Phi^X_s\circ
\tilde{\Phi}^Z_s)(t,x)\ , \label{multflow2} \end{eqnarray} or
equivalently,
\begin{eqnarray}
\Phi^{X+Y}_{(t,s)}&=&\Phi^X_{(t,s)}\circ \Phi^Z_{(t,s)} \ .
\label{multflow3}
\end{eqnarray}
Let us differentiate with respect to $t$ the left\textendash{}hand
side of Equation (\ref{multflow2}),
\begin{equation}\label{dleft}\ds{\frac{d}{dt}\Phi^{X+Y}_{(s,x)}(t)= (X+Y)(t,\Phi^{X+Y}_{(s,x)}(t))=
(X+Y)(t,\Phi^X_s(t,\Phi^Z_s(t,x)))} \; .
\end{equation}

The differentiation with respect to time of the
right\textendash{}hand side of Equation (\ref{multflow2}), for $f$
in ${\cal C}^{\infty}(M)$, is
\[\begin{array}{l} \ds{\frac{d}{d t}
(\Phi^X_s(t,\Phi^Z_s(t,x))) f=\lim_{h\rightarrow 0}
\frac{f((\Phi^X_s(t+h,
\Phi^Z_s(t+h,x))))-f((\Phi^X_s(t,\Phi^Z_s(t,x))))}{h}}\\
\\
\ds{=\lim_{h\rightarrow 0} \left\{\frac{(f\circ \Phi^X_{(t+h,s)})
(\Phi^Z_{(t+h,s)}(x))-(f\circ
\Phi^X_{(t+h,s)})(\Phi^Z_{(t,s)}(x))}{h}\right.}\\
\\\ds{+\left.\frac{(f\circ \Phi^X_s)(t+h,\Phi^Z_s(t,x))-(f\circ
\Phi^X_s)(t,\Phi^Z_s(t,x))}{h}\right\}}\\
\\\ds{ =Z(t,\Phi^Z_s(t,x))(f\circ \Phi^X_{(t,s)})+X(t,\Phi^X_s(t,\Phi^Z_s(t,x))) f}\\ \\\ds{=
{\rm T}_{\Phi^Z_s(t,x)} \Phi^X_{(t,s)} Z(t,\Phi^Z_s(t,x))f
+X(t,\Phi^X_s(t,\Phi^Z_s(t,x))) f}\; .
\end{array} \]
Hence
\[ \ds{\frac{d}{d t}
(\Phi^X_s(t,\Phi^Z_s(t,x)))={\rm T}_{\Phi^Z_s(t,x)} \Phi^X_{(t,s)}
Z(t,\Phi^Z_s(t,x))+X(t,\Phi^X_s(t,\Phi^Z_s(t,x)))} \; . \] From
Equation (\ref{dleft}) we have
\begin{eqnarray*}X(t,\Phi^X_s(t,\Phi^Z_s(t,x)))+Y(t,\Phi^X_s(t,\Phi^Z_s(t,x)))&=&{\rm
T}_{\Phi^Z_s(t,x)} \Phi^X_{(t,s)} Z(t,\Phi^Z_s(t,x))\\
&+&X(t,\Phi^X_s(t,\Phi^Z_s(t,x))) \; ,\end{eqnarray*} that is,
\[ Y((\tilde{\Phi}^X_s\circ
\tilde{\Phi}^Z_s)(t,x))={\rm T}_{\Phi^Z_s(t,x)} \Phi^X_{(t,s)}
Z(t,\Phi^Z_s(t,x))\; . \] Remember that the pushforward of a
time\textendash{}dependent vector field $Z$ is another
time\textendash{}dependent vector field given by
\[ (\Phi^X_{(t,s)*}Z)(t,x)={\rm T}_{(\Phi^X_{(t,s)})^{-1}(x)}\Phi^X_{(t,s)}(Z(t,(\Phi^X_{(t,s)})^{-1}(x))) \; .\]
Then
\[\begin{array}{l} (Y\circ \tilde{\Phi}^X_s\circ
\tilde{\Phi}^Z_s)
(t,x)=(\Phi^X_{(t,s)*}Z)(t,\Phi^X_{(t,s)}(\Phi^Z_s(t,x)))\\ \\
=(\Phi^X_{(t,s)*}Z)(\tilde{\Phi}^X_s(t,\Phi^Z_s(t,x)))=(\Phi^X_{(t,s)*}Z)(\tilde{\Phi}^X_s\circ
\tilde{\Phi}^Z_s)(t,x)\; ,\end{array}
\]
or equivalently,
\[Y\circ \tilde{\Phi}^X_s\circ
\tilde{\Phi}^Z_s=(\Phi^X_{(t,s)*}Z)\circ \tilde{\Phi}^X_s\circ
\tilde{\Phi}^Z_s \, , \] that is, $Y=\Phi^X_{(t,s)*}Z$.

Hence $Z=(\Phi^X_{(t,s)*})^{-1}Y=(\Phi^X_{(t,s)})^*Y$.
Now, going back to Equation (\ref{multflow3}) we have
\begin{equation}\label{flow}
\begin{array}{l}\Phi^{X+Y}_{(t,s)}(x)=(\Phi^X_{(t,s)}\circ
\Phi^{(\Phi^X_{(t,s)})^*Y}_{(t,s)})(x)\; .\end{array}
\end{equation}
\end{proof}

\begin{defin}\label{ReachSet} Let $M$ be a manifold, $U$ be a set in
$\mathbb{R}^k$ and $X$ be a vector field along the projection
$\pi\colon M\times U \rightarrow M$. The \textbf{reachable set
 from $x_0\in M$ at time $T\in I$} is the set
of points described by \begin{eqnarray*} {\mathcal
R}(x_0,T)=\left\{x\in M  \right. & | & \textrm{ there exists }
(\gamma,u)\colon [a,b] \rightarrow M\times U \textrm{ such that } \\
& & \left. \dot{\gamma}(t)=X(\gamma(t),u(t)), \; \gamma(a)=x_0, \;
\gamma(T)=x \right\}.\end{eqnarray*}
\end{defin}

Once we know how to express the flow of a sum of vector fields as a
composition of flows of different vector fields, we are going to
show that all the integral curves used to construct the reachable
set in Definition \ref{ReachSet} can be written as composition of
flows associated with vector fields given by vectors in the tangent
perturbation cone in Definition \ref{tangent}.

Each control system $X\in \mathfrak{X}(\pi)$ with the projection
$\pi\colon M\times U \rightarrow M$ is a time\textendash{}dependent
vector field $X^{\{u\}}$ when the control is given. Consider the
reference trajectory $(\gamma,u)$ to be an integral curve of
$X^{\{u\}}$ with initial condition $x_0$ at $a$. Take $\gamma(t_1)$
to be a reachable point from $x_0$ at time $t_1$. Let us consider
another control $\tilde{u}\colon I \rightarrow U$ and the integral
curve of $X^{\{\tilde{u}\}}$ with initial condition $x_0$ at $a$
denoted by $\tilde{\gamma}$. Then $\tilde{\gamma}(t_1)$ is another
reachable point from $x_0$ at time $t_1$.

 Let us see how to reach the
point $\tilde{\gamma}(t_1)$ using Equation (\ref{flow}),
\begin{equation}
\begin{array}{rcl}\ds{\tilde{\gamma}(t_1)}&=&\ds{\Phi^{X^{\{\tilde{u}\}}}_{(t_1,a)}(x_0)=
\Phi^{X^{\{u\}}+\left(X^{\{\tilde{u}\}}-X^{\{u\}}\right)}_{(t_1,a)}(x_0)}\\
&=&\ds{\left(\Phi^{X^{\{u\}}}_{(t_1,a)}\circ
\Phi^{\left(\Phi^{X^{\{u\}}}_{(t_1,a)}\right)^*{\left(X^{\{\tilde{u}\}}-X^{\{u\}}\right)}}_{(t_1,a)}\right)(x_0)}
\\&=& \ds{\left(\Phi^{X^{\{u\}}}_{(t_1,a)}\circ
\Phi^{\left(\Phi^{X^{\{u\}}}_{(t_1,a)}\right)^*{\left(X^{\{\tilde{u}\}}-X^{\{u\}}\right)}}_{(t_1,a)}\circ
\left(\Phi^{X^{\{u\}}}_{(t_1,a)}\right)^{-1}\circ \Phi^{X^{\{u\}}}_{(t_1,a)}\right)(x_0)}\\
&=&\ds{\left(\Phi^{X^{\{u\}}}_{(t_1,a)}\circ
\Phi^{\left(\Phi^{X^{\{u\}}}_{(t_1,a)}\right)^*{\left(X^{\{\tilde{u}\}}-X^{\{u\}}\right)}}_{(t_1,a)}\circ
\left(\Phi^{X^{\{u\}}}_{(t_1,a)}\right)^{-1}\right)(\gamma(t_1)) \;
. \label{flowt1} }
\end{array}\end{equation}
Hence, from $\gamma(t_1)$ we can get every reachable point from
$x_0$ at time $t_1$ through Equation (\ref{flowt1}) composing
integral curves of the vector fields $X^{\{u\}}$ and
$(\Phi^{X^{\{u\}}}_{(t_1,a)})^*{(X^{\{\tilde{u}\}}-X^{\{u\}})}\colon
I \times M \rightarrow TM$, this latter with initial condition $x_0$
at $a$.

In fact this is true for any time $\tau$ in $[a,t_1]$, that is,
\[\ds{\tilde{\gamma}(\tau)=\left(\Phi^{X^{\{u\}}}_{(\tau,a)}\circ
\Phi^{\left(\Phi^{X^{\{u\}}}_{(\tau,a)}\right)^*{(X^{\{\tilde{u}\}}-X^{\{u\}})}}_{(\tau,a)}\circ
\left(\Phi^{X^{\{u\}}}_{(\tau,a)}\right)^{-1}\right)(\gamma(\tau))
\; .}
\]
If we compose with the flow of $X^{\{u\}}$, we get a reachable point
from $x_0$ at time $t_1$ because it is a concatenation of integral
curves of the dynamical system,
\begin{equation}\label{flowtau}\begin{array}{rcl}\ds{\Phi^{X^{\{u\}}}_{(t_1,\tau)}
(\tilde{\gamma}(\tau))}&=&\ds{\left(\Phi^{X^{\{u\}}}_{(t_1,\tau)}\circ
\Phi^{X^{\{u\}}}_{(\tau,a)}\circ
\Phi^{\left(\Phi^{X^{\{u\}}}_{(\tau,a)}\right)^*{\left(X^{\{\tilde{u}\}}-X^{\{u\}}\right)}}_{(\tau,a)}\circ
\left(\Phi^{X^{\{u\}}}_{(\tau,a)}\right)^{-1}\right)(\gamma(\tau))}
\\ &=&\ds{\left(\Phi^{X^{\{u\}}}_{(t_1,a)}\circ
\Phi^{\left(\Phi^{X^{\{u\}}}_{(\tau,a)}\right)^*{\left(X^{\{\tilde{u}\}}-X^{\{u\}}\right)}}_{(\tau,a)}\circ
\left(\Phi^{X^{\{u\}}}_{(t_1,a)}\right)^{-1}\right)(\gamma(t_1))\; .
}
\end{array}
\end{equation}
Hence, from $\gamma(t_1)$ we can also get reachable points from
$x_0$ at time $t_1$ through composition of integral curves of the
vector fields $X^{\{u\}}$ and
$\left(\Phi^{X^{\{u\}}}_{(\tau,a)}\right)^*{\left(X^{\{\tilde{u}\}}-X^{\{u\}}\right)}$,
the latter with initial condition $\gamma(a)$ at time $a$.

On the other hand, the tangent perturbation cone at $\gamma(t_1)$ is
given by the closure of the convex hull of all the tangent vectors
$\left(\Phi^{X^{\{u\}}}_{(t_1,\tau)}\right)_*\left(X^{\{\tilde{u}\}}(\tau,\gamma(\tau))-X^{\{u\}}(\tau,
\gamma(\tau))\right)$ for every Lebesgue time $\tau$ in $[a,t_1]$.
These vectors are related with the vector fields $X^{\{u\}}$ through
Equations (\ref{flowt1}) and (\ref{flowtau}).

In this sense, we say that the tangent perturbation cone at
$\gamma(t_1)$ is an approximation of the reachable set in a
neighborhood of $\gamma(t_1)$.

\section{Convex sets, cones and hyperplanes}\label{hyperplane}

We study some properties satisfied by convex sets and cones; see
\cite{2001Bertsekas,98Rockafellar} for details. Unless otherwise
stated, we suppose that all the sets are in a
$n$\textendash{}dimensional vector space $E$. We need to define the
different kinds of cones and linear combinations used in this
report.

\begin{defin} A \textbf{cone $C$ with vertex at $0 \in
E$} satisfies that if $v\in C$, then $\lambda \, v \in C$ for every
$\lambda \geq 0$.
\end{defin}

\begin{defin}\label{conicconvex} Given a family of vectors $V\subset E$.
\begin{enumerate}
\item A \textbf{conic non\textendash{}negative
combination} of elements in $V$ is a vector of the form $\lambda_1
v_1 + \cdots + \lambda_r v_r$, with $\lambda_i \geq 0$ and $v_i\in
V$ for all $i \in \{1, \ldots , r\}$.
\item The \textbf{convex cone} generated by
$V$ is the set of all conic non\textendash{}negative
combinations of vectors in $V$.
\item An \textbf{affine combination} of elements in $V$ is a vector of the
form $\lambda_1 v_1 + \cdots + \lambda_r v_r$, with $v_i\in V$,
$\lambda_i \in \mathbb{R}$ for all $i \in \{1, \ldots , r\}$ and
$\sum_{i=1}^{r} \lambda_i = 1$.
\item A \textbf{convex combination} of elements in $V$ is a vector of the
form $\lambda_1 v_1 + \cdots + \lambda_r v_r$, with $v_i\in V$, $0
\leq \lambda_i \leq 1$ for all $i \in \{1, \ldots , r\}$ and
$\sum_{i=1}^{r} \lambda_i = 1$.
\end{enumerate}
\end{defin}
Remember that a set $A\subset E$ is \textit{convex} if, given two
different elements in $A$, then any convex combination of them is
contained in $A$. Thus, all the convex combination of elements in
$A$ are in $A$.
\begin{defin}
The \textbf{convex hull} of a set $A\subset E$, ${\rm
conv}(A)$, is the smallest convex subset containing $A$.
\end{defin}

Let us prove a characterization of the convex hull that will be
useful.
\begin{prop}\label{convexhull} The convex hull of a set $A$ is the set of the convex combinations of elements in $A$.
\end{prop}
\begin{proof}
Let us denote by $C$ the set of all convex combinations of
elements in $A$. First, we prove that $C$ is a convex set.
If $x$, $y$ are in $C$, then they are convex combinations
of elements in $A$; that is, $x=\sum^l_{i=1}\lambda_iv_i$,
$y=\sum^r_{i=1}\mu_iw_i$, with $\sum^l_{i=1}\lambda_i=1$,
$\sum^r_{i=1}\mu_i=1$. For $s\in (0,1)$, consider
$$sx+(1-s)y=s
\left(\sum^l_{i=1}\lambda_iv_i\right)+(1-s)\left(\sum^r_{i=1}\mu_iw_i
\right),$$ that will be in $C$ if the sum of the
coefficients is equal to $1$ and each of the coefficients
lies in $[0,1]$. Observe that
$s\sum^l_{i=1}\lambda_i+(1-s)\sum^r_{i=1}\mu_i=s+(1-s)=1$
and the other condition is satisfied trivially. As $C$ is
convex and contains $A$, the convex hull of $A$ is a subset
of $C$.

Second, we prove that $C\subset {\rm conv}(A)$ by induction on the
number of vectors in the convex combinations of elements in $A$.
Trivially, when the convex combination is given by an element in
$A$, it lies in the convex hull of $A$.

Now, suppose that a convex combination of $l-1$ elements of $A$ is
in ${\rm conv}(A)$, and we prove that a convex combination of $l$
elements of $A$ is in ${\rm conv}(A)$. Let
\[x=\sum_{i=1}^l\mu_iv_i=\sum_{i=1}^{l-1}\mu_iv_i+\mu_l v_l.\] If
$\sum_{i=1}^{l-1}\mu_i=0$, then $\mu_l=1$. By the first step of the
induction, $x$ is in ${\rm conv}(A)$.
If $\sum_{i=1}^{l-1}\mu_i\in(0,1]$, then $\mu_l\in [0,1)$
and we can rewrite $x$ as
$$x=(1-\mu_l)\left(\sum_{i=1}^{l-1}\mu_i(1-\mu_l)^{-1}v_i\right)+\mu_l v_l.$$
Observe that
$\sum_{i=1}^{l-1}\mu_i(1-\mu_l)^{-1}=(1-\mu_l)(1-\mu_l)^{-1}=1$, and
so $\sum_{i=1}^{l-1}\mu_i(1-\mu_l)^{-1}v_i$ is in ${\rm conv}(A)$.
By the first step of induction, $v_l$ is in ${\rm conv}(A)$. As
$(1-\mu_l)+\mu_l=1$, $x$ is in ${\rm conv}(A)$ because of the
convexity of ${\rm conv}(A)$. Thus $C\subset {\rm conv}(A)$ and so
$C={\rm conv}(A)$.
\end{proof}
\begin{prop} \label{convex}
Let $C$ be a convex set. If $\overline{C}$ and ${\rm int \
} C$ are the topological closure and the interior of $C$,
respectively, we have:
\begin{itemize}
\item[(a)]for every $x\in {\rm int \ } C$, if $y\in \overline{C}$, then
$(1-\lambda)x+\lambda y \in  {\rm int \ } C$ for all $\lambda\in
[0,1)$;
\item[(b)] $\overline{C}=\overline{ {\rm int \ }C}$;
\item[(c)] the interior of $C$ is empty if and only if the interior
of $\overline{C}$ is empty;
\item[(d)] $ {\rm int \ } C=  {\rm int \ } {\overline{C}}$.
\end{itemize}
\end{prop}
\begin{proof}
\textbf{(a)} If $x\in  {\rm int \ }\ C$, then there exists
$\epsilon_x>0$ such that $B(x,\epsilon_x)\subset C$, where
$B(x,\epsilon_x)$ denotes the open ball centered at $x$ of radius
$\epsilon_x$.

Observe that if $y\in \overline{C}$, for any $\epsilon>0$, $y\in
C+\epsilon B(0,1)=\{x+\epsilon z \, | \, x\in C, \, z\in B(0,1)\}$.

For every $\lambda\in [0,1)$, we consider
$x_{\lambda}=(1-\lambda)x+\lambda y$. Let us compute the value of
$\epsilon_{\lambda}$ such that $x_{\lambda} + \epsilon_{\lambda}
B(0,1)\subset C$.
$$x_{\lambda} + \epsilon_{\lambda} B(0,1)=(1-\lambda) x + \lambda y + \epsilon_{\lambda}
B(0,1)$$ $$\subset (1-\lambda)x+\lambda C +\lambda \epsilon
B(0,1)+\epsilon_{\lambda} B(0,1) =(1-\lambda) x+(\lambda \epsilon
+\epsilon_{\lambda} )B(0,1)+\lambda C.$$ If
$\epsilon_{\lambda}=(1-\lambda)\epsilon_x-\lambda \epsilon$, then
$(1-\lambda) x+(\lambda \epsilon +\epsilon_{\lambda} )B(0,1)\subset
(1-\lambda)C$ and $x_{\lambda} + \epsilon_{\lambda} B(0,1)\subset
C$. For $\epsilon>0$ small enough, $\epsilon_{\lambda}$ is positive.
Here we use the sum operation of convex sets, which is
well\textendash{}defined if the coefficients are positive (if $C_1$
and $C_2$ are convex sets, $\mu_1 C_1+\mu_2 C_2$ is a convex
set for all $\mu_1,\mu_2\geq 0$).\\

\textbf{(b)} As $ {\rm int \ }C\subset C$, $\overline{ {\rm int \
}C} \subset \overline{C}$.

On the other hand, each point in the closure of $C$ can be
approached along a line segment by points in the interior of $C$ by
$(a)$. Thus $\overline{C}\subset \overline{ {\rm
int \ }C}$.\\

\textbf{(c)} As $ {\rm int \ }C \subset  {\rm int \
}\overline{C}$, if $ {\rm int \ }\overline{C}$ is empty, then $
{\rm int \ }C$ is empty.

Conversely, if $ {\rm int \ }C$ is empty, then by $(b)$
$\overline{C}$ is empty. So $C$ is empty  and $ {\rm int \ }C$ is
also
empty.\\

\textbf{(d)} Trivially $ {\rm int \ }C \subset  {\rm int \
}\overline{C}$.

As the equality of the sets is true when they are empty
because of $(c)$, let us suppose that $ {\rm int \ }C$ is
not empty. If $z\in  {\rm int \ }\overline{C}$ and take
$x\in  {\rm int \ }C$, then there exists a small enough
positive number $\delta$ such that $y=z+\delta(z-x)\in {\rm
int \ }\overline{C} \subset \overline{C}$.

 Hence,
$$z=\frac{1}{1+\delta}y +\frac{\delta}{1+\delta}x.$$ Note
that $$ 0<\frac{1}{1+\delta}<1, \quad
0<\frac{\delta}{1+\delta}<1, \quad
\frac{1}{1+\delta}+\frac{\delta}{1+\delta}=1.
$$
As $y\in \overline{C}$, $x\in {\rm int \ }C$ and $1/(1+\delta)$ lies
in $(0,1)$. By $(a)$, $z\in {\rm int \ }C$.
\end{proof}
\begin{remark}
Consequently, if $C$ is convex and dense, then $C$ is the
whole space.
\end{remark}
The following paragraphs introduce elements playing an
important role in the proof of Pontryagin's Maximum
Principle in \S \ref{proofPMPfixed} and \S
\ref{SproofFPMP}.
\begin{defin} Let $C$ be a cone with
vertex at $0\in E$. A \textbf{supporting hyperplane to $C$ at $0$}
is a hyperplane such that $C$ is contained in one of the
half\textendash{}spaces defined by the hyperplane.
\end{defin}

\begin{remark} In a geometric framework, we will define a
hyperplane in $E$ as the kernel of a nonzero
$1$\textendash{}form $\alpha$ in the dual space $E^*$. Then
the hyperplane $P_{\alpha}$ associated to $\alpha$ is $\ker
\alpha$. Hence the supporting hyperplane to $C$ at $0$ is a
hyperplane $P_{\alpha}$ such that $\alpha (v) \leq 0$ for
all $v \in C$. A supporting hyperplane to $C$ at $0$ is not
necessarily unique.
\end{remark}

From now on, we consider that all the cones have vertex at
$0$.

\begin{defin} Let $C$ be a cone,
the \textbf{polar of $C$} is \[C^*= \{ \alpha \in E^* \; | \;
\alpha(v)\leq 0\ , \; \forall \; v \in C \}.
\]\end{defin}

Note that the polar of a cone is a closed and convex cone
in $E^*$.

\begin{defin}
Let $C$ be a cone, the set \[C^{\ast \ast}=\{ w \in E \; | \;
\alpha(w)\leq 0 \ , \; \forall \, \alpha \in C^* \}\] is called the
\textbf{polar of the polar of $C$}.
\end{defin}

Observe that $C\subset C^{\ast \ast}$. The following lemma
is used in the proof of the existence of a supporting
hyperplane to a cone with vertex at $0$.

\begin{lemma} \label{C**} The cone $C$ is closed and convex if
and only if $C^{\ast \ast}=C$.
\end{lemma}
\begin{proof}
Observe that $$C^{**}=\{w\in E \; | \; \alpha(w)\leq 0 \ , \;
\forall \, \alpha \in C^*\}=\bigcap_{\alpha \in C^*} \{w\in E \; |
\; \alpha(w)\leq 0 \}.$$ Then $C^{**}=\overline{{\rm conv}(C)}$,
because of Theorem 6.20 in Rockafellar \cite{98Rockafellar}: the
closure of the convex hull of a set is the intersection of all the
closed half\textendash{}spaces containing the set. Now, the result
is immediate.
%
%
\end{proof}

The following proposition guarantees the existence of a
supporting hyperplane to a cone with vertex at $0$. This
result is used throughout the proof of Pontryagin's Maximum
Principle in \S \ref{proofPMPfixed} and \S
\ref{SproofFPMP}.

\begin{prop} \label{existsupporting} If $C$ is a convex and closed
cone that is not the whole space, then there exists a
supporting hyperplane to $C$ at $0$.
\end{prop}

\begin{proof}
If there is no supporting hyperplane containing the cone in one of
the two half\textendash{}spaces, then for all $\alpha\in E^*$ there
exist $v_1, \, v_2 \in C$ with $\alpha(v_1)\leq 0$ and
$\alpha(v_2)\geq 0$. Thus $C^*=\{0\}$ and $C^{\ast \ast} =E$. Then,
by Lemma \ref{C**}, $C=C^{\ast \ast}=E$ in contradiction with the
hypothesis on $C$.
\end{proof}

\begin{corol}\label{convexsupporting}
If $C$ is a convex cone that is not the whole space, then
there exists a supporting hyperplane to $C$ at $0$.
\end{corol}
\begin{proof}
If $C\neq E$, then $\overline{C}\neq E$ by Proposition \ref{convex}
(d). Hence, by Proposition \ref{existsupporting}, there exists a
supporting hyperplane to $\overline{C}$ which is also a supporting
hyperplane to $C$.
\end{proof}

\begin{defin}
Let $C_1$ and $C_2$ be cones with common vertex $0$. They are
\textbf{separated} if there exists a hyperplane $P$ such that each
cone lies in a different closed half\textendash{}space defined by
$P$. This $P$ is called a \textbf{separating hyperplane of
$\mathbf{C_1}$ and $\mathbf{C_2}$}.
\end{defin}

A point $x$ is a \textit{relative interior point of a set $C$}, if
$x\in C$ and there exists a neighbourhood $V$ of $x$ such that
$V\cap {\rm aff}(V)\subseteq V$. Then, a useful characterization of
separated convex cones is the following:
\begin{prop} \label{propseparated} The convex cones $C_1$ and $C_2$,
with common vertex $0$, are separated if and only if one of
the two following conditions are satisfied:
\begin{itemize}
\item[(1)] there exists a hyperplane containing both $C_1$
and $C_2$,
\item[(2)] there is no point that is a relative interior
point of both $C_1$ and $C_2$.
\end{itemize}
\end{prop}

\begin{proof}
$\Rightarrow$ If $C_1$ and $C_2$ are separated, then there exists a
separating hyperplane $P_{\alpha}$ such that
\begin{equation*}
\alpha(v_1)\leq 0 \quad \forall \, v_1\in C_1, \;
\alpha(v_2)\geq 0 \quad \forall \, v_2\in C_2.
\end{equation*}
If $\alpha(v_i)=0$ for all $v_i\in C_i$ and $i=1,2$, then
we are in the first case.

If some $v_i \in C_i$ satisfies the strict inequality, then both
sets do not lie in the hyperplane $P_{\alpha}$. They lie in a
different closed half\textendash{}space. If the convex cones
intersect, the intersection lies in the boundary of the cones and in
the hyperplane. Hence, there is no point that is a relative interior
point of both $C_1$ and $C_2$.

$\Leftarrow$ First, we are going to prove that if $(1)$ is true,
then $C_1$ and $C_2$ are separated. As there exists a hyperplane
determined by $\alpha$ such that $\alpha(v_i)=0$ for all $v_i \in
C_i$,  $\alpha$ determines a separating hyperplane of $C_1$ and
$C_2$.

Now, we are going to prove that if $(2)$ is true, then $C_1$ and
$C_2$ are separated. As $C_1$ and $C_2$ are convex cones,
\[C_1-C_2=\{u\in E \; | \; u=v_1-v_2, \; v_1\in C_1, \; v_2\in
C_2\}\] is a convex cone. Since there is no relative interior point
of both $C_1$ and $C_2$, $0$ does not lie in $C_1-C_2$. By Corollary
\ref{convexsupporting} there exists a supporting hyperplane
$P_{\alpha}$ to $C_1-C_2$ such that $\alpha(v_1-v_2)\leq 0$, that
is, $\alpha(v_1)\leq \alpha(v_2)$, for all $v_1 \in C_1$, $v_2\in
C_2$.

Observe that a supporting hyperplane to $C_1-C_2$ is a supporting
hyperplane to $C_1$, because, taking $v_2=0$,
$\alpha(v_1)\leq\alpha(v_2)= 0$ for all $v_1\in C_1$.

As $\partial(C_1-C_2)\cap C_1 \subset \partial C_1$, we
consider a supporting hyperplane $P_{\alpha}$ to $C_1-C_2$
such that $\alpha(v_1)=0$ for some $v_1\in \partial C_1$.
Hence $\alpha(v_2)\geq\alpha(v_1)=0$ for all $v_2 \in C_2$.
As $\alpha(v_1)\leq 0$ for all $v_1\in C_1$, $\alpha$
determines a separating hyperplane of $C_1$ and $C_2$.
\end{proof}

This proposition gives us necessary and sufficient
conditions for the existence of a separating hyperplane of
two convex cones with common vertex. Observe that a
separating hyperplane of two cones with common vertex is
also a supporting hyperplane to each cone at the vertex.

\begin{corol}\label{splitSussmann}
If the convex cones $C_1$ and $C_2$ with common vertex $0$
are not separated, then $E=C_1-C_2$.
\end{corol}
\begin{proof}
If the cones are not separated, by Proposition \ref{propseparated}
there exists no any hyperplane containing both and the intersection
of their relative interior is not empty.

Let us suppose that the convex cone $C_1-C_2\neq E$. Then, by
Corollary \ref{convexsupporting}, there exists a supporting
hyperplane determined by $\lambda$ at the vertex such that
$\lambda(v)\geq 0$ for every $v$ in $C_1-C_2$.

Because of the definition of cones, if $v_1\in C_1$, then
$v_1\in C_1-C_2$ and $\lambda(v_1)\geq 0$. Analogously, if
$v_2\in C_2$, then $-v_2\in C_1-C_2$ and $\lambda(-v_2)\geq
0$, that is, $\lambda(v_2)\leq 0$.
\end{proof}

\section{One corollary of Brouwer Fixed\textendash{}Point
Theorem}\label{Brouwer}

From the statement of Brouwer Fixed\textendash{}point
Theorem, it is possible to prove a corollary in
\cite{67LeeMarkus} useful for the proof of Proposition
\ref{lemma2}.

\begin{teo}\label{fixedpoint} \textbf{(Brouwer Fixed\textendash{}point
Theorem)} Let $B_1^n$ be the closed unit ball in
$\mathbb{R}^n$. Any continuous function $G \colon B_1^{n}
\rightarrow B_1^n$ has a fixed point.
\end{teo}

\begin{corol} \label{scholium} Let $g\colon B^n_1\rightarrow \mathbb{R}^n$ be a
continuous map. Let $P$ be an interior point of $B^n_1$. If
$\|g(x)-x\|<\|x-P\|$ for every $x$ in the boundary
$\partial B^n_1$, then the image $g(B^n_1)$ covers $P$.
\end{corol}

\begin{proof}
Without loss of generality, we assume that $P$ is the origin of
$\mathbb{R}^n$. Consider the mapping $g$ as a continuous vector
field on the unit ball $B_1^n$.

As $\|g(x)-x\|<\|x\|$, we are going to show that $g(x)$
makes an acute angle with the outward ray from the origin
through $x$ for every $x \in
\partial B^n$. Let us consider the equality
$$ \|y-z\|^2+\|z-x\|^2=\|y-x\|^2+2\langle y-z,x-z \rangle,$$
and take $y=g(x)$ and $z=0$. Then
$$2\langle g(x),x \rangle=\|g(x)\|^2+\|x\|^2-\|g(x)-x\|^2 >
\|g(x)\|^2+\|x\|^2-\|x\|^2=\|g(x)\|^2\geq 0.$$ Thus
$g(x)$ makes an acute angle with $x$. So $g(x)$ has an
outward radial component at every point $x \in
\partial B_1^n$. The vector $-g(x)$ has a
negative radial component.
For a sufficiently small positive number $\alpha$ the
function $x \rightarrow x-\alpha \, g(x)$ goes from $B_1^n$
to $B_1^n$. By Theorem \ref{fixedpoint} there exists a
fixed point $x_0$ such that $x_0=x_0-\alpha \, g(x_0)$,
then $\alpha \, g(x_0)=0$ and $g(x_0)=0$ since $\alpha \in
\mathbb{R}^+$. As $g$ is continuous and $g(x_0)=0$, the
image of a neighbourhood of $x_0$ covers the origin.
\end{proof}

\section*{Acknowledgements}
We acknowledge the financial support of \emph{Ministerio de
Educaci\'on y Ciencia}, Project
MTM\-2005\textendash{}\-04947 and the Network Project
MTM2006\textendash{}27467\textendash{}E/. MBL also
acknowledges the financial support of the FPU grant
AP20040096.

We are grateful to Professor Andrew D. Lewis for
encouraging us to work on a deep understanding of this
subject. We thank Marina Delgado for collaborating with us
on this report during her postdoctoral stay in our
Department in 2005.

\end{document}